\documentclass[preprint]{imsart}
\usepackage{amsthm,amsmath,amssymb,hyperref}

\doi{}
\pubyear{0000}
\volume{0}
\firstpage{0}
\lastpage{0}

\def\MR#1{\href{http://www.ams.org/mathscinet-getitem?mr=#1}{MR#1}}
\def\begbib{\begin{itemize}}
\def\endbib{\end{itemize}}
\newcommand{\comment}[1]{}

\newcommand{\notes}{\paragraph{Notes and Comments}}
\newcommand{\oned}{one dimensional}
\newcommand{\te}{\to}
\newcommand{\ocdc}{op\'erateur carr\'e du champ}
\newcommand{\sqf}{square field operator}
\newcommand{\delm}{\Delta}
\newcommand{\gen}{{\cal G}}
\newcommand{\fint}{\int}
\newcommand{\tell}{{\tau_\ell}}
\newcommand{\QX}{ X}
\newcommand{\ball}{D}
\newcommand{\coloneqq}{:=}

\newcommand{\lb}[1]{\label{#1}}
\newtheorem{Theorem}{Theorem}

\newtheorem{claim}[Theorem]{Claim}

\theoremstyle{definition}

\newcommand {\hf} { \mbox{$ {1 \over 2 }$} }
\newcommand{\eval}{\left|\begin{array}{c}\\\\\end{array}\right.\!\!\!\!\!\!\!}
\newcommand{\BB}{{\cal B}}
\newcommand{\brk}[1]{\langle #1 \rangle }
\newcommand{\FF}{{\cal F}}
\newcommand{\GG}{{\cal G}}
\newcommand{\ZZ}{{\cal Z}}
\newcommand{\eq}{\begin{equation}}
\newcommand{\en}{\end{equation}}
\newcommand{\re}[1]{\mbox{(\ref{#1})}}

\newcommand {\PR} {\mathbb{P}}
\newcommand {\QR} {\mathbb{Q}}
\newcommand {\SR} {\mathbb{S}}
\newcommand {\ER} {\mathbb{E}}
\newcommand {\reals} {\mathbb{R}}
\newcommand {\matrices} {\mathbb{M}}
\newcommand {\complex} {\mathbb{C}}
\newcommand{\convd}{\mbox{$\ \stackrel{{\rm d} }{\rightarrow} \ $}}
\newcommand{\convas}{\mbox{$\ \stackrel{{\rm a.s.}}{\rightarrow} \ $}}
\newcommand{\ed}{\mbox{$\ \stackrel{d}{=} \ $}}
\newcommand{\Lev}{L\'evy}
\newcommand{\FD}{Feller-Dynkin}

\newcommand{\Erdos}{Erd\H{o}s}

\newcommand{\Ito}{It\^o}

\newcommand{\giv}{{\,|\,}}

\newcommand{\realsp}{\reals_{>0}}

\newcommand{\len}[1] {\lambda_{#1}}
\newcommand{\eps}{\varepsilon}
\newcommand{\EE}{{\cal E}}

\renewcommand{\Bbb} { B^\br}

\newcommand{\Bex} { B^\ex}
\newcommand{\Bme} { B^\me}

\newcommand{\br}{ {\mbox{$\scriptstyle{\rm br}$}}}

\newcommand{\ex}{ {\mbox{$\scriptstyle{\rm ex}$}}}
\newcommand{\me}{ {\mbox{$\scriptstyle{\rm me}$}}}
\newcommand{\underl}[1]{\underline{#1}}

\newcommand{\bb}{B^{\rm br}}

\newcommand{\bex}{B^{\rm \ex}}

\newcommand{\bexm}{B^{ {\rm \ex} \giv m}}
\newcommand{\bext}{B^{{\rm \ex},t}}
\newcommand{\bmet}{B^{{\rm \me},t}}
\newcommand{\bbrt}{B^{{\rm \br},t}}

\newcommand{\bme}{B^{\rm \me}}

\newcommand{\length}[1]{ \lambda _{#1} } 
\newcommand{\xs}{X_*}
\newcommand{\xomega}{ X }

\begin{document}

\begin{frontmatter}

\title{A guide to Brownian motion and related stochastic processes}

\runtitle{Guide to Brownian motion}


\begin{aug}
  \author{Jim Pitman \ead[label=e1]{pitman@berkeley.edu}} 
  \and
  \author{Marc Yor} 

  \address{Dept. Statistics, University of California, \\ 367 Evans Hall \# 3860, Berkeley, CA 94720-3860, USA \\
           \printead{e1}}

\end{aug}

\runauthor{J. Pitman and M. Yor}

\begin{abstract}
This is a guide to the mathematical theory of Brownian
motion and related stochastic processes, with indications of how this theory
is related to other branches of mathematics, most 
notably the classical theory of partial differential equations associated with the Laplace and heat 
operators, and various generalizations thereof.
As a typical reader, we have in mind a student, familiar with 
the basic concepts of probability based on measure theory, at the level of 
the graduate texts of Billingsley  \cite{bill95} and Durrett \cite{MR2722836}, 
and who wants a broader perspective on the theory of Brownian motion and related stochastic processes than can 
be found in these texts. 
\end{abstract}

\begin{keyword}
\kwd{Markov process}
\kwd{random walk}
\kwd{martingale}
\kwd{Gaussian process}
\kwd{L\'evy process}
\kwd{diffusion}
\end{keyword}

\begin{keyword}[class=AMS]
\kwd[Primary ]{60J65}
\end{keyword}

\end{frontmatter}
\tableofcontents
\section{Introduction}

This is a guide to the mathematical theory of Brownian
motion (BM) and related stochastic processes, with indications of how this theory
is related to other branches of mathematics, most 
notably the classical theory of partial differential equations associated with the Laplace and heat 
operators, and various generalizations thereof.

As a typical reader, we have in mind a student familiar with 
the basic concepts of probability based on measure theory, at the level of 
the graduate texts of Billingsley  \cite{bill95} and 
Durrett \cite{MR2722836}, and who wants a broader perspective on the
theory of BM and related stochastic processes than can be found in these texts. 
The difficulty facing such a student is that there are now too many advanced texts on BM and 
related processes. Depending on what aspects or applications are of interest, one can choose 
from any of the following texts, each of which contains excellent treatments of many facets of
the theory, but none of which can be regarded as a definitive or complete 
treatment of the subject.
\paragraph{General texts on BM and related processes}
\begbib
\item[\cite{fr83bd}]D.~Freedman. \newblock {\em Brownian motion and diffusion} (1983).    
\item[\cite{im65}]K. It{\^o} and H.~P. McKean, Jr. \newblock {\em Diffusion processes and their sample paths} (1965).
\item[\cite{MR86f:60049}]K.~It{\^o}. \newblock {\em Lectures on stochastic processes} (1984).
\item[\cite{kall02f}]O.~Kallenberg. \newblock {\em Foundations of modern probability} (2002).
\item[\cite{MR89c:60096}]I. Karatzas and S.~E. Shreve. \newblock {\em Brownian motion and stochastic calculus} (1988).
\item[\cite{MR3156223}] G. Kallianpur and S. P. Gopinath. {\em Stochastic analysis and diffusion processes} (2014).
\item[\cite{kni81}]F.~B. Knight. \newblock {\em Essentials of {B}rownian motion and diffusion} (1981). 
\item[\cite{MR2604525}] P. M\"orters and Y. Peres. {\em Brownian motion} (2010).
\item[\cite{ry99}]D.~Revuz and M.~Yor. \newblock {\em Continuous martingales and {B}rownian motion} (1999). 
\item[\cite{rw1}] and \cite{rw2} L.~C.~G. Rogers and D.~Williams. \newblock {\em Diffusions, Markov Processes and Martingales, Vols. I and II} (1994).
\endbib
This list does not include more specialized research monographs on 
subjects closely related to BM such as stochastic analysis, stochastic differential geometry, and more general theory of Gaussian and Markov processes.
Lists of such monographs classified by subject can be found in following sections.

\subsection{History}
The physical phenomenon of Brownian motion was discovered by Robert Brown,  a
19th century scientist who observed through a microscope the random 
swarming motion of pollen grains in water, now understood to be due to 
molecular bombardment.  The theory of Brownian motion was developed by 
Bachelier in his 1900 PhD Thesis \cite{bach1900}, 
and independently by Einstein in his 1905 paper \cite{einstein1905motion}
which used Brownian motion to estimate Avogadro's number and
the size of molecules.  The modern mathematical treatment of Brownian motion 
(abbreviated to BM), also called the {\em Wiener process} is due to 
Wiener in 1923 \cite{wiener23}.
Wiener proved that there exists a version of BM with continuous paths.  
L\'evy made major contributions to the theory
of Brownian paths, especially regarding the structure of their level sets, 
their occupation densities, and other fine features of their oscillations such
as laws of the iterated logarithm.
Note that BM is a Gaussian process, a Markov process, and a martingale. Hence its importance 
in the theory of stochastic process.  It serves as a basic building block 
for many more complicated processes.
For further history of Brownian motion and related processes we cite 
Meyer \cite{MR2002g:60107},
Kahane \cite{kahane97}, 
\cite{MR1640258} 
and Yor \cite{MR2001k:60116}.  

\subsection{Definitions}

This section records the basic definition of a Brownian motion $B$, along with 
some common variations in terminology which we use for some purposes.
The basic definition of $B$, as a random continuous function
with a particular family of finite-dimensional distributions,
is motivated in Section 
\ref{sec:rw} 
by the appearance of this process as a limit in
distribution of rescaled random walk paths.

Let $( \Omega , \mathcal{F} , \PR)$  be a probability space.   A
stochastic process
$( B(t, \omega )  , t  \geq 0 ,\omega \in \Omega )$
is a 
\emph{Brownian motion} if
\begin{enumerate}
\item [(i)]
For fixed each $t$, the random variable $B_t = B( t, \cdot   ) $ has Gaussian distribution with mean $0$ and variance $t$.
\item [(ii)]
The process $B$ has \emph{stationary independent increments}.
\item [(iii)]
For each fixed $ \omega \in \Omega$, the path $t \rightarrow B(t,   \omega$ ) is continuous.
\end{enumerate}
The meaning of (ii) is that if $0 \leq  t_1 < 
t_2 <  ...< t_n$, then $ B_{t_1}  , B_{t_2}  -  B_{t_1} ,...,B_{t_n}  -  B_{t_{n-1}} $ are independent, and the distribution of $B_{t_i} - B_{t_{i-1}}$ 
depends only on $t_i - t_{i-1}$. According to (i), this distribution is
normal with mean $0$ and variance $t_i - t_{i-1}$.
We say simply that $B$ is {\em continuous} to indicate that
$B$ has continuous paths as in (iii).
Because of the convolution properties of normal distributions, the joint distribution
of $ X_{t_1} , ...,X_{t_n} $ are consistent for any $ t_1 < 
...<  t_n $.  By Kolmogorov's consistency theorem 
\cite[Sec. 36]{bill95}, 
given such a consistent family of finite dimensional distributions, there 
exists a process ( $ X_t, t  \geq  0 $ ) satisfying (i) and (ii),
but the existence of such a process with continuous paths is not obvious.
Many proofs of existence of such a process can be found in the literature.
For the derivation from Kolmogorov's criterion for sample path
continuity, see  \cite[\S, Theorem (1.8)]{ry99}.
Freedman \cite{fr83bd} offers a more elementary approach via the 
following steps:
\begin{itemize}
\item
Step 1:  Construct $X_t $ for $ t \in D \ =  \{  \mbox{dyadic   rationals} \} 
=  \{\frac{k} {2^n }\} $.
\item
Step 2:  Show for almost all $ \omega $,  $t  \rightarrow  X( t ,  \omega  )$ is uniformly
continuous on $D \cap  [ 0,  T ]$ for  any finite T.
\item
Step 3:  For such $ \omega $, extend the definition of $X(t, \omega)$ to $t \in [0, \infty)$  by continuity.
\item
Step 4: Check that (i) and (ii) still hold for the process so defined.
\end{itemize}

Except where otherwise specified, a Brownian motion $B$ is assumed
to be one-dimensional, and to start at $B_0 = 0$, as in the above
definition.
If $\beta_t = x + B_t$ for some $x \in \reals$ then $\beta$ is
a {\em Brownian motion started at $x$}.  
Given a Brownian motion $(B_t, t\geq 0)$ starting from 0, the process
$X_t:=x+\delta t+\sigma B_t$ is called a {\em Brownian motion started at 
$x$ with drift parameter $\delta$ and variance parameter $\sigma^2$.}
The notation $\PR^x$ for probability or $\ER^x$ for expectation may be used to 
indicate that $B$ is a Brownian motion started at $x$ rather than $0$, with
$\delta = 0$ and $\sigma^2 = 1$.
A \emph{$d$-dimensional Brownian motion} is a process
$$
( B_t:= (B_t^{(1)}, \ldots, B_t^{(d) }); t \ge 0 )
$$
where the processes $B^{(i)}$ are $d$ independent one-dimensional
Brownian motions. 

\section{BM as a limit of random walks}
\lb{sec:rw}
Let $S_n:= X_1 + \ldots + X_n$ where the $X_i$ are independent random
variables with mean $0$ and variance $1$, and let $S_t$ for real $t$
be defined by linear interpolation between integer values.
Let $B$ be a standard one-dimensional BM.
It follows easily from the  central limit theorem 
\cite[Th. 27.1]{bill95}  that
\eq
\lb{btconv}
(S_{nt}/\sqrt{n}, t \ge 0 ) \convd  (B_t , t \ge 0 ) 
\mbox{ as } n \te \infty
\en
in the sense of weak convergence of finite-dimensional distributions.
According to \emph{Donsker's theorem} \cite{bill99,MR2722836,ry99}, 
this convergence holds also in the path space $C[0,\infty)$ equipped
with the topology of uniform convergence on compact intervals. That is
to say, for every $T >0$, and every functional 
$f$ of a continuous path $x = (x_s, 0 \le s \le T)$  that is bounded and
continuous with respect to the supremum norm on $C[0,T]$,
\eq
\lb{btconv1}
\ER [ f(S_{nt} /\sqrt{n} , 0 \le t \le T ) ] \to  \ER [ f(B_t , 0 \le t \le T )]
\mbox{   as } n \te \infty .
\en
Here, for ease of notation, we suppose that the random walk $X_1, X_2, \ldots$
and the limiting Brownian motion $B$ are defined on the same probability
space $( \Omega , \mathcal{F} , \PR)$, and $\ER$ denotes expectation
with respect to $\PR$.  
One way to do this, which can be
used to prove \re{btconv}, is to use the \emph{Skorokhod embedding} 
technique of constructing the $X_i= B_{T_i} - B_{T_{i-1}}$ for a
suitable increasing sequence of stopping times  $0 \le T_1 \le T_2 \cdots$
such that the ${T_i} - {T_{i-1}}$ are independent copies of $T_1$ with
$\ER (T_1) = \ER(X_1^2)$. 
See \cite[Th. 8.6.1]{MR2722836},  
\cite{bill99} or \cite{ry99} for details, and
\cite{MR2068476} for a survey of variations of this construction.
To illustrate a consequence 
of \re{btconv},
\eq
\lb{btconv2}
\frac{1 }{ \sqrt{n} } \max_{1 \le k \le n } S_{k}  \convd  \sup_{0 \le t \le 1 } B_t \ed | B_t |
\mbox{   as } n \te \infty
\en
where the equality in distribution is due to the well known 
\emph{reflection principle} for Brownian motion.
This can be derived
via Donsker's theorem from the corresponding principle for a simple random 
walk with $X_i = \pm 1$ discussed in \cite{fe1}, or proved directly in
continuous time \cite{bill95}\cite{MR2722836}\cite{fr83bd}\cite{ry99}.
See also \cite{MR0139211}, 
\cite{MR84d:60050} 
for various more refined senses in which Brownian motion may be approximated by random walks.

Many generalizations and variations of Donsker's theorem are known \cite{bill99}.
The assumption of independent and identically distributed $X_i$ can be weakened
in many ways: with suitable auxilliary hypotheses, the $X_i$ 
can be stationary, or independent but not identically distributed, or martingale differences, or otherwise weakly dependent, 
with little affect on the conclusion apart from a scale factor. 
More interesting variations are obtained by suitable conditioning.
For instance, 
assuming that the $X_i$ are integer valued, let
$o(\sqrt{n})$ denote any sequence of possible values of $S_n$ with
$o(\sqrt{n})/\sqrt{n} \te 0$ as $n \te \infty$. 
Then
\cite{MR0166814} 
\eq
\lb{donskbb}
(S_{nt}/\sqrt{n}, 0 \le t \le 1 \giv S_n = o(\sqrt{n}) ) \convd 
(\Bbb_t, 0 \le t \le 1)
\en
where $\Bbb$ is the {\em standard Brownian bridge},
that is, the centered Gaussian process with covariance
$\ER (\Bbb_s \Bbb_t) = s ( 1-t)$ and continuous paths which is obtained by conditioning 
$(B_t, 0 \le t \le 1)$ on $B_1 = 0$.
Some well known descriptions of the distribution
of $\Bbb$ are \cite[Ch. III, Ex (3.10)]{ry99} 
\begin{equation}
\lb{bbdef}
(\Bbb_t , 0 \le t \le 1) \ed  ( B_t - t B_1, 0 \le t \le 1)
\ed ((1 - t ) B_{t/(1-t)}, 0 \le t \le 1)
\end{equation}
where $\ed$ denotes equality of distributions on the path space $C[0,1]$,
and the rightmost process is defined to be $0$ for $t = 1$.
See Section \ref{sec:brmeex} for further discussion of this process.  
Let $T_{-}:= \inf \{n : S_n < 0\}$. Then 
as $n \te \infty$
\eq
\lb{donskme}
(S_{nt}/\sqrt{n}, 0 \le t \le 1 \giv T_{-} > n ) \convd 
(\Bme_t, 0 \le t \le 1)
\en
where $\Bme$ is the {\em standard Brownian meander} \cite{igl74,bol76}, and
as $n \te \infty$ through possible values of $T_{-}$
\eq
\lb{donskex}
(S_{nt}/\sqrt{n}, 0 \le t \le 1 \giv T_{-} = n ) \convd 
(\Bex_t, 0 \le t \le 1)
\en
where $\Bex_t$ is the {\em standard Brownian excursion\index{Brownian excursion}} \cite{kaigh76,csm81}.
Informally, 
$$
\begin{array}{lll}
\Bme& \ed & ( B \giv B_t > 0 \mbox{ for all } 0 < t < 1 )\\
\Bex& \ed & ( B \giv B_t > 0 \mbox{ for all } 0 < t < 1 , B_1 = 0)
\end{array}
$$
where $\ed$ denotes equality in distribution.
These definitions of conditioned Brownian motions
have been made rigorous in a number of ways: for instance by the method of Doob 
$h$-transforms \cite{kni81,sal84,fpy92}, 
and as weak limits as $\eps \downarrow 0$ 
of the distribution of $B$ given suitable events
$A_\eps$, as in \cite{dim77,blu83}, for instance 
\eq
\lb{conme}
(B \giv \underl{B}(0,1) > - \eps ) \convd  \Bme \mbox{ as } \eps \downarrow 0
\en
\eq
\lb{conex}
(\Bbb \giv \underl{\Bbb}(0,1) > - \eps ) \convd  \Bex \mbox{ as } \eps \downarrow 0
\en
where 
$\underl{X}(s,t)$ denotes the infimum of a process $X$ over the interval 
$(s,t)$.
See Section \ref{sec:brmeex} for further treatment of Brownian bridges, excursions and meanders.

The standard Brownian bridge arises also as a weak limit of \emph{empirical processes}:
for $U_1, U_2, \ldots$ a sequence of independent uniform $[0,1]$ variables, and
$$
H_n(t):=\sqrt{n}\:(\frac{1}{n}\sum_{i=1}^{n}1 (U_i\leq t)-t)
$$
so that 
$$\ER [ H_n(t) ] =0, \;\; \; \ER [ (H_n(t))^2  ] =t(1-t)$$
it is found that
$$
(H_n(t), 0 \le t \le 1) \convd (\Bbb_t, 0 \le t \le 1)
$$
in the sense of convergence of finite-dimensional distributions, and also
in the sense of weak convergence in the function space $D[0,1]$ of right-continuous
paths with left limits, equipped with the Skorokhod topology.
See \cite{sw86} for the proof and applications to empirical process theory.

Some further references related to random walk approximations
are Lawler \cite{lawler91i},
Spitzer \cite{MR52:9383}, 
Ethier and Kurtz \cite{MR88a:60130},   
Le Gall \cite{MR89g:60219}.

\section{BM as a Gaussian process}

A {\em Gaussian process} with index set $I$ is a collection of random variables
$(X_t, t \in I)$, defined on a common probability space
$(\Omega, \FF, \PR)$, such that every finite linear combination of these
variables $\sum_i a_i X_{t_i}$ has a Gaussian distribution.
The finite-dimensional distributions of a Gaussian process 
$(X_t, t\in I)$ are uniquely determined by the mean function 
$t\rightarrow \ER (X_t)$, which can be arbitrary, and the covariance function 
$(s,t)\rightarrow \ER ( X_s X_t) - \ER(X_s) \ER(X_t)$,
which must be symmetric and non-negative definite.
A Gaussian process is called {\em centered} if $\ER(X_t) \equiv 0 $.
Immediately from the definition of Brownian motion, there is 
the following characterization:
a real valued process $(B_t, t\geq 0)$ is a Brownian motion
starting from 0 
iff
\begin{itemize}
\item[(a)] $(B_t)$ is a centered Gaussian process with covariance function
\eq
\lb{bcov}
\ER[B_sB_t] =s\wedge t \mbox{ for all } s,t \geq 0;
\en
\item[(b)] with probability one, $t\rightarrow B_t$ is continuous.
\end{itemize}
Note that for a centered process $B$,
formula \re{bcov} is equivalent to
\eq
\lb{bsq}
B_0 = 0 \mbox{ and } \ER [ ( B_t - B_s )^2 ] = | t - s|  .
\en
The existence of Brownian motion can be deduced from Kolmogorov's general
criterion 
\cite[Theorem (25.2)]{rw1} 
for existence of a continuous version of a stochastic process.
Specialized to a centered Gaussian process $(X_t, t \in \reals^n)$, 
this shows that a sufficient condition for existence of a continuous version is that
$\ER( X_s X_t)$ should be locally H\"older continuous 
\cite[Corollary (25.6)]{rw1}.

\subsection{Elementary transformations}
This characterization of BM as a Gaussian process is often useful in checking 
that a process is a Brownian motion, as in the 
the following transformations of
a Brownian motion $(B_t, t\geq 0)$ starting from 0.

\begin{description}
\item[Brownian scaling]
For each fixed $c>0$, 
\eq\lb{bscale}
(c^{-1/2}B_{ct}, t\geq 0) 
\ed (B_t, t \ge 0) .
\en
\item[Time shift]
For each fixed $T>0$, 
\eq\lb{shift}
(B_{T+t}-B_{T}, t\geq 0) \ed (B_T, t \ge 0) .
\en
and the shifted process $(B_{T+ t}-B_{t}, t\geq 0)$ is independent of $(B_{u}, 0\leq u\leq T)$.
This is a form of the \emph{Markov property} of Brownian motion, discussed
further in the next section. 

\item[Time reversal] for each fixed $T >0$
$$
(B_{T-t}-B_T,  0\leq t\leq T)\ed (B_t, 0\leq t\leq T).
$$

\item [Time inversion]
\eq
\lb{tinv}
( tB_{1/t}, t > 0 ) \ed (B_t, t > 0 ) .
\en
\end{description}

\subsection{Quadratic variation}
Consider a subdivision of $[0,t]$ say
$$
0 = t_{n,0} < t_{n,1} < \cdots < t_{n,k_n} = t
$$
with mesh
$$
\max_{i} ( t_{n,i+1} - t_{n,i} ) \te 0 \mbox{ as } n \te \infty .
$$
Then
\eq
\lb{qvar}
\sum_i ( B_{t_{n,i+1} } - B_{t_{n,i}} )^2  \to t \mbox{ in }  L^2
\en
with convergence almost surely if the partitions are nested 
\cite[Th. 13.9]{kall02f}.
An immediate consequence is that the Brownian path 
has unbounded variation on every interval almost surely.
This means that stochastic integrals such as $\int_0^\infty f(t) dB_t $ cannot 
be defined in a naive way as Lebesgue-Stieltjes integrals. 

\subsection{Paley-Wiener integrals}
It follows immediately from \re{bsq} that
\eq
\lb{bsq1}
\ER \left[  \left( \sum_{i = 1}^ n a_i ( B_{t_{i + 1} } - B_{t_i} ) \right)^2 \right]
= \sum_{i = 1}^n a_i^2 (t_{i +1} - t_i )
\en
and hence that the \emph{Gaussian space} generated by $B$ is precisely the
collection of Paley-Wiener integrals
\eq
\lb{bsq2}
\left\{ \int_0^\infty f(t) dB_t   \mbox{ with } f \in L^2(\reals_+, dt ) \right\}
\en
where the stochastic integral is by definition
$\sum_{i = 1}^ n a_i ( B_{t_{i + 1} } - B_{t_i} )$ for
$f(t) = \sum_{i = 1}^ n a_i 1(t_{i} < t \le {t_{i+1}} )$,
and the definition is extended to $f \in L^2(\reals_+, dt ) $ by linearity
and isometry, along with the general formula
\eq
\lb{bsq3}
\ER \left[ \left( \int_0^\infty f(t) dB_t \right) ^2 \right] = \int_0^\infty f^2(t) dt.
\en
See \cite[\S 5.1.1]{MR1980149}
for background and further discussion.
The identity \re{bsq3} gives meaning to the intuitive idea of $dB_t$
as \emph{white noise}, whose intensity is the Lebesgue measure $dt$. 
Then $B_t$ may be understood as the total noise in $[0,t]$.
This identity also suggests the well known construction of BM from a
sequence of independent standard Gaussian variables $(G_n)$ as
\eq
\lb{gseq}
B_t = \sum_{n = 0}^\infty ( f_n, 1_{[0,t]} ) G_n
\en
where $(f_n)$ is any orthonormal basis of $L^2(\reals_+, du )$,
and
$$
( f_n, 1_{[0,t]} )  = \int_0^t f_n(u) du
$$
is the inner product of $f_n$ and the indicator $1_{[0,t]}$ in  
the $L^2$ space.
Many authors have their preferred basis:
L\'evy \cite{MR0044774}, 
\cite{MR1188411}, 
Knight \cite{MR82m:60098}, 
Ciesielski \cite{MR1215769}. 

Note also that once one BM $B$ has been defined on some probability
space, e.g. via \re{gseq}, then many other Brownian motions $B^{(k)}$
can be defined on the same probability space via
$$
B_t^{(k)} = \int_0^\infty k(t,u) d B_u
$$
for a bounded kernel $k$ such that
$\int_0^\infty k^2(t,u) du < \infty $ and
$$
\int_0^\infty ( k(t,u) - k(s,u) )^2 du = (t-s)  ~~~~\mbox{ for } 0 < s < t .
$$
Such a setup is called a non-canonical representation of Brownian motion.
Many such representations were studied by \Lev; See also
Hida and Hitsuda 
\cite{MR1216518}, 
Jeulin and Yor \cite{MR1232001}. 

\subsection{Brownian bridges}
\label{sec:bridges}
It is useful to define, as explicitly as possible,
a family of \emph{Brownian bridges} 
$\{(B^{x,y,T}_{u}, 0 \le u \le T), x,y \in \reals \}$ distributed
like a Brownian motion $(B_u, 0 \le u \le T)$ conditioned on $B_0 = x$ and $B_T = y$.
To see how to do this, assume first that $x = 0$, and write
$$
B_u = \left( B_u - \frac{u }{T} B_T \right) +  \frac{u }{ T } B_T
$$
Observe that each of the random variables $B_u - (u /T) B_T$ is orthogonal
to $B_T$, and hence the process $(B_u - (u /T) B_T, 0 \le u \le T)$ is independent of $B_T$.
It follows that the desired family of bridges can be constructed from
an unconditioned Brownian Motion $B$ with $B_0 = 0$ as
$$
B^{x,y,T}_{u} = x + B_u - (u/T) B_T + (u/T) (y-x)  = x + u ( y-x) + \sqrt{T}\Bbb(u/T)
$$
for $0 \le u \le T, x,y \in \reals$,
where $\Bbb$ is the \emph{standard Brownian bridge} as in \eqref{donskbb}.

\subsection{Fine structure of Brownian paths}

Regarding finer properties of Brownian paths, such as \Lev's modulus of continuity,
Kolmogorov's test, laws of the iterated lograrithm, upper and lower functions, Hausdorff measure of various
exceptional sets, see It\^o and McKean \cite{MR33:8031} for an early account, and M\"orters and Peres \cite{MR2604525} for a more recent exposition.

\subsection{Generalizations}
We mention briefly in this section a number of Gaussian processes which
generalize Brownian motion by some extension of either the covariance function or the index set.


\subsubsection{Fractional BM}
A natural generalization of Brownian motion is defined by
a centered Gaussian process with \re{bsq} generalized to
$$
\ER \left[ ( B_t^{(H)} - B_s^{(H)} ) ^2 \right] = | t -s |^{2H}
$$
where $H$ is called the \emph{Hurst parameter}.
This construction is possible only for $H \in (0,1]$,
when such a \emph{fractional Brownian motion} 
$(B_t^{(H)}, t \ge 0 )$ can be constructed from a standard Brownian
motion $B$ as
$$
B_t^{(H)} = C_\alpha \int_{0}^\infty (  u ^{\alpha} - (u - t )_+^{\alpha} ) d B_u
$$
for $\alpha = H - 1/2$ and some universal constant $C_\alpha$.
Early work on fractional Brownian motion was done by
Kolmogorov \cite{MR0003441},
Hurst \cite{hurst1951long}, 
Mandelbrot and Van Ness \cite{MR0242239}. 
See also \cite{MR3075342} 
for a historical review.
It is known that fractional BM is not a semimartingale except for $H = 1/2$ or $H = 1$.
See \cite{MR2052267} for an introduction to white-noise theory and 
Malliavin calculus for fractional BM. Other recent texts on fractional BM are
\cite{MR3076266} 
\cite{MR2378138} 
\cite{MR2387368}. 
See also the survey article \cite{MR3466837} on 
fractional Gaussian fields.

\subsubsection{L\'evy's BM}
This is the centered Gaussian process $Y$ indexed by $\reals^\delta$ such that
$$
Y_0 = 0  \mbox{ and } \ER ( Y_x - Y_y )^2 = |x-y| .
$$
McKean 
\cite{mckean1963brownian} 
established the remarkable fact that this process has a spatial Markov property in odd dimensions, but not in even dimensions.
See also
\cite{chentsov57}, 
\cite{MR0391282}, 
\cite{MR0394835}. 

\subsubsection{Brownian sheets}
\newcommand{\deltaN}{N}
Instead of a white noise governed by Lebesgue measure on the line, consider 
a white noise in the positive orthant of $\reals^\deltaN$ for some $\deltaN = 1,2, \ldots$,
with intensity Lebesgue measure. Then for $t_1, \ldots, t _\deltaN \ge 0$ the
noise in the box $[0,t_1] \times \cdots \times [0,t_\deltaN]$ is a centered Gaussian variable
with variance $\prod_{i = 1}^\deltaN t_i$, say $X_{t_1, \ldots, t_d}$, and
the covariance function is 
$$
\ER ( X_{s_1, \ldots, s_ \deltaN} X_{t_1, \ldots, t_ \deltaN} ) = \prod_{i = 1}^\deltaN (s_i \wedge t_i )
$$
This $N$ parameter process is known as the
\emph{standard Brownian sheet}.
See Khosnevisan \cite{MR2004a:60003} for a general treatment of multiparameter processes,
and Walsh \cite{MR876085} for an introduction to the related area of stochastic partial differential equations.

\subsection{References}

\paragraph{Gaussian processes : general texts}
\begbib
\item[\cite{MR82h:60103}]R. ~J. Adler. \newblock {\em The geometry of random fields} (1981). 
 \newblock John Wiley \& Sons Ltd., Chichester, 1981. \newblock Wiley Series in Probability and Mathematical Statistics.   

\item[\cite{MR92g:60053}]R. ~J. Adler. \newblock {\em An introduction to continuity, extrema, and related topics for   general {G}aussian processes} (1990).

\item[\cite{MR56:6829}]H.~Dym and H.~P. McKean. \newblock {\em Gaussian processes, function theory, and the inverse spectral   problem} (1976).

\item[\cite{MR99f:60078}]X. Fernique. \newblock {\em Fonctions al\'eatoires gaussiennes, vecteurs al\'eatoires   gaussiens} (1997).

\item[\cite{MR1216518}]
T. Hida and M. Hitsuda. \newblock {\em Gaussian processes} (1993).

\item[\cite{MR22:10012}]T. Hida. \newblock {\em Canonical representations of {G}aussian processes and their   applications} (1960).

\item[\cite{MR39:4935}]M. Hitsuda. \newblock {\em Representation of {G}aussian processes equivalent to {W}iener   process } (1968). 

\item[\cite{MR99f:60082}]S. Janson. \newblock {\em Gaussian {H}ilbert spaces} (1997).

\item[\cite{MR98k:60059}]M.~A. Lifshits. \newblock {\em Gaussian random functions} (1995). 

\item[\cite{MR0272042}] \newblock {\em Processus al\'eatoires gaussiens} (1968).

\item[\cite{MR32:8408}]L.~A. Shepp. \newblock {\em Radon-{N}ikod\'ym derivatives of {G}aussian measures } (1966).
\endbib

\paragraph{Gaussian free fields}

\begbib
\item[\cite{MR2322706}] S. Sheffield {\em Gaussian free fields for mathematicians} (2007).
\item[\cite{MR3052311}] N.-G. Kang and N. K. Makarov. {\em Gaussian free field and conformal field theory} (2013).
\endbib

\section{BM as a Markov process}

\subsection{Markov processes and their semigroups}
A \emph{Markov process} $X$ with measurable state space $(E,\EE)$ 
is an assembly of mathematical objects
$$
X = ( \Omega, \FF, (\FF_t)_{t \ge 0} , (X_t) _{t \ge 0 }, (P_t)_{t \ge 0 }, \{ \PR^x \}_{x \in E} )
$$
where
\begin{itemize}
\item  $( \Omega, \FF )$ is a measurable space, often a \emph{canonical space}
of paths in $E$ subject to appropriate regularity conditions;
\item   $(\FF_t)$ is a filtration;
\item   $X_t: \Omega \to E$ is an $\FF_t/\EE$ measurable random variable, regarded as representing the state of the process at time $t$;
\item   $P_t: E \times \EE \to [0,1]$ is a \emph{transition probability kernel}
on $(E, \EE)$, meaning that $A \to P_t(x, A)$ is a probability distribution on
$(E, \EE)$, and for each fixed $x \in E$,  and $A \to P_t(x, A)$ is $\EE$ measurable;
\item the $P_t$ satisfy the \emph{Chapman-Kolmogorov equation}
\eq
\lb{ck}
P_{s+t}(x,A) = \int_E P_s(x,dy) P_t(y,A) ~~~~~~~~~~(x \in E, A \in \EE )
\en
\item under the probability law $\PR^x$ on $( \Omega, \FF )$, the process
$(X_t)$ is Markovian relative to $(\FF_t)$ with transition probability kernels
$(P_t)$, meaning that
\eq
\lb{ck}
\ER^x [ f( X_{s+t} ) \giv \FF_s ] = (P_t f ) (X_s)   ~~~~~\PR^x \mbox{ a.s. }
\en
where the left side is the $\PR^x$ conditional expectation of 
$f( X_{s+t} )$  given $\FF_s$,
with $f$ an arbitrary bounded or non-negative
$\EE$ measurable  function, and on the right side $P_t$ is regarded as
an operator on such $f$ according to the formula
\eq
\lb{ptf}
(P_t f )(x)  = \int_E P_t(x,dy) f(y) .
\en
\end{itemize}
All Markov processes discussed here will be such that $P_0(x,A) = 1(x \in A)$,
so $\PR^x(X_0 = x) = 1$, 
though this condition is relaxed in the theory of Ray processes \cite{rw1}.
See \cite{rw1,ry99,ms88} for background and treatment of Markov processes at
various levels of generality. 

The conditioning formula \re{ck} implies the following description of
the finite-dimensional distributions of $X$: for $0 \le t_1  < t_2 \cdots < t_n:$
$$
\PR^x ( X_{t_i} \in dx_i , 1 \le i \le n )
= P_{t_1}(x,dx_1) P_{t_2 - t_1}(x_1, dx_2) \cdots P_{t_{n} - t_{n-1}}(x_{n-1} , dx_n )
$$
meaning that for all bounded or non-negative product measurable functions  $f$
$$
\ER^x [ f(X_{t_1}, \ldots, X_{t_n}) ]
= \int_E \cdots \int_E f(x_1, \ldots, x_n) P_{t_1}(x,dx_1) \cdots P_{t_{n} - t_{n-1}}(x_{n-1} , dx_n )
$$ 
where the integration is done first with respect to $x_n$ for fixed $x_1, \ldots, x_{n-1}$, then
with respect to $x_{n-1}$, for fixed $x_1, \ldots, x_{n-2}$, and so on.
Commonly, the transition kernels $P_t$ are specified by 
\emph{transition probability densities} $p_t(x,y)$ relative to some reference measure
$m(dy)$ on  $(E, \EE)$, meaning that $P_t(x,dy) = p_t(x,y) m(dy)$. In terms of such
densities, the Chapman-Kolmogorov equation becomes
\eq
\lb{ckd}
p_{s+t}(x,z) = \int_E m(dy) p_s(x,y) p_t(y,z)
\en
which in a regular situation is satisfied for all $s,t \ge 0$
and all $x,z \in E$. In terms of the operators $P_t$  defined on bounded or non-negative
measurable functions on $(E, \EE)$ via \re{ptf}, the Chapman-Kolmogorov equations correspond to
the \emph{semigroup property}
\eq
\lb{semig}
P_{s+t} = P_s\circ P_t\ \ \ \ \ \ \ \ \ s,t\geq 0 .
\en
The general analytic theory of semigroups, in particular the
Hille-Yosida theorem \cite[Theorem III (5.1)]{rw1} is therefore available as
a tool to study Markov processes.

Let $\PR^0$ be Wiener measure i.e. the
distribution on $\Omega:= C[0,\infty)$ of a standard Brownian motion 
starting from $0$, where $\Omega$ is equipped with the sigma-field
$\FF$ generated by the coordinate process $(X_t, t \ge 0)$ defined by
$X_t(\omega) = \omega(t)$.
Then standard Brownian is realized under $\PR^0$ as 
$B_t = X_t$.
Let $\mathbb{P}^x$ be the $\PR^0$ distribution of $(x+B_t, t\geq 0)$ on $\Omega$.
Then a Markov process with state space $E = \reals$,
and $\EE$ the Borel sigma-field,
is obtained from this canonical setup with the
\emph{Brownian transition probability density function}
relative to Lebesgue measure on $\reals$
\eq
\lb{btf}
p_t(x,y) = \frac{1}{\sqrt{2\pi t}}e^{-\frac{(y-x)^2}{2t}}
\en
which is read from the Gaussian density of Brownian increments.
The Chapman-Kolmogorov identity \re{ckd} then reflects how the
sum of independent Gaussian increments is Gaussian.

This construction of Brownian motion as a Markov process 
generalizes straightforwardly to
$\reals^d$ instead of $\reals$, allowing a Gaussian distribution $\mu_t$
with mean $b  t $ and covariance matrix $\sigma \sigma^{T}$ for some $d \times d$
matrix $\sigma$ with transpose $\sigma^T$. The semigroup is then specified by
\eq
\lb{shom}
P_t f (x) = \int_{\reals^d} f(x + y ) \mu_t(dy)
\en
for all bounded or non-negative Borel measurable functions $f: \reals^d \to \reals$.
The corresponding Markovian laws $\PR^x$ on the space of continuous paths in
$\reals^d$ can be defined by letting $\PR^x$ be the law of the process
$(x + \sigma B_t + b t, t \ge 0)$
where $B$ is a \emph{standard Brownian motion in $\reals^d$}, that is a process whose coordinates are $d$ independent standard one-dimensional Brownian motions.
According to a famous result of \Lev\ \cite[Theorem (28.12)]{rw1}, this construction
yields the most general Markov process $X$ with state space $\reals^d$, continuous
paths, and transition operators of the spatially homogenous form \re{shom} corresponding
to stationary independent increments.
More general spatially homogeneous Markov processes $X$ with state space $\reals^d$, 
realized with paths which are right continuous with left limits,
are known as \emph{\Lev\ processes}. These correspond via \re{shom} to convolution semigroups of 
probability measures $(\mu_t, t \ge 0 )$ generated by an infinitely divisible distribution $\mu_1$ on
$\reals^d$. See \cite[\S 28]{rw1}, \cite{bert96lev}, \cite{sato99} for treatments of \Lev\ processes.

Another important generalization of Brownian motion is obtained by considering Markov processes
where the spatial homogeneity assumption \re{shom} is relaxed, but continuity of paths is retained.
Such Markov processes, subject to some further regularity conditions which vary from one authority to 
another, are called \emph{diffusion processes}. The state space can now be $\reals^d$, or a suitable subset of $\reals^d$, 
or a manifold.  The notion of infinitesimal generator of a Markovian semigroup, discussed further 
in Section \ref{infgen} is essential for the development of this theory.

\subsection{The strong Markov property}

If $B$ is an $(\FF_t)$ Brownian motion then for each fixed time $T$ the process
\eq
\lb{future}
(B_{T+s} - B_T, s \ge 0 )
\en
is a Brownian motion independent of $\FF_T$. According to the \emph{strong Markov property}
of Brownian motion, this is true also for all $(\FF_t)$ \emph{stopping times}
$T$, that is random times $T: \Omega\rightarrow \reals_+ \cup\{\infty\}$ 
such that $(T\leq t) \in \mathcal{F}_{t}$ for all $t \ge 0$. 
The sigma-field $\FF_{T}$ of events determined by time $T$ is defined as: 
$$\displaystyle \mathcal{F}_{T}:=\{A\in \FF: A \cap (T\leq t) \in \mathcal{F}_{t}\} $$
and the process \re{future} is considered only conditionally on the event $T < \infty$.
See \cite{MR2722836}, \cite{bill95} for the proof and numerous applications.
More generally, a Markov process $X$ with filtration $(\FF_t)$ 
is said to have the strong Markov property 
if for all $(\FF_t)$ stopping times
$T$, conditionally given $\FF_T$ on $(T < \infty)$ with $X_T = x$ the process  $(X_{T+s}, s \ge 0 )$ is
distributed like $(X_s, s \ge 0 )$ given $X_0 = x$.
It is known \cite[Th. III (9.4)]{rw1} that the strong Markov property holds for \Lev\ processes, and more generally
for the class of \emph{\FD\ processes} $X$ defined by the following regularity conditions:
the state space $E$ is
locally compact with countable base, $\EE$ is the Borel sigma field on $E$, and
the transition probability operators $P_t$
act in a strongly  continuous way on the Banach space $C_0(E)$ of bounded continuous functions on $E$ 
which vanish at infinity 
\cite[Ch. 19]{kall02f}, 
\cite[III.2]{ry99},
\cite[Def. (6.5)]{rw1}.

%

\subsection{Generators}
\lb{infgen}

The \emph{generator} $\gen$ of a Markovian
semigroup $(P_t)_{t \ge 0}$ is the operator defined as
\eq
\lb{gendef}
\gen f:= \lim_{t\downarrow 0}\frac{P_tf-f}{t}
\en
for suitable real-valued functions $f$, meaning that the limit exists in some sense,
e.g. the sense of convergence of functions in a suitable Banach
space, such as the space $C_0(E)$ involved in the definition of
Feller-Dynkin processes.
Then $f$ is said to belong to the \emph{domain of the generator}.
From \re{gendef} it follows that for $f$ in the domain of the generator
\eq
\lb{kol1}
\frac{ d }{ dt  } P_t f = \GG  P_t f = P_t \GG f
\en
and hence
\eq
\lb{ptf1}
P_t f -f  = \int_0^t \GG P_s f ds = \int_0^t P_s \GG f ds 
\en
in the sense of strong differentiation and Riemann integration in 
a Banach space.
In particular, if $(P_t)$ is the semigroup of Brownian motion, then it
is easily verified 
\cite[p. 6]{rw1}
that for $f \in C_0(\reals)$ with two bounded continuous derivatives,
the generator of the standard Brownian semigroup is given by
\eq
\lb{bgen}
\gen f = \hf f^{\prime \prime} = \frac{1 }{ 2 }  \frac {d^2 f }{ dx ^2 }
\en
In this case, the first equality in \re{kol1} reduces to \emph{Kolmogorov's backward equation}
for the Brownian transition density:
$$
{ \partial \over \partial t } p_t(x,y) = 
{1 \over 2}
{\partial ^2 \over \partial x^2 } p_t (x,y) .
$$
Similarly, the second equality in \re{kol1} yields \emph{Kolmogorov's forward equation}
for the Brownian transition density:
$$
{ \partial \over \partial t } p_t(x,y) = 
{1 \over 2}
{\partial ^2 \over \partial y^2 } p_t (x,y) . 
$$
This 
partial differential equation is also known as the \emph{heat equation},
due to its physical interpretation in terms of heat flow  \cite{MR0079376}. 
Thus 
for each fixed $x$ the Brownian transition density function $p_t (x,y)$
is identified as the fundamental solution of the heat equation  with a pole at $x$
as $t \downarrow 0$.
If we consider instead of standard Brownian motion $B$ a Brownian motion with
drift $b$ and diffusion coefficient $\sigma$, we find instead of \re{bgen} that
the generator acts on smooth functions of $x$ as
\eq
\lb{bgen1}
\gen = b { d \over d x } +  \hf \sigma^2 {d^2 \over dx ^2 }  .
\en
This suggests that given some space-dependent
drift and variance coefficients $b(x)$ and $\sigma^2(x)$ subject to
suitable regularity conditions, a Markov process which behaves when started 
near $x$ like a Brownian motion with drift $b(x)$ and variance $\sigma(x)$
should have as its generator the second order differential operator
\eq
\lb{bgen2}
\gen = b(x) { d \over d x } +  \hf \sigma^2 (x) {d^2 \over dx ^2 } 
\en
Kolmogorov 
\cite{MR1512678} 
showed that the semigroups of such Markov processes 
could be constructed by establishing the existence of suitable solutions
of the Fokker-Planck-Kolmogorov equations determined by this generator.
More recent approaches to the existence and uniqueness of such diffusion processes
involve martingales in an essential way, as we discuss in the next section.

\subsection{Transformations}
The theory of Markov processes provides a number of ways of starting
from one Markov process $X$ and transforming it into another Markov process
$Y$ by some operation on the paths or law of $X$. 
The semigroup of $Y$ can then typically be derived quite 
simply and explicitly from the semigroup of $X$. Such operations include
suitable transformations of the state-space, time-changes, and killing.
Starting from $X = B$ a Brownian motion in $\reals$ or $\reals^d$, these
operations yield a rich collection of Markov processes whose properties
encode some features of the underlying Brownian motion.

\subsubsection{Space transformations}
The simplest means of transformation of a Markov process $X$ with state
space $(E,\EE)$ is to consider the distribution of the process 
$(\Phi(X_t),t \ge 0 )$  for a suitable measurable $\Phi:E \to E ^\prime$ for 
some $(E^\prime, \EE^\prime)$.
Typically, such a transformation destroys the Markov property, unless
$\Phi$ respects some symmetry in the dynamics of $X$, as follows 
\cite[I (14.1)]{rw1}. 
Suppose that $\Phi$ maps $E$ onto $E'$, and that for each 
$x \in \EE$ the $P_t(x, \cdot)$ distribution of $\Phi$ depends only on $\Phi(x)$, so that
\eq
\lb{qdef}
\PR^x [\Phi(X_t) \in A^\prime ] = Q _t( \Phi(x), A ^\prime)     ~~~~(x \in E, A^\prime \in \EE^\prime )
\en
for some family of Markov kernels $(Q_t, t \ge 0)$  
on $(E^\prime, \EE^\prime)$.
Assuming for simplicity that $X$ has continuous paths, 
that $\Phi$ is continuous,
and that $\PR^x$ governs $X$ with semigroup $(P_t)$ with $X_0 = x$. 
Then the $\PR^x$ distribution of $(\Phi(X_t), t \ge 0 )$ is that of a
Markov process with semigroup $(Q_t)$ and initial state $\Phi(x)$. 
Refer to Dynkin \cite{dyn65}, 
Rogers-Williams \cite{rw1}, 
Rogers and Pitman \cite{rp81}
Glove and Mitro \cite{gl90}.
Let $\QR^y$ for $y \in \EE'$ denote the
common distribution of this process on $C([0,\infty), E^\prime)$ for all
$ x $ with $\Phi(x) = y$. Then $(\QR^y, y \in E')$ defines the collection
of laws on $C([0,\infty), E^\prime)$ of the canonical Markov process
with semigroup $(Q_t, t \ge 0 )$. Following is a well known example.

\subsubsection{Bessel processes}
\label{sec:bes}
Let $B_t = (B_t^{(i)} , 1 \le i \le \delta )$ be a $\delta$-dimensional Brownian motion for some
fixed positive integer $\delta$. Let
\begin{equation}
\lb{normbm}
R_t ^{(\delta)}: = |B_t| = \sqrt{  \sum_{i = 1}^\delta (B_t^{(i)} )^2 }
\end{equation}
be the radial part of $(B_t)$.
If $(B_t)$ is started at $x \in \reals^\delta$, and $O_{x,y}$ is an orthogonal linear transformation of $\reals^\delta$ which maps $x$ to $y$, with
$|O_{x,y}(z)| = |z|$ for all $z \in \reals^\delta$, then $(O_{x,y}(B_t))$ is 
a Brownian motion starting at $y$ whose radial part is pathwise identical 
to the radial part of $(B_t)$. It follows immediately from this observation
that $(R_t ^{(\delta)}, t \ge 0 )$ given $B_0$ with $|B_0| = r$
defines a Markov process, the \emph{$\delta$-dimensional Bessel process},
with initial state $r$ and transition semigroup $(Q_t)$ on $[0,\infty)$
which can be defined via \re{qdef}
by integration of the Brownian transition density function
over spheres in $\reals ^\delta$. That
gives an explicit formula for the transition density of the
Bessel semigroup in terms of Bessel functions \cite[p. 446]{ry99}. 
Since the generator of $B$ acting on
smooth functions is the half of the $\delta$-dimensional Laplacian
$$
{1 \over 2 } \sum_{i = 1}^\delta { d^2 \over dx_i ^2 }
$$
it is clear that the generator of the $\delta$-dimensional Bessel process,
acting on smooth functions with domain $[0,\infty)$ which vanish in 
a neighbourhood of $0$,
must be half of the radial part of the Laplacian, that is 
\eq
\lb{besd}
{\delta - 1  \over 2 r }   {d \over dr } + {1 \over 2 } {d^2 \over dr^2 }
\en
In dimensions $\delta \ge 2$, it is known \cite[Ch. XI]{ry99} that this action of the
generator uniquely determines the Bessel semigroup $(Q_t)$, 
essentially because when started away from $0$ the Bessel process never
reaches $0$ in finite time, though in two dimensions it approaches
$0$ arbitrarily closely at large times. 

In dimension one, the process $(|B_t|, t \ge 0 )$ is called 
\emph{reflecting Brownian motion}. It is obvious, and
consistent with the vanishing drift term in formula \re{besd} for 
$\delta = 1$, that the reflecting Brownian motion started at $x >0$ is 
indistinguishable from ordinary Brownian motion up until the random 
time $T_0$ that the path first hits $0$.  In dimension one, the expression
\re{besd} for the infinitesimal generator must be supplemented by
a \emph{boundary condition} to distinguish the reflecting motion from
various other motions with the same dynamics on $(0,\infty)$, but
different behaviour once they hit $0$. See Harrison \cite{MR798279} 
for further treatment of reflecting Brownian motion and its applications to stochastic flow systems.
See also R. Williams \cite{MR99m:60129} regarding semimartingale reflecting Brownian motions in an orthant,
and Harrison \cite{MR1994047} for a broader view of Brownian networks.

The theory of Bessel processes is often simplified by consideration of the
{\em squared Bessel process of dimension $\delta$} which for $\delta = 1,2, \ldots$
is simply the square of the norm of a $\delta$-dimensional Brownian motion:
\begin{equation}
\lb{normbmsq}
X_t^{(\delta)} := (R_t ^{(\delta)})^2: = |B_t|^2 = \sum_{i = 1}^\delta (B_t^{(i)} )^2 .
\end{equation}
This family of processes enjoys the key {\em additivity property} that if $X^{(\delta)}$ and $X^{(\delta')}$ are two
independent squared Bessel processes of dimensions $\delta$ and $\delta'$, started at values $x,x' \ge 0$, then
$X^{(\delta)} + X^{(\delta')}$ is  a squared processes of dimension $\delta + \delta'$, started at $x + x'$.
As shown by Shiga and Watanabe \cite{sw73}, this property can be used to extend the definition of the squared Bessel process to arbitrary non-negative
real values of the parameter $\delta$. The resulting process is then a Markovian diffusion on $[0,\infty)$ with generator
acting on smooth functions of $x > 0 $ according to
$$
\delta  \frac{d}{dx} + 4 x \frac{1}{2} \frac{d^2}{ dx ^2 } 
$$
meaning that the process when at level $x$ behaves like a Brownian motion with drift $\delta$ and variance parameter $4x$.
See \cite{ry99} and
\cite[\S 3.2]{MR3616274} 
for further details.
For $\delta = 0$ this is the {\em Feller diffusion} \cite{fel51}, 
which is the continuous state branching process obtained as a scaling limit of critical Galton-Watson branching processes in discrete time.
Similarly, the squared Bessel process of dimension $\delta \ge 0$ may be interpreted as a continuous state branching process immigration rate $\delta$.
See \cite{kw71}, \cite{lam67br}, \cite{lam67berk}, \cite{lind72}. 
See also \cite{MR1997032} 
for a survey and some generalizations of {B}essel processes, including the Cox-Ingersoll-Ross diffusions which are of interest in mathematical finance.

\subsubsection{The Ornstein-Uhlenbeck process}
If $(B(t), t \ge 0)$ is a standard Brownian motion  it is easily shown by Brownian scaling that the process
$(e^{-r} B(e ^{2r}), r \in \reals)$ is a two-sided stationary Gaussian Markov process, known as an {\em Ornstein-Uhlenbeck process}.
See \cite[Ex. III (1.13)]{ry99} for further discussion.


\subsection{L\'evy processes}

Processes with stationary independent increments, now called {\em L\'evy processes}, were introduced by L\'evy at the same time as he derived
the {\em L\'evy - Khintchine formula} which gives a precise representation of the characteristic functions of all infinitely divisible distributions. 
The interpretation of various components of the L\'evy - Khintchine formula shows that Brownian motion with a general drift vector and covariance
matrix is the only kind of L\'evy process with continuous paths. Other L\'evy processes can be constructed to have right continuous paths with left limits,
using a Brownian motion for their continuous path component, and a suitably compensated integral of Poisson processes to create the jumps.
Some of these L\'evy processes, called {\em stable processes}, share with Brownian motion a generalization of the Brownian scaling property.

\paragraph{Two reference books}

\begbib
\item[\cite{MR98e:60117}]Jean Bertoin. \newblock {\em L\'evy processes}, volume 121 of {\em Cambridge Tracts in   Mathematics}. \newblock Cambridge University Press, Cambridge, 1996.   

\item[\cite{sato99}]K.~Sato. \newblock {\em L\'evy processes and infinitely divisible distributions}. \newblock Cambridge University Press, Cambridge, 1999. \newblock Translated from the 1990 Japanese original, revised by the author.   
\endbib

\paragraph{Tutorials}

\begbib
\item[\cite{MR1861722}]Jean Bertoin. \newblock Some elements on {L}\'evy processes. \newblock In {\em Stochastic processes: theory and methods}, volume~19 of {\em   Handbook of Statist.}, pages 117--144. North-Holland, Amsterdam, 2001.   

\item [\cite{MR1833690}] K. Sato. {\em Basic results on {L}\'evy processes} (2001).

\endbib

\paragraph{Subordinators}
\begbib
\item[\cite{MR2002a:60001}]Jean Bertoin. \newblock Subordinators: examples and applications. \newblock In {\em Lectures on probability theory and statistics (Saint-Flour,   1997)}, volume 1717 of {\em Lecture Notes in Math.}, pages 1--91. Springer,   Berlin, 1999.   
\endbib

\paragraph{Applications}
\begbib

\item[\cite{MR2001m:60004}]Ole~E. Barndorff-Nielsen, Thomas Mikosch, and Sid Resnick, editors. \newblock {\em L\'evy processes}. \newblock Birkh\"auser Boston Inc., Boston, MA, 2001. \newblock Theory and applications.   

\item[\cite{MR96k:82057}]Michael~F. Shlesinger, George~M. Zaslavsky, and Uriel Frisch, editors. \newblock {\em L\'evy flights and related topics in physics}, volume 450 of   {\em Lecture Notes in Physics}, Berlin, 1995. Springer-Verlag.   

\item [\cite{MR2512800}] D. Applebaum {\em L\'evy processes and stochastic calculus} (2009).

\item [\cite{MR1882824}] S. I. Boyarchenko and S. Z. Levendorski\u\i {\em Option pricing and hedging under regular {L}\'evy processes of exponential type} (2001).
 
\endbib

\paragraph{Survey papers}
\begbib
\item[\cite{MR52:6886}]N.~H. Bingham. \newblock Fluctuation theory in continuous time. \newblock {\em Advances in Appl. Probability}, 7(4):705--766, 1975.   

\item[\cite{MR2002d:60039}]Ronald Doney. \newblock Fluctuation theory for {L}\'evy processes. \newblock In {\em L\'evy processes}, pages 57--66. Birkh\"auser Boston, Boston,   MA, 2001.   

\item[\cite{MR53:4240}]Bert Fristedt. \newblock Sample functions of stochastic processes with stationary, independent   increments. \newblock In {\em Advances in probability and related topics, Vol. 3}, pages   241--396. Dekker, New York, 1974.   

\item[\cite{MR52:15692}]S.~J. Taylor. \newblock Sample path properties of processes with stationary independent   increments. \newblock In {\em Stochastic analysis (a tribute to the memory of Rollo   Davidson)}, pages 387--414. Wiley, London, 1973.   

\endbib

\paragraph{Generalizations}

\begbib
\item[\cite{MR2003a:47104}]N.~Jacob. \newblock {\em Pseudo differential operators and {M}arkov processes. {V}ol.   {I}}. \newblock Imperial College Press, London, 2001. \newblock Fourier analysis and semigroups.  

\item[\cite{MR2003k:47077}]N.~Jacob. \newblock {\em Pseudo differential operators and {M}arkov processes. {V}ol.   {II}}. \newblock Imperial College Press, London, 2002. \newblock Generators and their potential theory.   

\item[\cite{MR2060091}]Ming Liao. \newblock {\em L\'evy processes in {L}ie groups}, volume 162 of {\em Cambridge   Tracts in Mathematics}. \newblock Cambridge University Press, Cambridge, 2004.  
\endbib

\subsection{References for Markov processes}
\begbib
\item[\cite{MR41:9348}]R.~M. Blumenthal and R.~K. Getoor. \newblock {\em Markov processes and potential theory} (1968).

\item[\cite{MR97g:60105}] K. L. Chung. \newblock {\em Green, {B}rown, and probability} (1995). 

\item[\cite{MR2152573}] K. L. Chung and John Walsh {\em Markov processes, Brownian motion, and time symmetry} (2005).

\item[\cite{dyn65}]E.~B. Dynkin. \newblock {\em Markov Processes, I,II} (1965).   

\item[\cite{MR2574430}]  T. M. Liggett {\em Continuous time {M}arkov processes} (2010).

\item[\cite{MR36:2219}]P.-A. Meyer. \newblock {\em Processus de {M}arkov} (1967).

\item[\cite{MR95d:60131}]M. Nagasawa. \newblock {\em Schr\"odinger equations and diffusion theory} (1993). 

\item[\cite{MR29:6542}]M. Nagasawa. \newblock {\em Time reversions of {M}arkov processes} (1964). 

\item[\cite{MR89m:60169}]M. Sharpe. \newblock {\em General theory of {M}arkov processes} (1989).

\item[\cite{dw80r}] D. Williams. \newblock Review of: {\em Multidimensional diffusion processes} by {D}. {S}troock,   {S}.{R}.{S}. {V}aradhan (1980).

\item[\cite{MR50:11483}] D. Williams. \newblock {\em Brownian motions and diffusions as {M}arkov processes} (1974).

\item[\cite{dw82l}]David Williams. \newblock {\em Lecture on {B}rownian motion given in {L}iverpool} (1982).
\endbib

\section{BM as a martingale}

Let $(\FF_t, t \ge 0)$ be a filtration in a probability space $(\Omega, \FF, \PR)$.
A process $M = (M_t)$ is called an $(\FF_t)$ \emph{martingale} if
\begin{enumerate}
\item
$M_t$ is  $\FF_t$ measurable for each $t \ge 0.$
\item
$\ER(M_t|\mathcal{F}_s)=M_s$ for all $0\leq s\leq t$.
\end{enumerate}
Implicitly here, to make sense of the conditional expectation, it is assumed that 
$\ER |M_t | < \infty$.
It is known that if the filtration $(\mathcal{F}_t)$ is right continuous, i.e. $\mathcal{F}_t^+=\mathcal{F}_t$ up to null sets,
then every martingale has a version which has right continuous paths (even with left limits). See \cite{ry99}.

If $B$ is a standard Brownian motion relative for a filtration $(\FF_t, t \ge 0)$, meaning that $B_{t+s} - B_t$ is independent of $\FF_t$ with 
Gaussian distribution with mean $0$ and variance $s$, for each $s,t \ge 0$, then both $(B_t)$ and $(B_t^2 - t )$ are $(\FF_t)$ martingales.
So too is $$M_t^{(\theta)}=\exp(\theta B_t-\theta^2t/2)$$ for each $\theta$ real or complex, where a process with values in the complex plane, or in
$\reals^d$ for $d \ge 2$, is called a martingale if each of its one-dimensional components is a martingale.

\paragraph{Optional Stopping Theorem} If ($M_t$) is a right
continuous martingale relative to a right continuous filtration
$(\mathcal{F}_t)$, and $T$ is a stopping time for
($\mathcal{F}_t$), meaning $(T \le t ) \in {\mathcal{F}}_t$ for each $t \ge 0$, and $T$ is bounded, i.e., $T\leq C<\infty$ for some constant
$C$,
then $$\ER M_T= \ER M_0.$$
Moreover, if an $(\mathcal{F}_t)$ adapted process $M$ has this property
for all bounded $(\mathcal{F}_t)$ stopping times $T$, then
$M$ is an $(\mathcal{F}_t)$ martingale.
Variations or corollaries with same setup:
If $T$ is a stopping time (no bound now), then
\[\ER (M_{T\wedge t}|\mathcal{F}_s)=M_{T\wedge s}\]
so $(M_{T\wedge t}, t \ge 0 )$ is an
$(\mathcal{F}_t)$-martingale.
If $T$ is a stopping time of Brownian motion $B$ with $\ER(T) <\infty$,
then 
$$
\ER(B_T) =0 \mbox{ and } \ER (B_T^2)=\ER(T).
$$
\subsection{L\'evy's characterization}
\lb{asmg}
Let 
$$
( B_t:= (B_t^{(1)}, \ldots, B_t^{(d) }); t \ge 0 )
$$
denote a $d$-dimensional process, 
and let
$$(\FF_t := \sigma \{ B_s, 0 \le s \le t \} , t \ge 0)$$ 
denote its filtration. 
It follows immediately from the property of stationary independent
increments that if $(B_t)$ is a Brownian motion, then
\begin{itemize}
\item[(i)] $(B_t^{(i)})$ is an $(\FF_t)$ martingale with continuous paths, for each $i$;
\item[(ii)] $( B_t^{(i)} B_t^{(j)} )$ is an $(\FF_t)$ martingale for $1 \le i < j \le d $;
\item[(iii)] $ (  ( B_t^{(i)} )^2 - t ))$ is an $(\FF_t)$ martingale for each $1 \le i \le d$
\end{itemize}
It is an important result, due to \Lev, that if a 
$d$-dimensional process $(B_t)$ has these three properties relative to the
filtration $(\FF_t)$ that it generates, then $(B_t)$ is a Brownian
motion. Note how a strong conclusion regarding the distribution of
the process is deduced from what appears to be a much weaker collection
of martingale properties.
Note also that continuity of paths is essential: if $(B_t)$ is any
process with stationary independent increments such that 
$\ER( B_1 ) = 0 $ and $\ER (B_1^2 ) = 1$, for instance $B_t := N_t - t$
where $N$ is a Poisson process with rate $1$, then both $(B_t)$ and
$(B_t ^2 - t )$ are martingales.  

More generally, if a process $(B_t)$ has the above three properties
relative to some filtration $(\FF_t)$ with
$$\FF_t \supseteq \BB_t:= \sigma \{ B_s, 0 \le s \le t \} ~~~~~~(t \ge 0),$$ 
then it can be concluded that $(B_t)$ is an 
\emph{$(\FF_t)$ Brownian motion},
meaning that $(B_t)$ is a Brownian motion and that for all $0 \le s \le t$ 
the increment $B_t - B_s$ is independent of $\FF_s$.

\subsection{\Ito's formula}
It is a key observation that if $X$ is a Markov process with semigroup $(P_t)$, and $f$ and $g$ 
are bounded Borel functions with $\gen f = g$, then the equality between
the first and last expressions in the Chapman-Kolmogorov equation \re{ptf1} can be recast as
\eq
\lb{genfg}
\ER^x [ f(X_T) ] - f(x) =  \ER^x \int_0^T ds \,g (X_s)  ~~~~~(T \ge 0, x \in E ).
\en
Equivalently, by application of the Markov property,
\eq
\lb{genfg1}
M_t^f:=  f(X_t) - f(X_0) - \int_0^t ds \, g (X_s), t \ge 0 ) \mbox{ is a } (\PR^x,\FF_t) \mbox{ martingale}
\en
for all $x \in E$. The optional stopping theorem applied to this martingale
yields \emph{Dynkin's formula},
\cite[\S 10]{rw1}, according to which \re{genfg} holds also for $(\FF_t)$
stopping times $T$ with $\ER^x (T) < \infty$.
If $X = B$ is  a one-dimensional Brownian motion, and $f \in C^2_b$,
meaning that $f$, $f^\prime$ and $f^{\prime \prime}$ are all bounded and continuous, 
then $\gen f = \hf f^{\prime \prime}$ and \re{genfg1} reads
\eq
\lb{pre-ito}
f(B_t) - f(B_0) = M_t ^f + \hf 
\int_0^t  f^{\prime \prime}(B_s) ds .
\en
To identify more explicitly the martingale $M_t^f$ appearing here,
consider a subdivision of $[0,t]$ say
$$
0 = t_{n,0} < t_{n,1} < \cdots < t_{n,k_n} = t
$$
with mesh
$$
\max_{i} ( t_{n,i+1} - t_{n,i} ) \te 0 \mbox{ as } n \te \infty .
$$
By a second order Taylor expansion,
\eq \lb{dd1}
f(B_t) - f(B_0) = 
\sum_i f^\prime(B_{t_{n,i}})( B_{t_{n,i+1} } - B_{t_{n,i}} )
+ \hf 
\sum_i f^{\prime \prime} ( \Theta_{n,i})( B_{t_{n,i+1} } - B_{t_{n,i}} )^2
\en
for some $\Theta_{n,i}$ between $B_{t_{n,i} }$  and $B_{t_{n,i+1}}$. 
Since $\Theta_{n,i}$ is bounded and $( B_{t_{n,i+1} } - B_{t_{n,i}} )^2$ has mean
$t_{n,i+1}  - t_{n,i}$ and variance a constant times $(t_{n,i+1}  - t_{n,i})^2$, it is
easily verified that as $n \te \infty$ there is the following easy extension of the
fact \re{qvar} that the quadratic variation of $B$ on $[0,t]$ equals $t$:
\eq \lb{dd2}
\sum_i f^{\prime \prime} ( \Theta_{n,i}) ( B_{t_{n,i+1} } - B_{t_{n,i}} )^2
-
\sum_i f^{\prime \prime} ( \Theta_{n,i}) ( t_{n,i+1}  - t_{n,i} )
\te 0 \mbox{ in } L^2
\en
while the second sum in \re{dd2} is a Riemann sum which 
approximates $\int_0^t  f^{\prime \prime}(B_s) ds$ almost surely.
Consequently, the first sum must converge to the same limit in $L^2$, and we learn from \re{dd1} that the
martingale $M_t^f$ in \re{pre-ito} is 
\eq
\lb{dd3}
M_t^f = \int_0^t f^\prime (B_s) dB_s := 
\lim _{ n \te \infty} \sum_i f^\prime(B_{t_{n,i}})( B_{t_{n,i+1} } - B_{t_{n,i}} )
\en
where the limit exists in the sense of convergence in probability. Thus we obtain a first version 
of \emph{\Ito's formula}: for $f$ which is bounded with two bounded continuous derivatives:
\eq
\lb{pre-ito1}
f(B_t) - f(B_0) = \int_0^t f^\prime(B_s) dB_s + \hf \int_0^t  f^{\prime \prime}(B_s) ds
\en
where the \emph{stochastic integral} $(\int_0^t f^\prime(B_s) dB_s, t \ge 0)$ is an
$(\FF_t)$-martingale if $B$ is an $(\FF_t)$-Brownian motion.
\Ito's formula \re{pre-ito1}, along with an accompanying theory of stochastic integration
with respect to Brownian increments $dB_s$, has been extensively generalized to a
theory of stochastic integration with respect to semi-martingales \cite{ry99}. 
The closely related theory of \emph{Stratonovich stochastic integrals} is obtained
by defining for instance
\eq
\lb{stratint}
\int_0^t f^\prime (B_s) \circ dB_s := 
\lim _{ n \te \infty} \sum_i \hf ( f^\prime(B_{t_{n,i}}) + f^\prime(B_{t_{n,i+1}}) ) 
( B_{t_{n,i+1} } - B_{t_{n,i}} ) .
\en
This construction has the advantage that it is better connected to geometric 
notions of integration, such as integration of a differential form along a continuous path
\cite{MR80d:60075}
\cite{MR81h:58061}
\cite{MR84e:60084}
and there is the simple formula
\eq
\lb{pre-strat}
f(B_t) - f(B_0) = \int_0^t f^\prime(B_s) \circ dB_s  .
\en
However, the important martingale property of \Ito\ integrals is hidden by the
Stratonovich construction. 
See \cite[Chapter 3]{oks03} and
\cite[Chapter 8]{MR1980149} 
for further comparison of \Ito\ and Stratonovich integrals.

\subsection{Stochastic integration}
The theory of stochastic integration defines integration of 
suitable random integrands $f(t,\omega)$ with respect to
random ``measures'' $dX_t(\omega)$ derived from suitable
stochastic processes $(X_t)$. The principal definitions in this 
theory, of \emph{local martingales} and \emph{semimartingales}, are motivated 
by a powerful calculus, known as \emph{stochastic} or \emph{\Ito}\ calculus, 
which allows the representation of various functionals of such processes 
as stochastic integrals.
For instance, the formula \re{pre-ito1} can be justified for $f(x) = x^2$ to identify
the martingale $B_t^2 - t$ as a stochastic integral:
\eq
\lb{ito1}
B_t ^2 - B_0^2 - t = 2 \int_0^t B_s dB_s .
\en
Similarly, the previous derivation of formula \re{dd3} is easily extended to a
Brownian motion $B$ in $\reals^\delta$ for $\delta = 1,2, 3, \ldots$ to show that
for $f \in C^2(\reals^\delta)$
\eq
\lb{pre-ito1}
f(B_t) - f(B_0) = \int_0^t (\nabla f)(B_s) \cdot dB_s + \hf \int_0^t  (\Delta f)(B_s) ds
\en
with $\nabla$ the gradient operator and $\Delta$ the Laplacian. Again, the stochastic
integral is obtained as a limit in probability of Riemann sums, along with the \Ito\ formula,
and the stochastic integral is a martingale in $t$.

It is instructive to study carefully what happens in \Ito's formula \re{pre-ito1} for $\delta \ge 2$
if we take $f$ to be a radial harmonic function with a pole at $0$, say
$$
f(x ) = \log | x |   \mbox{ if } \delta = 2
$$
and
$$
f(x)  = |x|^{2 - d} \mbox{ if } \delta \ge 3
$$
these functions being solutions on $\reals^\delta - \{0\}$ of
\emph{Laplace's equation}  $\Delta f = 0$, so the last term in \re{pre-ito1} vanishes. 
Provided $B_0 = x \ne 0$ the remaining stochastic integral is a well-defined
almost sure limit of Riemann sums, and moreover $f(B_t)$ is integrable, and even square integrable for
$\delta \ge 3$. It is tempting to jump to the conclusion that $(f(B_t), t \ge 0 )$ is a martingale. 
But this is not the case.  Indeed, it is quite easy to compute $\ER^x f(B_t)$ in these
examples, and to check for example that this function of $t$ is strictly decreasing for $\delta \ge 3$.
For $\delta = 3$, according to a relation discussed further in Section \ref{sec:brmeex},
$\ER^x (1/|B_t|)$ equals the probability
that a one-dimensional Brownian motion started at $|x|$ has not visited $0$ before time $t$.
The process $(f(B_t), t \ge 0 )$ in these examples is not a martingale but rather a \emph{local martingale}.

Let $X$ be a real-valued process, and
assume for simplicity that $X$ has continous paths and $X_0 = x_0$ for some 
fixed $x_0$. Such a process $X$ is called a \emph{local martingale} relative
to a filtration $(\FF_t)$ if for each $n = 1,2, \ldots$, the 
stopped process
$$
(X_{t \wedge T_n (\omega)} (\omega), t \ge 0 , \omega \in \Omega)
\mbox{ is an } ( \FF_t) \mbox{ martingale }
$$
for some sequence of stopping times $T_n$ increasing to $\infty$, which 
can be taken without loss of generality to be
$
T_n : = \inf \{ t : |X_t| > n \}.
$
For any of the processes $(f(B_t), t \ge 0)$ considered above for
a harmonic function $f$ with a pole at $0$, these processes stopped when they
first hit $\pm n$ are martingales, by consideration of \Ito's formula \re{pre-ito1}
for a $C^2$ function $\hat {f}$ which agrees with $f$ where $f$ has values in
$[-n,n]$, and is modified elsewhere to be $C^2$ on all of $\reals^\delta$.
Consideration of these martingales obtained by stopping processes is very useful, because by
application of the optional sampling theorem they
immediately yield formulae for hitting probabilities of the radial part of $B$ in $\reals^\delta$,
as discussed in \cite[I.18]{rw1}. 
A {\em continuous semimartingale} $X$ is the sum of a continuous local martingale and a process
with continous paths of locally bounded variation.

Given a filtration $(\FF_t)$,
a process $H$ of the form
$$
H_s(\omega) := \sum_{i = 1}^k H_i (\omega)1[ T_i(\omega) < s \le T_{i+1}(\omega)]
$$
for an increasing sequence of stopping times $T_i$, and 
$H_i$ an $\FF_{T_i}$ measurable random variable, is called
an \emph{ elementary predictable process}. If $B$ is an
$(\FF_t)$ Brownian motion, and
$H$ is such an elementary predictable process,
one can define
\eq
\lb{hdef}
\int_0^t H_s d B_s  := \sum_{i = 1}^k H_i ( B _{t \wedge T_{i +1} } - B _{t \wedge T_i } )
\en
and check the identity
\eq
\lb{hid}
\ER \left[ \left( \int_0^t H_s d B_s \right) ^2 \right] = 
\ER \left[ \int_0^t H_s^2 ds \right] 
\en
which allows the definition \re{hdef} to be extended by completion in $L^2$ 
to any pointwise limit $H$ of elementary predictable processes such that
\eq
\lb{hid2}
\ER \left[ \int_0^t H_s^2 ds \right]  < \infty
\en
for each $t >0$, and the identity \re{hid} then holds for such a limit process
$H$.
Replacing $H(s,\omega)$ by $H(s,\omega)1( s \le T(\omega))$, it follows
that both
\eq
\lb{hid3}
\int_0^t H_s dB_s \mbox{ and }
\left( \int_0^t H_s dB_s \right)^2 -
\int_0^t H_s ^2 ds 
\en
define $(\FF_t)$ martingales.

A similar stochastic integral, with $B$ replaced by an $(\FF_t)$ martingale
or even a local martingale $M$, is obtained hand in hand with the existence
of an increasing process $\brk{M}$ such that
$$
(M_t^2 - \brk{M}_t, t \ge 0 ) \mbox{ is an $(\FF_t)$ martingale}
$$
and the above discussion \re{hdef} -\re{hid} -\re{hid2} -\re{hid3} generalizes
straightforwardly with $dM_s$ instead of $dB_s$ and $d \brk{M}_s$ instead
of $ds$.
In particular, there is then the formula
$$
M_t^2 - M_0^2  - \brk{M}_t = 2 \int_0^t M_s dM_s 
$$
of which \re{ito1} is the special case for $M = B$.
See e.g. Yor \cite{MR689534} 
and Lenglart \cite{MR719513} 
for details of this approach to \Ito's formula for continuous semimartingales.

Let $M$ be a continuous  $(\FF_t)$ local martingale and $A$ an $(\FF_t)$
adapted continuous process of locally bounded variation. Then for suitably
regular functions $f = f(m,a)$ there is the following form of
\Ito's formula for semimartingales:
\eq\lb{ito1}
f(M_t,A_t) - f(M_0,A_0) = 
\int_0^t f_m ^\prime (M_s, A_s) d M_s 
+ \int_0^t f_a ^\prime (M_s, A_s) dA_s 
+ {1 \over 2} 
\int_0^t f^{\prime\prime} _{m,m} 
(M_s,A_s) d \brk{M}_s
\en
where 
$$
f_{m}^\prime (m,a):= {\partial \over \partial m } f(m,a); ~~~~
f_{a}^\prime (m,a):= {\partial \over \partial a } f(m,a);
f_{m,m}^{\prime \prime} (m,a):= {\partial \over \partial m } {\partial \over \partial m } f(m,a).
$$
Note that the first integral on the right side of \re{ito1} defines
a local martingale, and that the sum of the second and third integrals
is a process of locally bounded variation. It is the third integral,
involving the second derivative $f_{m,m}$, which is the special
feature of \Ito\ calculus.

More generally,  for a vector of $d$ local martingales
$M = (M^{(i)}, 1 \le i \le d )$, and a process $A$ of
locally bounded variation, \Ito's formula reads
$$
f(M_t ,A_t) - 
f(M_0 ,A_0) 
= \int_0^t \nabla_m f(M_s,A_s ) \cdot dM_s
+ \int_0^t \nabla_a f(M_s,A_s ) \cdot dA_s
$$
$$
+ {1 \over 2 } \int_0^t \sum_{i,j = 1}^d f_{m_i,m_j}^{\prime \prime }( M_s, A_s ) d \brk{ M^{(i)}, M^{(j)} } _s
$$
where for two local martingales $M^{(i)}$ and $M^{(j)}$, their {\em bracket}
$\brk{ M^{(i)}, M^{(j)}} $ is the unique continuous process $C$ with
bounded variation such that $M^{(i)}_t M^{(j)}_t - C_t$ is a local martingale.
See Section \ref{sec:pdes} for connections between \Ito's formula and
various second order partial differential equations.

\subsection{Construction of Markov processes}

\subsubsection{Stochastic differential equations}

One of \Ito's fundamental insights was that the theory of stochastic integration could
be used to construct a diffusion process $X$ in $\reals^N$ as the solution
of a stochastic differential equation
\eq\lb{itosde}
dX_t = \sigma(X_t) dB_t + b(X_t) dt
\en
for $\sigma : \reals^N \to \matrices^{N \times N}$ a field of matrices and
$b : \reals^N \to \reals^{N}$ a vector field. More formally, the meaning of
\re{itosde} with initial condition $X_0 = x \in \reals^N$ is that
\eq\lb{itosde1}
X_t = x + \int_0^t \sigma(X_s) dB_s + \int_0^t b(X_s) ds .
\en
It is known that under the hypothesis that $\sigma$ and $b$ are Lipschitz, 
the equation \re{itosde1} has a \emph{unique strong solution}, meaning that for a given
Brownian motion $B$ and initial point $x$ the path of $X$ is uniquely determined  almost surely 
for all $t \ge 0$. Moreover, this solution is obained as the limit of the classical Picard
iteration procedure, the process $X$ is adapted to the filtration  $(\FF_t)$ generated by $B$,
and the family of laws of $X$, indexed by the initial point $x$, defines a 
Markov process of Feller-Dynkin type with state space $\reals^N$.
See \cite{ry99} for details.
For a given Brownian motion $B$, one can consider the dependence of
the solution $X_t$ in \re{itosde1} in the initial state $x$, say
$X_t = X_t (x)$. Under suitable regularity conditions the map $x \te X_t(x)$
defines a random diffeomorphism from $\reals^N$ to $\reals^N$. 
This leads to the notion of a \emph{Brownian flow of diffeomorphisms}
as presented by Kunita \cite{MR91m:60107}.

For $f \in C_b^2$, \Ito's formula applied to $f(X_t)$ shows that
$$
M_t^f:= f(X_t) - f(X_0) - \int_0^t \gen f (X_s) ds \mbox{ is an } (\FF_t) \mbox{ martingale,}
$$
where
\eq
\lb{gendag}
\gen  f(x)  := \hf 
\sum_{i,j} a^{ij} (x) 
{ \partial ^2 f \over \partial x_i \partial x_j } (x)
+
\sum_{i} b_i(x) (x) 
{ \partial f \over \partial x_i } (x)
\en
with $a(x):= (\sigma^T \sigma)(x)$.
Thus, the infinitesimal generator of $X$, restricted to $C^2_b$, is the elliptic operator
$\gen$ defined by \re{gendag}.
To see the probabilistic meaning of the coefficients $a^{ij}(x)$, observe that
if we write $X_t = (X_t^i, 1 \le i \le N)$ and $M_t^i$ instead of $M_t^f$ for $f(x) = x_i$,
so
$$
M_t^i = X_t^i - X_0^i - \int_0^t b_i(X_s) ds
$$
then
$$
\langle M^i, M^j \rangle_t = \int_0^t a^{ij} (X_s ) ds
$$
and more generally
$$
\langle M^f, M^g \rangle_t = \int_0^t ds \Gamma (f,g) (X_s) ds
$$
where
\eq\lb{sqfield}
\Gamma(f,g) (x) := \nabla f (x) \cdot ( a(x) \nabla g (x) ) = 
\gen( f g ) (x ) - f(x) \gen g (x) - g(x) \gen f(x)
\en
is the \emph{square field operator} (\emph{\ocdc}).
In particular, $a^{ij}(x) = \Gamma(x_i,x_j)(x)$. 
For a Markov process $X$ on a more general
state space, 
Kunita \cite{MR91m:60107}
takes a basis of functions $(u_i)$ with respect to which
he considers an infinitesimal generator of the
form
$$
\gen f = \hf \sum_{i,j} \Gamma( u_i, u_j ) (x) { \partial^2 f \over \partial u_i \partial u_j}(x) + \cdots
$$
with $\cdots$  a sum of drift terms of first order and integral terms related to jumps.
See Bouleau and Hirsch \cite{MR93e:60107} for further developments.

\subsubsection{One-dimensional diffusions}
\label{sec:onedimdiff}
The Bessel processes defined as the radial parts of Brownian motion in $\reals^\delta$ 
are examples of {\em one-dimensional diffusions}, that is to say strong Markov processes with continuous
paths whose state space is a subinterval of $\reals$. Such processes have been extensively studied
by a number of approaches: through their transition densities as solutions of a suitable
parabolic differential equation, by space and time changes of Brownian motion, and as solutions of
stochastic differential equations. 
For instance, an {\em Ornstein-Uhlenbeck process} $X$ may be defined by the stochastic differential equation (SDE)
$$
X_0 = x;  ~~~ dX_t = d B_t + \lambda X_t dt
$$
for $x \in \reals$ and a constant $\lambda >0$. This SDE is taken to mean
$$
X_t = x + B_t + \lambda \int_0^t X_s ds
$$
which is one of the rare SDE's which can be solved explicitly:
$$
X_t = e^{\lambda t } \left( x + \int_0^t e^{- \lambda s } d B_s  \right) .
$$
The result is a Gaussian Markov process which admits a number of alternative representations.
See Nelson \cite{MR35:5001} for the physical motivation and background. 

Following is a list of texts on one-dimensional diffusions:

\begbib
\item[\cite{fr83bd}] D.~Freedman. \newblock {\em Brownian motion and diffusion} (1983).    
\item[\cite{MR33:8031}] K. It{\^o} and H.~P. McKean, Jr. \newblock {\em Diffusion processes and their sample paths} (1965).
\item[\cite{MR0247667}] P. Mandl. {\em Analytical treatment of one-dimensional Markov processes} (1968).
\item[\cite{rw1}] and \cite{rw2} L.~C.~G. Rogers and D.~Williams. \newblock {\em Diffusions, Markov Processes and Martingales, Vols. 1 and 2} (1994).
\endbib

See also 
\cite{borsal02} for an extensive table of Laplace transforms of functionals one-dimensional diffusions, including Brownian motion, Bessel processes and
the Ornstein-Uhlenbeck process.

\subsubsection{Martingale problems}

Kunita \cite{kunita69} used \re{genfg1} to define the 
\emph{extended infinitesimal generator $\gen$} of a Markov process $X$ 
by the correspondence between pairs of bounded Borel functions $f$ and
$g$ such that \re{genfg1} holds. 
This leads to the idea of defining the family of probability measures
$P^x$ governing a Markov process $X$ via the \emph{martingale problem}
of finding $\{\PR^x\}$ such that \re{genfg1} holds whenever 
$\gen f = g$ for some prescribed infinitesimal generator $\gen$ such
as \re{bgen1}.
This program was carried out for diffusion processes 
by Stroock and Varadhan \cite{MR81f:60108}.
See  also \cite[\S III.13]{rw1}, \cite[Chapter VII]{ry99}. 
Komatsu \cite{komatsu73} and Stroock \cite{MR55:6587} treat the case of Markov processes 
with jumps.

\subsubsection{Dirichlet forms}

The \sqf\ $\Gamma$ introduced in \re{sqfield} leads to naturally to consideration of 
the \emph{Dirichlet form}
$$
\epsilon_{\mu}(f,g) := \int \mu(dx) \Gamma(f,g) (x)
$$
where $\mu$ is an invariant measure for the Markov process. In the case of Brownian motion
on $\reals^N$,
the Dirichlet form is
$$
\int_ dx \nabla f (x) \cdot \nabla g(x)
$$
where $dx$ is Lebesgue measure.
A key point is that this operator on pairs of functions $f$ and $g$,
which makes sense for $f$ and $g$ which may not be twice differentiable, can be used to characterize BM.
See the following texts for development of this idea, and the general notion of a \emph{Dirichlet process} 
which can be built from such an operator.

\begbib
\item[\cite{MR96f:60126}] M. Fukushima, Y. {\=O}shima, and M. Takeda. \newblock {\em Dirichlet forms and symmetric {M}arkov processes}, (1994)
\item[\cite{MR95c:31001}]E.~Fabes, M.~Fukushima, L.~Gross, C.~Kenig, M.~R{\"o}ckner, and D.~W. Stroock. \newblock {\em Dirichlet forms}, (1993).
\item[\cite{MR81f:60105}] M. Fukushima. \newblock {\em Dirichlet forms and {M}arkov processes}, (1980).
\item[\cite{MR93e:60107}]N. Bouleau and F. Hirsch. \newblock {\em Dirichlet forms and analysis on {W}iener space}, (1991).
\item[\cite{MR96f:31001}]Z.~M. Ma, M.~R{\"o}ckner, and J.~A. Yan, editors. \newblock {\em Dirichlet forms and stochastic processes}. (1995).
\item[\cite{MR94d:60119}]Z. ~M. Ma and M. R{\"o}ckner. \newblock {\em Introduction to the theory of (nonsymmetric) {D}irichlet forms} (1992).
\item[\cite{MR99e:31002}] J. Jost, W. Kendall, U. Mosco, M. R{\"o}ckner and   K.-T. Sturm. \newblock {\em New directions in {D}irichlet forms}, (1998).
\endbib

\subsection{Brownian martingales}
\subsubsection{Representation as stochastic integrals}
A (local) martingale relative to the natural filtration
$(\BB_t, t \ge 0)$ of a $d$-dimensional Brownian motion
$$
( B_t^{(1)}, \ldots, B_t^{(d) }); t \ge 0 )
$$
is called a {\em Brownian (local) martingale }.
According to an important result of 
\Ito 
and Kunita- Watanabe \cite{MR0217856},
every Brownian local martingale admits a continuous version
$(M_t, t \ge 0 )$ which may be written as
\eq
\lb{mg1}
M_t = c + \int_0 ^t m_s \cdot d B_s
\en
for some constant $c$ and some $\reals^d$-valued predictable process
$(m_s, s \ge 0)$ such that $\int_0^t |m_s|^2 ds < \infty$.

In particular, every $L^2( \BB_\infty)$ random variable $Y$ may be represented
as
\eq
\lb{mg2}
Y =  \ER (Y) + \int_0^\infty y_s \cdot dB_s
\en
for some $\reals^d$-valued predictable process $(y_s, s \ge 0)$ such that 
\eq
\lb{mg3}
\ER \left[ \int_0^\infty |y_s|^2 ds \right] < \infty .
\en
Such a representing process is unique in $L^2( \Omega \times \reals_+, \PR(\BB),
d \PR \, d s  )$.
The Clark-Ocone formula 
\cite{MR0270448} 
\cite{MR749372} 
gives some general expression for
$(y_s, s \ge 0)$ in terms of $Y$. This expression plays an important
role in Malliavin calculus. See references in Section \ref{sec:malliavin}.
    
A class of examples of particular interest arises when
$$
E[ Y \giv \BB_t ] = \Phi(t,\omega; B_t(\omega) )
$$
for suitably regular $\Phi(t,\omega,x)$. In particular, if $\Phi$
is of bounded variation in $t$, and sufficiently smooth in $x$,
one of Kunita's extensions of \Ito's formula gives
\eq
\lb{mg4}
E[ Y \giv \BB_t ] = E(Y) + \int_0^t \nabla_x  \Phi( s, \omega; B_s(\omega)) \boldsymbol{\cdot } d B_s 
\en
where $\nabla_x$ is the gradient operator with respect to $x$.
So, with the notations \re{mg1} and \re{mg2}, we get,
for the $L^2$-martingale $M_t = E[ Y \giv \BB_t ]$,
\eq
\lb{mg5}
m_s = y_s = \nabla_x  \Phi( s, \omega; B_s(\omega))
\en
See \cite{MR2015458,MR2324166} 
for some interesting examples of such computations.

\subsubsection{Wiener chaos decomposition}
\label{sec:chaos}
These representation results \re{mg1} and \re{mg2} may also be
regarded as corollaries of the 
\emph{Wiener chaos decomposition} of $L^2(\BB_\infty)$ as
\eq
\lb{mg5}
L_2(\BB_\infty) = \bigoplus_{n = 0}^\infty C_n
\en
where $C_n$ is the subspace of $L^2(\BB_\infty)$ spanned by
$n$th order multiple integrals of the form
$$
\int_0^\infty dB_{t_1}^{(i_1)} \int_0^{t_1} d B_{t_2}^{(i_2)} \cdots \int_0^{t_{n-1}} d B_{t_{n}}^{(i_n)} f_n(t_1, \ldots, t_n)
$$
for $f_n$ subject to
$$
\int_{0 \le t_n \le t_{n-1} \le \cdots \le t_1} dt_1 \cdots dt_n \, f_n^2(t_1, \ldots, t_n ) < \infty
$$
and $1 \le i_j \le d$ for $1 \le j \le n$.
This space $C_n$, consisting of iterated integrals obtained from
{\em deterministic} functions $f_n$, is called the $n$th Wiener chaos. 

To prove 
the martingale representation \re{mg1}-\re{mg2}, it suffices to establish
\re{mg2} for the random variable
$$
Y = \exp \left\{ \int_{0}^\infty f(u) \cdot d B_u - \hf \int_{0}^\infty | f(u)|^2 du \right\} .
$$
for $f \in L^2( \reals_+ \to \reals^d; du )$.
The formula \re{mg2} now follows from \Ito's formula, with
$$
y_s = f(s) \,  \exp \left\{ \int_{0}^s  f(u) \cdot d B_u - \hf \int_{0}^s | f(u)|^2 du \right\} .
$$
Similarly, the Wiener chaos representation \re{mg5} follows by consideration
of the generating function 
$$
\exp ( \lambda x - \hf \lambda^2 u ) = \sum_{n = 0}^ \infty { \lambda^n \over n!} H_n(x,u)
$$
of the Hermite polynomials $H_n(x,u)$, using the consequence of \Ito's
formula and $( \partial /\partial x ) H_n = H_{n-1}$ 
that
$$
H_n \left ( \int_0^t f(s) \cdot d B_s , \int_0^t | f(s)| ^2 ds \right)
=
\int_0^t H_{n-1} \left ( \int_0^s f(u) \cdot d B_u , \int_0^s | f(u)| ^2 du \right)
\, f(s) \cdot dB_s .
$$
We discuss in Section \ref{sec:quad} some techniques for
identifying the distribution of Brownian functionals in $C_0 \bigoplus C_2$.

It is known that if $X \in \bigoplus_{k = 0}^n C_k$ then there
exists $\alpha >0$ such that
$$
\ER \left[ \exp ( \alpha |X| ^{2/n} ) \right] < \infty
$$
and also some $\alpha_0$ such that for any $\beta > \alpha _0$
$$
\ER \left[ \exp ( \alpha |X| ^{2/n} ) \right]  = \infty
$$
assuming that $X = X_0 + \cdots + X_n$ with $X_i \in C_i$ and $X_n \ne 0$.
This gives some indication of the tail behaviour of distributions of
various Brownian functionals
Such results may be found in the book of
Ledoux and Talagrand  \cite{MR1102015}. 
We do not know of any exact computation of the law of
a non-degenerate  element of $C_3$, or of 
$C_0 \bigoplus C_1 \bigoplus C_2 \bigoplus C_3$.
We regard as ``degenerate" a variable such as
$(\int_0^\infty f(s) dB_s)^3$, whose law can be found by
simple transformation of the law of some element of 
$C_0 \bigoplus C_1 \bigoplus C_2$.  


\subsection{Transformations of Brownian motion}
In this section, we examine how a BM $(B_t, t \ge 0)$ is affected by
the following sorts of changes:
\begin{description}
\item[Locally absolutely continuous change of probability:]  The background probability 
law $\PR$ is modified by some density factor $D_t$ on the $\sigma$-field $\FF_t$ of
events determined by $B$ up to time $t$, to obtain a new probability law $\QR$.
\item[Enlargement of filtration:]  The background filtration, with respect to which $B$
is a Brownian motion, is enlarged in some way which affects the description of $B$ as
a semimartingale.
\item[Time change:] The time parameter $t \ge 0$ is replaced by some increasing family of
stopping times $(\tau_u, u \ge 0 )$.
\end{description}

The scope of this discussion can be expanded in many ways, to include e.g. the transformation
induced by a stochastic differential equation, or space-time transformations, scale/speed description of a diffusion,
reflection, killing, \Lev's transformation, and so on.
One effect of such transformations is that simple functionals of the transformed process are just
more complex functionals of BM.  This has provided motivation for the study of more and more complex functionals of BM.


\subsubsection{Change of probability}
\label{chprob}
The assumption is that the underlying probability $\PR$ is replaced by $\QR$ defined on
by
\label{chprob}
\eq
\lb{qdef}
\QR \eval _{\FF_t} = D_t \cdot \PR \eval _{ \FF_t }
\en
meaning that every non negative $\FF_t$-measurable trandom variable $X_t$
has $\QR$-expectation
$$
\ER^{\QR} X _t := \ER ( D_t X _t) .
$$
This definition is consistent as $t$-varies, and defines a probability distribution on
the entire path space, if and only if $(D_t, t \ge 0 )$ is an $(\FF_t, \PR )$ martingale. 
Then, Girsanov's theorem \cite[Ch. VIII]{ry99} states that
\eq
\lb{girs}
B_t = \tilde{B}_t + \int_0^t { d \langle D, B \rangle _s \over D_s }  ,
\en
with $(\tilde{B}_t )$ an $( (\FF_t), \QR )$ Brownian motion. In the first instance,
$(\tilde{B}_t )$ is just identified as an $( (\FF_t), \QR )$ local martingale.
But $\brk{ \tilde{B} } = \brk{B}_t =  t$, and hence
$(\tilde{B}_t )$ is an $( (\FF_t), \QR )$ Brownian motion, by \Lev's theorem.

This application of Girsanov's theorem has a number of important consequences for 
Brownian motion.
In particular, for each $f \in 
L_{\rm loc}^2 ( \reals_+, ds )
$ the law $\QR^{(f)}$ 
of the process
$$
\left( B_t + \int_0^t f(s) ds , t \ge 0 \right)
\ed \left( \tilde{B}_t + \int_0^t f(s) ds , t \ge 0 \right)
$$
is locally equivalent to the law $\PR$ of BM,
with the density relation
$$
\QR ^{(f)}  \eval _{\FF_t } =  D_t^{(f)}
\cdot \PR \eval _{ \FF_t } .
$$
where the density factor is
$$
D_t^{(f)} = \exp \left( \int_0^t f(s) d B_s - \hf \int_0^t f^2(s) ds   \right) = 1 + \int_0^t D_s^{(f)} f(s) dB_s .
$$
In other words, the Wiener measure $P$ is quasi-invariant under translations by functions
$F$ in the \emph{Cameron-Martin space}, that is 
$$
F(t) = \int_0^t f(s) ds \mbox{ for } f \in L_{\rm loc}^2 ( \reals_+, ds ) .
$$
As a typical application of the general Girsanov formula \re{girs}, the law $\PR _x^{\lambda}$ of the Ornstein-Uhlenbeck process of
Section \ref{ou} is found to satisfy 
$$
P_x^{\lambda} \eval _{\FF_t }
= \exp \left\{ {\lambda \over 2 } (B_t^2 - x^2 )  - {\lambda^2 \over 2 } \int_ 0^t B_s^2 ds \right\}
\cdot \PR^x \eval _{\FF_t}
$$
where the formula $\hf ( B_t^2 - x^2 ) = \int_0^t B_s d B_s$ has been used.

Girsanov's formula can also be applied to study the bridge of 
length $T$ defined by starting a diffusion process $X$ started at some
point $x$ at time $0$, and conditioning on arrival at $y$ at time $T$.
Then, more or less by definition 
\cite{MR1278079}, 
for $ 0 < s < T$
\eq
\lb{xbridge}
\ER^x [ F( X_u , 0 \le u \le s ) \giv X_T = y ] =
\ER^x \left[  F( X_u , 0 \le u \le s ) { p_{T-s}(X_s,y) \over p_T(x,y)} \right]
\en
where $p_t(x,y)$ is the transition density for the diffusion.
In particular, for $X$ a Brownian bridge of length $T$ from $(0,x)$
to $(T,y)$ we learn from Girsanov's formula that
\eq
\lb{xbridge1}
X_s = x + \beta_s + \int_0^s du \, { (y - X_u ) \over ( T - u ) }
\en
This discussion generalizes easily to a $d$-dimensional Brownian motion,
and to other Markov processes.
See e.g. \cite[\S 6.2.2]{MR1980149}. 
\subsubsection{Change of filtration}
\label{chfilt}
Consider now the description of an $(\FF_t)$ Brownian motion $(B_t)$ relative to some
larger filtration $(\GG_t)$, meaning that $\GG_t \supseteq \FF_t$ for each $t$.
Provided $\GG_t$ does not import too much information relative to $\FF_t$, the Brownian
motion $(B_t)$, and more generally every
$(\FF_t)$ martingale, will remain a $(\GG_t)$ semimartingale, or at worst a $(\GG_t)$ 
\emph{Dirichlet process}, 
meaning the sum of a $(\GG_t)$ martingale and a process with zero quadratic variation.
In particular, such an enlargement of filtration allows the original Brownian motion
$B$ to be decomposed as
$$
B_t = \tilde{B}_t + A_t
$$
where $(\tilde{B}_t)$ is (again by \Lev's characterization) a $(\GG_t)$ Brownian motion,
and $(A_t)$ has zero quadratic variation.
As an example, if we enlarge the filtration $(\FF_t)$ generated by $(B_t)$ to
$\GG_t$ generated by $\FF_t$ and $\int_0^\infty f(s) dB_s$ for some $f \in L^2( \reals_+, ds)$,
then 
\eq
\lb{star}
B_t = \tilde{B}_t + \int_0^t { ds \, f(s) \int_s^\infty f(u) d B_u \over \int_s^\infty f^2(u) du  }
\en
where $(\tilde{B}_t)$ is independent of the sigma-field $\GG_0$ of events generated by
$\int_0^\infty f(u) d B_u$.
The best known example arises when $f(s) = 1( 0 \le s \le T)$, so $\GG_t$ is 
generated by $\FF_t$ and $B_T$. Then \re{star} reduces to
$$
B_t = \tilde{B}_t + \int_0^{t \wedge T}  { ds \, ( B_T - B_s ) \over (T-s ) } .
$$
Since $(\tilde{B}_t)$ and $B_T$ are independent, we can condition on $B_T = y$ to deduce that
\newcommand{\bridge}{B^{\rm br}}
the Brownian bridge $(\bridge_t)$ 
of length $T$ from $(0,0)$ to $(T,x)$ can be related to an unconditioned
Brownian motion $\tilde{B}$ by the equation
$$
\bridge_t = \tilde{B}_t + \int_0^{t \wedge T}  { ds \, ( x - \bridge_s ) \over (T-s ) } .
$$
which can be solved explicitly to give
$$
\bridge_t = (t/T) y + ( T - t ) \int_0^{t}  { d \tilde{B}_s \over (T-s ) } 
$$
or again, by time-changing
$$
\bridge_t = (t/T) y + ( T - t ) \beta_{t/( T (T-t))}
$$
for another Brownian motion $\beta$.
Compare with Section \ref{sec:bridges}.

Other enlargements $(\GG_t)$ of the original Brownian filtration
$(\FF_t)$ can be obtained by turning some particular random times $L$
into $(\GG_t)$ stopping times, so $\GG_t$ is the $\sigma$-field
generated by $\FF_t$ and the random variable $L \wedge t$.
If $L = \sup \{ t: (t,\omega) \in A \}$ for some $(\FF_t)$ predictable
set $A$, then there is the decomposition
$$
B_t = \tilde{B}_t + 
\int_0^{t \wedge L } 
{ 
d \langle B, Z^L \rangle _s 
\over 
Z^L  _s 
}
+ 
\int_L^{t} 
{ 
d \langle B, 1 - Z^L \rangle _s 
\over 
1 - Z^L _s 
}
$$
where 
$Z^L  _t:= \PR( L >t \giv \FF_t)$ and 
$(\tilde{B}_t)$ is a $(\GG_t)$ Brownian motion.

The volume 
\cite{MR884713} 
provides many applications of the theory of enlargement of filtrations,
in particular to provide explanations in terms of stochastic calculus to path decompositions of 
Brownian motion at last exit times and minimum times.
See also \cite{MR604176}. 
The article of Jacod \cite{jacod1985grossissement} treats the problem of \emph{initial enlargement}
from $(\FF_t)$ to $( \GG_t)$ with $\GG_t$ the $\sigma$-field generated by $\FF_t$ and $Z$
for some random variable $Z$ whose value is supposed to be known at time $0$.
See also 
\cite[Ch. 6]{MR2273672}, 
\cite[Ch. 12]{MR1442263}, 
and \cite{MR3720124} 
for a recent overview.
On the other hand, the theory of progressive enlargements has developed very little since 1985.

\subsubsection{Change of time}
\label{sec:chtime}
If the time parameter $t \ge 0$ is replaced by some increasing and
right continuous family of stopping times $(\tau_u, u \ge 0 )$, then 
according to the general theory of semimartingales 
we obtain from Brownian motion $B$ a process $(B_{\tau_u}, u \ge 0 )$ which is
a semimartingale relative to the filtration $(\FF_{\tau_u}, u \ge 0 )$.
In particular, it is follows from the Burkholder-Davis-Gundy inequalities
\cite[\S IV.4]{ry99}
that if $\ER ( \sqrt{ \tau_u } ) < \infty $ then
$(B_{\tau_u}, u \ge 0 )$ is a martingale.
%
Monroe 
\cite{MR0343354} 
showed that every semimartingale can be obtained, in distribution, as $(B_{\tau_u}, u \ge 0 )$ for a suitable
time change process $(\tau_u)$.


A beautiful application of \Lev's characterization of BM
is the representation of continuous martingales as
time-changed Brownian motions. Here is the precise statement.
Let $(M_t , t \ge 0)$ be a $d$-dimensional continuous local martingale
relative to some filtration $(\FF_t)$, such that
\begin{itemize}
\item[(i)] $\langle M^{(i)} \rangle_t =  A_t $ for some increasing process $(A_t)$ with
$A_\infty = \infty$, and all $i$.
\item[(ii)] $\langle M^{(i)} , M^{(j)} \rangle _t \equiv 0 $,
which is to say that the product 
$M^{(i)}_t  M^{(j)}_t$ is an $(\FF_t)$ local martingale,
for all $i \ne j$. 
\end{itemize}
Let 
$$
\tau_t := \inf \{s : A_s > t \} \mbox { and } B_t := M_{\tau_t}  .
$$
Then the process  $(B_t)$ is an $(\FF_{\tau(t)})$ Brownian motion,
and 
\eq
\lb{starstar}
M_u = B_{A_u}  ~~~~~(u \ge 0).
\en
Doeblin in 1940 discovered the instance of this result 
for $d = 1$ and $M_t = f(X_t) - \int_0^t (Lf)(X_s) ds$ for $X$ a 
one-dimensional diffusion and $f$ a function in the domain of 
the infinitesimal generator $L$ of $X$. See \cite[p. 20]{MR1885582}. 
Dambis \cite{dambis65} and Dubins-Schwarz \cite{MR31:2756} gave the general
result for $d=1$, while 
Getoor and Sharpe 
\cite{MR0305473} 
formulated
it for $d = 2$, as discussed in the next subsection.

As a simple application of \re{starstar}, we mention the following:
if $(M_u, u \ge 0 )$ is a non-negative local martingale, such that
$$
M_0 = a \mbox{ and } \lim_{u \to \infty } M_u = 0
$$
then 
\eq
\lb{dmax}
\PR \left( \sup_{u \ge 0} M_u \ge x \right) = a/x        ~~~~~(x \ge a ).
\en
To prove \re{dmax}, it suffices thanks to \re{starstar} to check it
for $M$ a Brownian motion started at $a$ and stopped at its first
hitting time of $0$. The conclusion \re{dmax} can also be deduced quite easily by optional sampling.
See \cite{MR2234288} for further results in this vein. 

\subsubsection{Knight's theorem}
\label{sec:ktthm}
A more general result  on time-changes is obtained by consideration of
a $d$-dimensional continuous local martingale $(M_t , t \ge 0)$ 
relative to some filtration $(\FF_t)$, such that
$$\langle M^{(i)} , M^{(j)} \rangle _t \equiv 0 $$
and each of the processes $\langle M^{(i)} \rangle _t $ grows to
infinity almost surely, but these increasing processes  are
not necessarily identical. Then, Knight's theorem \cite[Theorem V (1.9)]{ry99}
states that if $(B_u^{(i)})$ denotes the Brownian motion such that
$M_t^{(i)} = B^{(i)}_{ \langle M^{(i)} \rangle_t}$, then
the Brownian motions $B^{(1)}, \ldots, B^{(d)}$ are independent.

As an example, if $\Gamma^{(i)}$ for $1 \le i \le d$ are
$d$ disjoint Borel subsets of $\reals$, each with positive Lebesgue
measure, then 
$$
\int_{0}^t 
1( B_s \in \Gamma^{(i)}) 
dB_s  = B^{(i)} \left( \int_0^t 
1( B_s \in \Gamma^{(i)}) ds
\right)
$$
for some independent Brownian motions $B^{(i)}$.

\subsection{BM as a harness}

Another characterization of BM is obtained by considering the
conditional expectation of $B_u$ for some $u \in [s,t]$
conditionally given the path of $B_v$ for $v \notin (s,t)$. As a
consequence of the Markov property of $B$ and exchangeability of increments,
there is the basic formula
\eq \lb{hrn}
\ER [ B_u \giv B_v, v \notin (s,t) ] = { t - u \over t -s } B_s + {u-s \over t-s} B_t
~~~~~(0 \le s <  u  <  t )
\en
which just states that given the path of $B$ outside of $(s,t)$, the 
path of $B$ on $(s,t)$ is expected to follow the straight line from 
$(s,B_s)$ to $(t,B_t)$.
Following Hammersley 
\cite{MR0224144} 
a process $B$ 
with this property is called a {\em harness}.
D. Williams 
showed around 1980  
that every harness with continuous paths parameterized by $[0,\infty)$
may be represented as
$(\sigma B_s + \mu s , s \ge 0 )$ for $\sigma$ and $\mu$ two
random variables which are measurable with respect to the germ
$\sigma$-field
$$
\cap_{0 < s < t < \infty}  \sigma( B_v,  v \notin (s,t) ) .
$$
See also Jacod-Protter \cite{MR929066} 
who showed that every integrable \Lev\ process is a harness. 
Further discussion and references can be found in
\cite{MR2111197}. 

\subsection{Gaussian semi-martingales}

In general, to show that a given adapted, right continuous process is, or is not, a semimartingale, may be quite subtle. 
This question for Gaussian processes was studied by Jain and Monrad \cite{MR650607}, 
Stricker
\cite{MR85c:60054} 
\cite{MR87k:60107} 
and Emery \cite{MR688911}. 
Interesting examples of Gaussian processes which are not semimartingales are the fractional Brownian motion, for all values of their Hurst parameter $H$ except $1/2$ and $1$.
Many studies, including the development of adhoc stochastic integration, have been made for fractional Brownian motions. 
In particular, P. Cheridito  \cite{MR1873835} obtained the beautiful result that the addition of a fractional Brownian motion with Hurst parameter $H > 3/4$ and an independent Brownian motion produces a semimartingale.

\subsection{Generalizations of martingale calculus}

Stochastic calculus for processes with jumps:
Meyer's appendix \cite{meyer89si}, 
Meyer's course  \cite{mey76},
the books of Chung-Williams \cite{MR92d:60057} and Protter 
\cite{MR1675059}. 
Extension of stochastic calculus to {\em Dirichlet processes}, that is sums of
a martingale and a process of vanishing quadratic variation:
Bertoin \cite{MR941983}. 
Anticipative stochastic calculus of Skorokhod and others: some references 
are
\cite{MR1160401} 
\cite{MR920267} 
\cite{MR3308895}.  

\subsection{References}

\paragraph{Martingales}

Most modern texts on  probability and stochastic processes contain an
introduction at least to discrete time martingale theory. Some further references are:

\begbib

\item[\cite{MR56:6844}]A. ~M. Garsia. \newblock {\em Martingale inequalities: {S}eminar notes on recent progress}.  (1973)
\item[\cite{MR53:6728}] J. Neveu. \newblock {\em Martingales \`a temps discret}. (1972).
\item[\cite{MR93d:60002}] D. Williams. \newblock {\em Probability with martingales}. (1991) 
\endbib

\paragraph{Semi-Martingales}

Pioneering works:
\begbib
\item[\cite{MR49:4092}]J. Pellaumail. \newblock {\em Sur l'int\'egrale stochastique et la d\'ecomposition de   {D}oob-{M}eyer} (1973).

\item[\cite{MR58:7841}]A.~U. Kussmaul. \newblock {\em Stochastic integration and generalized martingales} (1977).
\endbib

The following papers present a definitive account of 
semi-martingales as ``good integrators''

\begbib
\item[\cite{MR82k:60122}]K. Bichteler. \newblock {\em Stochastic integrators} (1979).

\item[\cite{MR82g:60071}]K. Bichteler. \newblock {\em Stochastic integration and {$L^{p}$}-theory of semimartingales} (1981).

\item[\cite{MR83i:60069}]C.~Dellacherie. \newblock {\em Un survol de la th\'eorie de l'int\'egrale stochastique} (1980).

\endbib

Elementary treatments:

\begbib
\item[\cite{MR97k:60148}]R. Durrett. \newblock {\em Stochastic calculus: a practical introduction} (1996).
\item  [\cite{MR1783083}] J. M.  Steele. \newblock{\em Stochastic calculus and financial applications.} (2001)
\endbib

\paragraph{Stochastic integration: history}

Undoubtedly, the inventor of Stochastic Integration is K. It\^o,
although there were some predecessors: Paley and Wiener who
integrated deterministic functions against Brownian motion, and L\'evy
who tried to develop a stochastic integration framework by
randomizing the Darboux sums, etc... However, K. It\^o stochastic
integrals, which integrate, say, predictable processes against
Brownian motion proved to provide the right level of generality to
encompass a large number of applications. In particular, it led to
the definition and solution of stochastic differential equations,
for which in most cases, Picard's iteration procedure works, and
thus, probabilists were handed a pathwise construction of many
Markov processes, via It\^o construction. To appreciate the scope of
It\^o achievement, we should compare the general class of Markov processes 
obtained through his method with those which Feller obtained from Brownian motion
via time and space changes of variables. Feller's method works extremely well in one
dimension, but does not generalize easily to higher dimensions.
It\^o's construction was largely unappreciated until the publication of McKean's wonderful book \cite{MR40:947} 25 years after It\^o's original paper.
In 1967 Paul-Andr\'e Meyer expounded 
S\'eminaire de Probabilit\'es I \cite{MR0231445} 
the very
important paper of Kunita and Watanabe \cite{MR0217856} 
on the application of Stochastic
integration to the study of martingales associated with Markov processes. 
It\^o theory became better known in France in 1972, following a well attended
course by J. Neveu, which was unfortunately only recorded in handwritten form. 
The next step was taken by Paul-Andr\'e Meyer in his course 
\cite{MR54:14091}
where he blended the It\^o--Kunita--Watanabe development with the general theory of processes, 
to present stochastic integration with respect to general semi martingales. 
Jacod's Lecture Notes  
\cite{MR542115} 
built on the Strasbourg theory 
of predictable and dual predictable projections, and so forth, 
aiming at the description of all martingales with respect to the filtration of a given
process, such as a L\'evy process. 
A contemporary to Jacod's lecture notes is the book
of D. Stroock and S. Varadhan \cite{MR81f:60108},
where the martingale problem associated with an infinitesimal generator
is used to characterize and construct diffusion processes, thereby extending 
L\'evy's characterization of Brownian motion.

\begbib

\item[\cite{MR40:947}]H.~P. McKean, Jr. \newblock {\em Stochastic integrals} (1969).

\item[\cite{MR54:14091}]P.-A. Meyer. \newblock {\em Martingales and stochastic integrals. {I}. } (1972).

\item[\cite{MR58:18721}]P.~A. Meyer. \newblock {\em Un cours sur les int\'egrales stochastiques} (1976).

\item[\cite{MR92d:60057}]K.~L. Chung and R.~J. Williams. \newblock {\em Introduction to stochastic integration} (1990).

\item[\cite{meyer89si}]P.-A. Meyer. \newblock {\em A short presentation of stochastic calculus} (1989).

\endbib

\section{Brownian functionals}
\subsection{Hitting times and extremes}

For $x\in\mathbb{R}$, let $T_x: = \inf\{t:t\geq 0, B_t=x\}$.  
Then for $a,b>0$, 
by optional sampling,
$$\PR^0(T_a<T_{-b})=\frac{b}{a+b}$$
and hence
$$\PR^0(T_x<\infty)=1 \mbox{  for all } x\in\mathbb{R}.$$
Let 
$$M_t:= \max_{0\leq s\leq t} B_s$$ and notice that $(T_x,x\geq 0)$ is the left
continuous inverse for $(M_t,t\geq 0)$.  
Since $(M_t\geq x)=(T_x\leq t)$, if we know the distribution of $M_t$ for all
${t>0}$ then we know the distribution of $T_x$ for all ${x>0}$.  
Define the reflected 
path,
\begin{displaymath}
\hat{B}(t)=\left \{ \begin{array}{ll}
B(t) & \textrm{if $t\leq T_x$}\\
x-(B(t)-x) & \textrm{if $t>T_x$}
\end{array} \right .
\end{displaymath} 
By the Strong Markov Property and the fact that $B$ and $-B$
are equal in distribution we can deduce the {\em reflection principle} that $\hat{B}$ and $B$ are equal in
distribution: $\hat{B} \ed B$. Rigorous proof of this involves some measurability issues:  see e.g. 
Freedman \cite[\S 1.3]{fr83bd} 
Durrett \cite{MR2722836} for details.
Observe that for $x,y>0$,
$$(M_t\geq x,B_t\leq x-y)=(\hat{B}_t\geq x+y)$$
so $$\PR^0(M_t\geq x,B_t\leq x-y)=\PR^0(B_t\geq x+y)$$
Taking $y=0$ in the previous expression we have
$$\PR^0(M_t\geq x,B_t\leq x)=\PR^0(B_t\geq x).$$
But $(B_t>x)\subset (M_t\geq x)$, so 
$$\PR^0(M_t\geq x,B_t> x)=\PR^0(B_t> x)=\PR^0(B_t\geq x)$$ 
by continuity of the distribution.  Adding these two results we find that
$$\PR^0(M_t\geq x)=2\PR^0(B_t\geq x)$$
So the distributions of $M_t$ and $|B_t|$ are the same: $M_t \ed |B_t|$.

Now recall that $\PR^0(M_t\geq x) = \PR^0(T_x\leq t )$ so
\begin{eqnarray*}
\PR^0(T_x\leq t)&=& \PR^0(|B_t|\geq x) = \PR^0(\sqrt{t} |B_1|\geq x)\\
 &=& \PR^0(B^2_1\geq \frac{x^2}{t}) = \PR^0(\frac{x^2}{B^2_1} \leq t)
\end{eqnarray*}
So $T_x \ed \frac{x^2}{B^2_1}$.
As a check, this implies $T_x \ed x^2 T_1$, which is explained by Brownian scaling.

The joint distribution of the minimum, maximum and final value of $B$ on an interval can be obtained by
repeated reflections. See e.g. \cite{MR1700749} 
and  \cite[\S 4.1]{MR1848256} 
for related results involving the extremes of Brownian bridge and excursion.
See also \cite{borsal02} for corresponding results up to various random times.

\subsection{Occupation times and local times}

For $f(x) = 1(x \in A )$ for a Borel set $A$ the integral
\eq \lb{occT}
\int_0^T f(B_s) ds
\en
represents the amount of time that the Brownian path has spent
in $A$ up to time $T$, which might be either fixed or random.
As $f$ varies, this integral functional defines a random measure on
the range of the path of $B$, the \emph{random occupation measure}
of $B$ on $[0,T]$.
A basic technique for finding the distribution of the integral functional
\re{occT} is provided by the method of Feynman-Kac, which is discussed in
most textbooks on Brownian motion.
Few explicit formulas are known, except in dimension one.
A well known application of the Feynman-Kac formula is Kac's derivation of
L\'evy 's arcsine law for $B$ a BM$(\reals)$, that is for all fixed times
$T$
\eq
\lb{eq:arcsine}
\PR^0 \left( {1 \over T } \int_0^T 1(B_s >0 ) ds \le u \right)=\frac{2}{\pi}\arcsin(\sqrt{u}) \qquad (0 \le u \le 1).
\en
See for instance 
\cite{MR2604525} 
for a recent account of this approach, and
Watanabe \cite{MR1335470} 
for various generalizations to one-dimensional diffusion processes and random walks.
Other generalizations of L\'evy's arcsine law for occupation times were developed by Lamperti
\cite{MR0094863} 
and Barlow, Pitman and Yor \cite{MR1022918},  
\cite{MR1478738}. 
See also \cite{MR1292188} 
\cite{MR1671824}. 
See Bingham and Doney \cite{MR929510} and
Desbois \cite{MR2305165} 
regarding higher-dimensional analogues of the arc-sine law,
and  Desbois \cite{MR1947234} 
\cite{MR2525259} 
for occupation times for Brownian motion on a graph.

It was shown by Trotter 
\cite{MR0096311} 
that almost surely the random occupation measure induced by
the sample path of a one dimensional Brownian motion $B = (B_t, t \ge 0 )$ admits 
a jointly continuous local time process $(L_t^x(B); x \in \reals, t \ge 0)$ 
satisfying the {\em occupation density formula}
\eq
\lb{locdef}
\int_0^t f( B_s ) ds = \int_{-\infty}^\infty L_t^x(B) f(x) dx .
\en
See \cite{mck75,kni81,ry99} for proofs of this. 
Immediately from \re{locdef} there is the almost sure approximation
\eq
\lb{loc1}
L_t^x = \lim_{\epsilon \te 0 } {1 \over 2 \epsilon} \int_0^t 1( |B_s - x | \le \epsilon)
\en
which was used by \Lev\ to define the process $(L_t^x, t \ge 0 )$ for each fixed
$x$.  Other such approximations, also due to \Lev, are
\eq
\lb{loc2}
L_t^x = \lim_{\epsilon \te 0 } \epsilon D[x,x + \epsilon,B,t]
\en
where $D[x,x + \epsilon,B,t]$ is the number of downcrossings of the interval
$[x,x + \epsilon]$ by $B$ up to time $t$, and
\eq
\lb{loc3}
L_t^x = \lim_{\epsilon \te 0 } \sqrt{ \frac{ \pi \epsilon }{2 } } N[x,\epsilon,B,t]
\en
where $N[x,\epsilon,B,t]$ is derived from the 
random closed level set $\ZZ_x : = \{s : B_s = x \}$
as the number component intervals of 
$[0,t] \setminus \ZZ_x $ whose length exceeds $\epsilon$.
See \cite[Prop. XII.(2.9)]{ry99}.
According to Taylor and Wendel \cite{MR0210196} 
and Perkins \cite{MR639146}, 
the local time $L_t^x $ is also the random Hausdorff $\ell$-measure of $\ZZ_x \cap [0,t]$ for
$\ell(v) = ( 2 v |\log |\log v ||)^{1/2}$.

\subsubsection{Reflecting Brownian motion}
For a one-dimensional BM $B$, let $\underl{B}_t:= \inf_{0 \le s \le t } B_s$.
\Lev\ showed that
\eq
\lb{levyed}
( B - \underl{B} , - \underl{B} ) \ed (|B|, L)
\en
where $L:= (L_t^0, t \ge 0)$ is the local time process of $B$ at $0$.
This basic identity in law of processes leads a large number of identities in distribution between
various functionals of Brownian motion.
For instance if $G_T$ is the time of the last $0$ of $B$ on $[0,T]$, and
$A_T$ is the time of the last minimum of $B$ on $[0,T]$, (or the time of the last maximum),
then  $G_T \ed A_T$. The distribution of $G_T/T$ and $A_T/T$ is the same for all fixed times $T$, by Brownian scaling,
and given by L\'evy's {\em arcsine law} displayed later in \re{eq:arcsine}.
For further applications of L\'evy's identity see \cite{dsy2000}. 

\subsubsection{The Ray-Knight theorems}

This subsection is an abbreviated form the account of the Ray-Knight theorems in \cite[\S 8]{csp}.
Throughout this section let $R$ denote a reflecting Brownian motion on
$[0,\infty)$, which according to L\'evy's theorem \re{levyed}
may be constructed from a standard Brownian motion $B$ either as $R= |B|$, or as $R = B- \underl{B}$.
Note that if $R = |B|$ then for $v \ge 0$ the occupation density of $R$ at level $v$ up to time $t$ is
\eq
\lb{twos}
L_t^v (R)  = L_t^v(B) + L_t^{-v}(B)
\en
and in particular $L_t^0 (R)  = 2 L_t^0(B)$.
For $\ell \ge 0$ let
\eq
\lb{telldef}
\tell:= \inf\{t : L_t^0(R) > \ell \} = \inf\{t : L_t^0(B) > \ell/2 \}.
\en
For $0 \le v < w$ let
$$
D(v,w, t) := \mbox{number of downcrossings of $[v, w]$ by $R$ before $t$}
$$
Then there is the following basic description of the
process counting downcrossings of intervals up to an inverse
local time \cite{NP89b}. 
See also \cite{wal78d} for more about Brownian downcrossings and their
relation to the Ray-Knight theorems.

The process
$$
(D( v, v + \eps ,  \tell ), v \ge 0)
$$
is a time-homogeneous Markovian birth and death process on $\{0,1, 2 \ldots \}$,
with state $0$ absorbing, transition rates
$$
n-1 ~ \stackrel{\,{n\over \eps}}{\longleftarrow} ~
~n ~\stackrel{n \over \eps}{\longrightarrow}~ n+1
$$
for $n = 1,2, \ldots$, and initial state $D( 0, \eps ,  \tell )$ which has
Poisson$(\ell /(2 \eps))$ distribution.

In more detail, the number
$D( v, v + \eps ,  \tell )$ is the number of branches at level $v$ in
a critical binary $(0,\eps)$ branching process started with a
Poisson$(\ell /(2\eps))$ number of initial individuals.
From the Poisson distribution of
$D( 0, \eps ,  \tell )$, and the law of large numbers,
$$
\lim_{ \eps \downarrow 0 } \eps D(0,\eps , \tell ) = \ell  ~~~~~~~~~~~\mbox{ almost surely }
$$
and similarly, for each $v >0$ and $\ell >0$,
by consideration of excursions of $R$ away from level $v$
$$
\lim_{ \eps \downarrow 0 } 2 \eps D(v, v + \eps , \tell ) =  L^v_{ \tell}(R) ~~~~~~~~~~~\mbox{ almost surely.}
$$
This process
$(2 \eps D(v, v + \eps , \tell ), v \ge 0)$, which
serves as an approximation to $(L^v_{ \tell}(R),v \ge 0)$,
is a Markov chain whose state space is the set of integer multiples of $2 \eps$, with
transition rates
$$
x - 2 \eps  ~
\stackrel{\,\,{x \over 2 \eps^2}}{\longleftarrow}
 ~ x ~
\stackrel{{x \over 2 \eps^2}}{\longrightarrow}
~
x + 2 \eps  
$$
for $x = 2 \eps n >0$.
The generator $G_\eps$ of this Markov chain
acts on smooth functions $f$ on $(0,\infty)$ according to

$$
(G_\eps f )(x) = 
{ x \over 2 \eps^2} f(x - 2 \eps ) +
{ x \over 2 \eps^2} f(x + 2 \eps ) 
- { x \over  \eps^2} f(x  ) 
$$
$$
~~~~~~~~~~~~~= 4 x  \, { 1 \over (2 \eps)^2 } \left[ \hf f(x - 2 \eps ) + \hf f( x + 2 \eps ) - f (x )\right]
$$
$$
\te 4 x \, {1 \over 2 } {d ^2 \over d x^2 } f  \mbox{   as } \eps \te 0.~~~~~~~~~~
$$
Hence, appealing to a suitable approximation of
diffusions by Markov chains \cite{kush74,kp91}, we obtain the following
{\em Ray-Knight theorem} (Ray \cite{ray63}, Knight \cite{kt63}):

For each fixed $\ell >0$,  and $\tell:= \inf \{t : L_t^0(R) > \ell \}$,
where $R = |B|$,
\eq
\lb{rko}
(L_{\tell}^v(R), v \ge 0 )  \ed
( \QX_{\ell,v}^{(0)}, v \ge 0  )
\en
where $( \QX_{\ell,v}^{(\delta)}, v \ge 0  )$ for $\delta \ge 0$ denotes
a squared Bessel process of dimension $\delta$ started at $\ell \ge 0$,
as in Section \ref{sec:bes}.
Moreover, if $T_\ell:= \tau_{2 \ell}:= \inf \{t > 0 : L_t^0 (B) = \ell \}$,
the the processes $(L_{T_\ell}^v(B), v \ge 0 )$ and
$(L_{T_\ell}^{-v}(B), v \ge 0 )$ are two independent copies of
$( \QX_{\ell,v}^{(0)}, v \ge 0  )$.
The squared Bessel processes and their bridges, especially for $\delta = 0,2,4$,
are involved in the description of the local time processes of numerous
Brownian path fragments \cite{kt63,ray63,wil74,py82}.
For instance, if $T_1:= \inf \{t : B_t = 1 \}$, then according to
Ray and Knight
\eq
\lb{rk2}
(L_{T_1}^v(B), 0 \le v \le 1 )
\ed (\QX_{0, 1-v}^{(2)}, 0 \le v \le 1 ) .
\en
Many proofs, variations  and extensions of these basic Ray-Knight theorems can be
found in the literature.  See for instance
\cite{kw71,ry99,jp95cyc,wal78d,jeulin83rk} and papers cited there.
The appearance of squared Bessel processes processes embedded in the local times
of Brownian motion is best understood in terms of the construction of
these processes as weak limits of Galton-Watson branching processes\index{Galton-Watson branching process!limit}
with immigration, and their consequent interpretation as continuous state
branching processes with immigration \cite{kw71}.
For instance, there is the following expression of
the \Lev-\Ito\ representation of squared Bessel processes,
and its interpretation in terms of Brownian excursions \cite{py82}, due to
Le Gall-Yor \cite{ly86c}:
For $R$ a reflecting Brownian motion on $[0,\infty)$, with $R_0 = 0$,
let
$$
Y^{(\delta)}_t := R_t + L_{t}^0(R)/\delta ~~~~(t \ge 0) .
$$
Then for $\delta >0$ the process of ultimate local times of
$Y^{(\delta)}$ is a squared Bessel process of dimension $\delta$ started at $0$:
\eq
\lb{legy}
(L_{\infty}^v (Y^{(\delta)}), v \ge 0 ) \ed (\QX_{0,v}^{(\delta)}, v \ge 0 ) .
\en
See \cite[\S 8]{csp} for further discussion of these results and their explanation in terms
of random trees embedded in Brownian excursions.

\subsection{Additive functionals}

A process $(F_t, t \ge 0)$ derived from the path of a process $X$ is called an
{\em additive functional} if for all $s,t \ge 0$
$$
F_{s+t} ( X_u, u \ge 0 )  = F_s ( X_u, u \ge 0 ) + F_t ( X_{s + u}, u \ge 0 ) 
$$
almost surely. Basic additive functionals of any process $X$ are the
integrals
\eq \lb{intf}
F_t = \int_0^t f(X_s) ds
\en
for suitable $f$. For each $x \in \reals$, the local time process  $(L_t^x, t \ge 0 )$
is an additive functional of a one-dimensional Brownian motion $B$. McKean and Tanaka \cite{MR0131295}
showed that for $X = B$ a one-dimensional Brownian motion,
every continuous additive functional of locally bounded variation can be represented as
$$
F_t = \int \mu(dx) L_t^x
$$
for some signed Radon measure $\mu$ on $\reals$. According to the occupation density formula \re{locdef},
the case $\mu(dx) = f(x) dx$ reduces to \re{intf} for $X = B$.
For a $d$-dimensional Brownian motion with $d \ge 2$ there is no such representation in terms
of local times. However each additive functional of bounded variation
can be associated with a signed measures on the state space, called  its \emph{Revuz measure} 
\cite{MR0279890} \cite{MR0281261}
\cite{MR958650} 
\cite{MR2454468}. 
Another kind of additive functional is obtained from the stochastic integral
$$
G_t = \int_0^t g(B_s) \cdot dB_s
$$
where the integrand is a function of $B_s$.
Such martingale and local martingale additive functionals were studied by Ventcel' \cite{MR0139201} for Brownian motion and by
Motoo, Kunita and Watanabe 
\cite{MR0217856}
\cite{MR0196808} 
for more general Markov processes.  

\subsection{Quadratic functionals}
\lb{sec:quad}
By a \emph{quadratic Brownian functional}, we mean primarily
a functional of the form
$$
\int \mu(ds) B_s^2
$$
for some positive measure $\mu(ds)$ on $\reals_+$.
But it is also of interest to consider the more general
functionals
$$
\int \mu(ds) \left( \int_0^\infty f(s,t) dB_t \right)^2
$$
for $\mu$ and $f$ such that
$$
\int \mu(ds) \int_0^\infty d t \, f^2(s,t) < \infty .
$$
In terms of the Wiener chaos decomposition \re{mg5}, these
functionals belong to $C_0 \bigoplus C_2$. So in full generality,
we use the term \emph{quadratic Brownian functional} to mean any functional
of the form
$$
c + \int_0^\infty d B_s \int_0^s d B_u \phi(s,u)
$$
with $c \in \reals$ and 
$\int_0^\infty d s \int_0^s d u \phi^2 (s,u) < \infty$.
We note that, with the help of Kahunen-Lo\'eve expansions, the
laws of such functionals may be decribed via their characteristic
functions. These may be expanded as infinite products, which can sometimes
be evaluated explicitly in terms of hyperbolic functions or other special
functions.
See e.g.
Neveu \cite{MR0272042} 
Hitsuda \cite{MR39:4935}. 
Perhaps the most famous example is \Lev's 
stochastic area formula
\begin{align}
\nonumber
\ER \left[ \exp \left( i \lambda \int_0^t ( X_s dY_s - Y_s dX_s ) \right) \right] &= \ER \left[ \exp  - \, \left( {\lambda^2 \over 2 } \int_0^t ds ( X_s ^2 + Y_s ^2 ) \right) \right] \\
\lb{levarea} 
&= { 1 \over \cosh ( \lambda t ) }
\end{align}
where $X$ and $Y$ are two independent standard BMs.
See 
\Lev\cite{MR0044774}  
Gaveau   \cite{MR0461589} 
Berthuet \cite{MR859843}
Biane-Yor \cite{MR898500} 
for many variations of this formula, some of which are reviewed in
Yor \cite{MR1193919}. 

A number of noteworthy identities in law between quadratic Brownian
functionals are consequences of the following elementary observation:
$$
\int_0^\infty ds \left( \int_0 ^\infty f(s,t) d B_t \right)^2 
\ed
\int_0^\infty ds \left( \int_0 ^\infty f(t,s) d B_t \right)^2 
$$
for $f \in L^2 ( \reals_+ ^2 ; ds \, dt )$.
Consequences of this observation include the following identity, which
was discovered by chemists studying the radius of gyration of random
polymers 
$$
\int_0^1 ds \, \left( B_s - \int_0^1 du B_u \right)^2 \ed \int_0^1 ds \, ( B_s - s B_1 )^2
$$
where the left side involves centering at mean value of the Brownian
path on $[0,1]$, while the right side involves a Brownian bridge.
The right side is known in empirical process theory to describe the
asymptotic distribution of the von Mises statistic 
\cite{sw86}. 
More generally
Yor \cite{MR1118934} 
explains how the Cieselski-Taylor identities, which relate the laws of occupation times and hitting times of Brownian motion in
various dimensions, may be understood in terms of such identities in law between two quadratic Brownian functionals.
See also  \cite{MR1159295} and   
\cite[Ch. 4]{MR2454984}. 

\subsection{Exponential functionals}

Some references on this topic are
\cite{MR1854494} 
\cite{MR2203675} 
\cite{MR2203676} 
\cite{MR2082657} 
\cite{MR1987321}. 
		
\section{Path decompositions and excursion theory}

A basic technique in the analysis of Brownian functionals, especially
additive functionals, is to decompose the Brownian path into various
fragments, and to express the functional of interest in terms of these path
fragments. Application of this technique demands an adequate description
of the joint distribution of the {\em pre-$\rho$} and
{\em post-$\rho$} fragments
$$
(B_t, 0 \le t \le \rho)  \mbox{ and } (B_{\rho+s} , 0 \le s < \infty)  
$$
for various random times $\rho$.
If $\rho$ is a stopping time, then according to the strong Markov property
these two fragments are conditionally independent given $B_\rho$, and the
post $\rho$ process is a Brownian motion with random initial state $B_\rho$.
But the strong Markov property says nothing about the distribution of the
pre-$\rho$ fragment.
More generally, it is of interest to consider decompositions of the
Brownian path into three or more fragments defined by cutting at two or
more random times. 

\subsection{Brownian bridge, meander and excursion}
\label{sec:brmeex}

To facilitate description of the random path fragment of
random length, the following notation is very convenient.
For a process $X:= (X_t, t \in J)$ parameterized by an interval $J$,
and $I = [G_I,D_I]$ a random subinterval of $J$ with length 
$\len{I} := D_I - G_I > 0$, we denote by
$\xomega[I]$ or $\xomega[G_I,D_I]$ the {\em fragment of $\xomega$  on $I$},
that is the process
\eq
\lb{shifty}
\xomega[I] _ u := \xomega _ {G_I + u } ~~~~~~~~~~~~~~~(0 \le u \le \len{I}) .
\en
We denote by $\xs [I]$ or $\xs [G_I,D_I]$
the {\em standardized fragment of $X$ on $I$},
defined by the {\em Brownian scaling operation}
\eq
\lb{bscale}
\xs [I]_u 
:=
\frac{ X _{G_I+ u \length{I}} - X_{G_I} }{\sqrt{\length{I} }} 
~~~~~~~( \ 0 \leq u \leq 1).  
\en
Note that the fundamental invariance of Brownian motion under
Brownian scaling can be stated in this notation as
$$
B_*[0,T] \ed B[0,1]           
$$
for each fixed time $T >0$.
Let $G_{T} := \sup \{ s : s \le T, B_s = 0 \}$ be the
last zero of $B$ before time $T$ and 
$D_{T} := \inf \{ s : s  > T , B_s = 0 \}$ be the first zero of
$B$ after time $T$. 
Let $|B|:= (|B_t|, t \ge 0)$, called {\em reflecting Brownian motion}.
It is well known 
\cite{im65,chu76,ry99} that 
there are the following identities in distribution derived by 
Brownian scaling: for each fixed $T>0$ 
\eq
\lb{bm}
B_*[0,G_T] \ed \Bbb
\en
where $\Bbb$ is a \emph{standard  Brownian bridge},
\eq
\lb{meander}
|B|_*[G_T,T] 
\ed \Bme
\en
where $\Bme$ is a \emph{standard  Brownian meander}, and
\eq
\lb{excn}
|B|_*[G_T,D_T] \ed \Bex.  
\en
where $\Bex$ is a \emph{standard  Brownian excursion}.
These identities in distribution provide a convenient unified definition of the
standard bridge, meander and excursion, which arise also as limits in
distribution of conditioned random walks, as discussed in Section \ref{sec:rw}.
It is also known that $\bb, \bme $ and $\bex$ can be constructed 
by various other operations on the paths of $B$, and
transformed from one to another by further operations \cite{bp92}.

\paragraph{The excursion straddling a fixed time}

For each fixed $T >0$, the path of $B$ on $[0,T]$ can be reconstructed in
an obvious way from the four random elements
$$
G_T, ~B_*[0,G_T], ~|B|_*[G_T,1],  ~\mbox{sign}(B_T)
$$
which are independent, the first with distribution
$$
\PR( G_T/T \in du ) = \frac{ du  }{ \pi \sqrt{u (1-u) } }  ~~~~( 0 < u < 1 )
$$
which is one of \Lev's arc-sine laws, the next a standard bridge, the
next a standard meander, and the last a uniform random sign $\pm \hf$.
Similarly, the path of $B$ on $[0,D_T]$ can be reconstructed 
from the four random elements
$$
(G_T, D_T ), ~B_*[0,G_T], ~|B|_*[G_T,D_T],  ~\mbox{sign}(B_T)
$$
which are independent, with the joint law of $(G_T, D_T)$ given by
$$
\PR( G_T /T \in du, D_T /T \in dv ) =  \frac{ du \,  dv   } { 2 \pi u^{1/2} (v-u)^{3/2}  } ~~~~( 0 < u < 1 < v  ),
$$
with $B_*[0,G_T]$
a standard bridge, $|B|_*[G_T,D_T]$ a standard excursion, and $\mbox{sign}(B_T)$ a uniform random sign $\pm \hf$.
See \cite[Chapter 7]{MR3134857}. 

For $0 < t < \infty$ let $\bbrt$ be a 
{\em Brownian bridge of length $t$}, which may be regarded as a random 
element of $C[0,t]$ or of $C[0,\infty]$, as convenient:
\eq
\lb{bbrtdef}
\bbrt(s):= \sqrt{t} \bb( (s/t) \wedge 1 )  ~~~~~(s \ge 0).
\en
Let $\bmet$ denote a {\em Brownian meander of length} $t$,
and
$\bext$ be a {\em Brownian excursion of length} $t$,
defined similarly to \re{bbrtdef} with $\bme$ or $\bex$ instead of
$\bb$.


\paragraph{Brownian excursions and the three-dimensional Bessel process\index{Three-dimensional Bessel process}}

There is a close connection between Brownian 
excursions and a particular time-homogeneous diffusion process $R_3$ on
$[0,\infty)$, commonly known as the 
{\em three-dimensional Bessel process} BES(3), due to the
representation
\eq
\lb{bestrep}
(R_3(t), t \ge 0 ) \ed 
\left( \sqrt{ \sum_{i = 1}^3 (B_i(t))^2 }, t \ge 0 \right)
\en
where the $B_i$ are three independent standard Brownian motions.
It should be understood however that this particular representation 
of $R_3$ is a relatively unimportant coincidence in distribution.
What is more important, and can be understood entirely in terms of
the random walk approximations of Brownian motion and Brownian excursion \re{btconv} and \re{donskex}, 
is that there exists a 
time-homogeneous diffusion process $R_3$ on
$[0,\infty)$ with $R_3(0) = 0$, which has the same self-similarity
property as $B$, meaning invariance under Brownian scaling, and which
can be characterized in various ways, including \re{bestrep}, but
most importantly as a Doob $h$-transform of Brownian motion.

For each fixed $t >0$, the Brownian excursion $\bext$ of length $t$
is the BES(3) bridge from $0$ to $0$ over time $t$, meaning that
$$
(\bext(s), 0 \le s \le t ) \ed (R_3(s), 0 \le s \le t \giv R_3(t) = 0).
$$
Moreover, as $t \te \infty$
\eq
\lb{exinf}
\bext \convd R_3, 
\en
and $R_3$ can be characterized in two other ways as follows:

\begin{itemize}
\item[(i)] \cite{mckean63,wil74}
The process $R_3$ is a Brownian motion on $[0,\infty)$ started at $0$
and conditioned never to return to $0$, as defined by the Doob $h$-transform,
for the harmonic function $h(x) =x$ of Brownian motion on $[0,\infty)$,
with absorbtion at $0$. 
That is, $R_3$ has continuous paths starting 
at $0$, and for each $0 < a < b$ the stretch of $R_3$ between when it 
first hits $a$ and first hits $b$ is distributed like $B$ with $B_0 = a$ 
conditioned to hit $b$ before $0$.

\item[(ii)] \cite{p75} 
There is the identity
\eq
\lb{2m-x}
R_3(t) = B(t) - 2 \underl{B}(t) ~~~~~( t \ge 0)
\en
where $B$ is a standard Brownian motion with past minimum process
$$\underl{B}(t):= \underl{B}[0,t] = - \underl{R_3}[t,\infty).$$
\end{itemize}

The identity in distribution \re{2m-x} admits numerous variations
and conditioned forms \cite{p75,bp92,MR1703609}. 
For instance, by application of \Lev's identity \re{levyed}
\eq
\lb{2m-xvar}
(R_3(t), t \ge 0 ) \ed ( |B_t| + L_t, t \ge 0 ) .
\en
where $(L_t, t \ge 0)$ is the local time process of $B$ at $0$.
\subsection{The Brownian zero set}
\newcommand{\Z}{Z}
Consider the {\em zero set} of one-dimensional Brownian motion $B$:
$$\Z(\omega):=\{t:B_t(\omega)=0\}.$$
Since $B$ has continuous paths, $\Z(\omega)$ is closed subset of $[0,\infty)$, 
which depends on $\omega$ through the path of $B$. Intuitively,  
$\Z(\omega)$ is a random closed subset of $[0,\infty)$, and this is
made rigorous by putting an appropriate $\sigma$-field 
on the set of all closed subsets of $[0,\infty)$.
Some almost sure properties of the Brownian zero are:
\begin{itemize}
\item $\Z(\omega)$ has Lebesgue measure equal to $0$;
\item $\Z(\omega)$ has Hausdorff dimension $1/2$;
\item $\Z(\omega)$ has no isolated points;
\item $\Z(\omega)$ is the set of points of increase of the local time process at $0$.
\item $\Z(\omega)$ is the closure of the range of the inverse local time process, which is a stable subordinator
of index $1/2$.
\end{itemize}
See \cite{MR1434129} for a study of the distribution of ranked lengths of component open intervals
of $(0,t) \setminus \Z(\omega)$, and generalizations to a stable subordinator of index $\alpha \in (0,1)$.


\subsection{\Lev-\Ito\ theory of Brownian excursions}
\index{Brownian excursion}
The \Lev-\Ito\ excursion theory allows the Brownian path to be reconstructed
from its random zero set, an ensemble of independent standard Brownian excursions,
and a collection of independent random signs, one for each excursion.
The zero set can first be created as the closed range of 
a stable subordinator $(T_\ell, \ell \ge 0)$ which ends up being the inverse
local time process of $B$. Then for each $\ell$ such that $T_{\ell-} < T_\ell$
the path of $B$ on $[T_{\ell-} , T_\ell]$ can be recreated by shifting and
scaling a standard Brownian excursion to start at time $T_{\ell-}$ and end at time $T_\ell$.

Due to \re{levyed}, the process of excursions of $|B|$ away from $0$ is
equivalent in distribution to the process of excursions of $B$ above 
$\underl{B}$.  
According to the \Lev-\Ito\ description of this process,
if $I_\ell:= [T_{\ell-}, T_\ell]$ for $T_{\ell}:= \inf \{t : B(t) < - \ell\}$,
the points
\eq
\lb{ppp1}
\{ ( \ell, \mu(I_\ell), ( B - \underl{B} )[I_\ell] ): \ell > 0, \mu( I_\ell) > 0 \},
\en
where $\mu$ is Lebesgue measure, are the points of 
a Poisson point process on $\realsp \times \realsp \times C[0,\infty)$ with intensity
\eq
\lb{ppp2}
d \ell  \, {dt \over \sqrt{2 \pi } \, t^{3/2} } \, \PR( \bext \in d \omega ) .
\en
On the other hand, according to Williams \cite{wil79v1},
if $M_\ell := \overline{B}[I_\ell] - \underline{B}[I_\ell]$ is the maximum
height of the excursion of $B$ over $\underl{B}$ on the interval $I_\ell$,
the points
\eq
\lb{ppp1w}
\{ ( \ell, M_\ell, ( B - \underl{B} )[I_\ell] ): \ell > 0, \mu( I_\ell) > 0 \},
\en
are the points of a Poisson point process on 
$\realsp \times \realsp \times C[0,\infty)$ with intensity
\eq
\lb{ppp2w}
d \ell  \, {dm \over m^2} \, \PR( \bexm \in d \omega ) 
\en
where $\bexm$ is a {\em Brownian excursion conditioned to have maximum $m$}.
That is to say $\bexm$ is a process $X$ with $X(0) = 0$ 
such that 
for each $m >0$,
and $H_x(X):= \inf \{t: t > 0 , X(t) = x\}$,
the processes $X[0, H_m(X)]$
and $m - X[H_m(X),H_0(X)]$ are two independent copies of
$R_3[0, H_m(R_3)]$, and $X$ is stopped at $0$ at time $H_0(X)$.
{\em \Ito's law of Brownian excursions}
is the $\sigma$-finite measure $\nu$ on $C[0,\infty)$ 
which can be presented in two different ways according to \re{ppp2} and
\re{ppp2w} as
\eq
\lb{itolaw}
\nu(\cdot) = \int_0^\infty 
{dt \over \sqrt{2 \pi } t^{3/2} } \, \PR( \bext \in \cdot )
= 
\int_0^\infty 
{dm \over m^2} \, \PR( \bexm \in \cdot ) 
\en
where the first expression is a disintegration according to the
lifetime of the excursion, and the second according to its maximum.
The identity \re{itolaw} has a number of interesting applications
and generalizations \cite{bpy99z,py95agr,MR1675059}.  
See 
\cite[Ch. XII]{ry99}, 
\cite{MR2295611}  
and \cite{MR3134857} 
for more detailed accounts of It\^o's excursion theory and its applications.
\notes
See \cite{rp81,lg.will,bert92a,ry99,hmo01}
for different approaches to the basic path transformation \re{2m-x} from 
$B$ to $R_3$, its discrete analogs, and various extensions.
In terms of $X:= -B$ and $M:= \overline{X} = - \underline{B}$,
the transformation takes $X$ to $2M-X$.
For a generalization to exponential functionals, see Matsumoto and Yor
\cite{matyor99c}.
This is also discussed in \cite{oy01}, where an alternative
proof is given using reversibility and symmetry arguments,
with an application to a certain directed polymer problem.
A multidimensional extension 
is presented in \cite{oy02}, where a representation for
Brownian motion conditioned never to exit a (type A)
Weyl chamber is obtained using reversibility and symmetry
properties of certain queueing networks.  See also
\cite{oy01,kor02} and the survey paper \cite{oc.survey}.
This representation theorem is closely connected to
random matrices, Young tableaux, the Robinson-Schensted-Knuth correspondence,
and symmetric functions theory \cite{oc.rsk,oc.rsk1}.
A similar representation theorem has been obtained in \cite{bj}
in a more general symmetric spaces context, using quite
different methods.  
These multidimensional versions of the transformation from
$X$ to $2M-X$ are intimately connected with combinatorial representation
theory and Littelmann's path model \cite{littelman95}.



\section{Planar Brownian motion}
\subsection{Conformal invariance}
\Lev\  showed that if
$$
Z_t := B_t^{(1)} + i B_t^{(2)}
$$
is a $2$-dimensional Brownian motion, regarded here as a $\complex$-valued
process, and
$f: \complex \to \complex $ is a non-constant holomorphic function,
then
\eq
\lb{dagger}
f(Z_t) = \widehat{Z} \left( \int_0^t | f'(Z_s)|^2 ds  \right)
\en
where $(\widehat{Z}(u), u \ge 0 )$ is another $\complex$-valued Brownian
motion.
This is an instance of the Dambis-Dubins-Schwarz result of Section \ref{sec:chtime} in two dimensions.
For 
$$
f(Z_t) = F_t^{(1)} + i F_t^{(2)}
$$
for real-valued processes $F^{(1)}$ and $F^{(2)}$
which are two local martingales relative to the filtration of
$(Z_t)$, with
$$
\langle  F^{(1)} \rangle_ t = \langle F^{(2)} \rangle_ t  \mbox{  and  }
\langle F^{(1)}, F^{(2)} \rangle_ t  \equiv 0
$$
as a consequence of \Ito's formula and the Cauchy-Riemann equations for $f$.

More generally, Getoor and Sharpe 
\cite{MR0305473} 
defined a 
{\em conformal martingale} to be any $\complex$-valued continuous local
martingale $(Z_t:= X_t + i Y_t, t \ge 0 )$ for real-valued $X$ and $Y$
such that
$$
\langle X \rangle_ t = \langle Y \rangle_ t   \mbox{    and    } \langle X,Y \rangle_ t \equiv 0.
$$
They showed in this setting that
$$
Z_t = \widehat{Z}_{\langle X \rangle_ t}    ~~~~~~(t \ge 0 )
$$
where $(\widehat{Z}_{u}, u \ge 0 )$ is a $\complex$-valued Brownian motion.
Note that conformal martingales are stable by composition with an entire
holomorphic function, and also by continuous time-changes.
See \cite{MR87a:60054}, 
\cite{MR80j:60105}   
\cite{MR2604525} 
for many applications.
\subsection{Polarity of points, and windings}
According to \Lev's theorem,
\eq \lb{expl}
\exp( Z_t ) = \widehat{Z} \left( \int_0^t ds \exp (2 X_s)\right)
\en
with $( \widehat{Z}(u), u \ge 0 )$ another planar BM started at $1$,
and $X_s$ the real part of $Z_s$.
It follows immediately from \re{expl} that $\widehat{Z}$ will never
visit $0$ almost surely. From this, it follows easily that for
two arbitrary points $z_0$ and $z_1$ with $z_0 \ne z_1$
$$
\PR^{z_0}( Z_t = z_1 \mbox{ for some } t \ge 0 ) = 0 .
$$
In particular, the winding number process, that is a continuous
determination $(\theta_t^{(z_1)}, t \ge 0 )$ of the argument of $Z_t - z_1$ along the path of
$Z$, is almost surely well defined for all $t \ge 0$, and
so is the corresponding complex logarithm of $Z_t - z_1$, according to
the formula
\eq \lb{wind}
\log ( | Z_t - z_1 |)  + i \theta_t^{(z_1)} :=  \int_0^t { d Z_u \over Z_u - z_1 }
= \tilde{Z} \left( \int_0^t { du \over | Z_u - z_1 |^2 } \right)
\en
for $t \ge 0$, where, by another application of \Lev's theorem, the
process $\tilde{Z}$ is a complex Brownian motion starting at $0$.
Moreover, from the trivial identity
$$
Z_t - z_1 = (z_0 - z_1)   + \int_0 ^t { dZ_u \over (Z_u - z_1 ) } ( Z_u - z _1)
$$
considered as a linear integral equation in $(Z_u - z_1)$,
we see that
$$
Z_t - z_1 = (z_0 - z_1 )  \cdot  \exp \left( \int_0 ^t { dZ_u \over (Z_u - z_1 ) } \right).
$$
Taking $z_1 = 0$, this yields the \emph{skew-product representation} of the planar BM
$Z$ started from $z_0 \ne 0$:
\eq\lb{skew}
Z_t = |Z_t| \, \exp \left[ i \gamma \left(     \int_0^t {ds \over |Z_s|^2 } \right) \right]
\en
for $t \ge 0$, where $(\gamma(u), u \ge 0 )$ is a \oned\ BM, independent of the
radial part $( |Z_t|, t \ge 0 )$, which is by definition a 2-dimensional
Bessel process.
This skew-product representation reduces a number of problems involving
a planar Brownian motion $Z$ to problems involving just its radial part.


\subsection{Asymptotic laws}
In 1958, the following two basic asymptotic laws of planar Brownian motion
were discovered independently:

\paragraph{Spitzer's law \cite{MR0104296}} 
$$
{2 \theta _t \over \log t } \convd C_1 \mbox{ as } t \to \infty
$$
where $\theta_t$ denotes the winding number of $Z$ around $0$, assuming
$Z_0 \ne 0$, and $C_1$ denotes a standard Cauchy variable.

\paragraph{The Kallianpur-Robbins law \cite{MR0056233}} 
$$
{1 \over (\log t ) ||f||} \int_0^t ds f(Z_s) \convd e_1 
\mbox{ as } t \to \infty
$$
for all non-negative measurable functions $f$ with 
$||f||:= \int\int f(x + i y ) dx dy < \infty$, and $e_1$ a standard exponential
variable, along with the ratio ergodic theorem
$$
{\int_0^t ds f(Z_s) \over \int_0^t ds g(Z_s) } \convas { ||f|| \over ||g|| }
\mbox{ as } t \to \infty
$$
for two such functions $f$ and $g$.
It was shown in \cite{MR841582} 
that the skew product representation of planar  BM allows a unified derivation of these
asymptotic laws, along with descriptions of the joint asymptotic behaviour
of windings about several points and an additive functions
$$
{1 \over \log t } \left( 
\theta_t^{(z_1)}, \ldots,
\theta_t^{(z_k)},  \int_0^t ds f(Z_s) \right) .
$$
For a number of extensions of these results, and a review of literature
around this theme, see 
Yor \cite[Ch. 8]{MR1193919}. 
See also 
Watanabe \cite{MR1834236} 
Franchi  \cite{MR1141243}  
Le Gall and Yor  
\cite{MR946219} 
\cite{MR876259} 
\cite{MR816704} 
for various extensions of these
asymptotic laws: to more general recurrent diffusion processes in the plane,
to Brownian motion on Riemann surfaces, and to windings of Brownian motion
about lines and curves in higher dimensions.
See also \cite[Chapter 7]{yor92z} where some asymptotics are obtained for
the self-linking number of BM in $\reals^3$.
Questions about knotting and entanglement of Brownian paths and random 
walks are studied in \cite{MR81f:60110} 
and \cite{MR98e:82031}. 
Much more about planar Brownian motion can be found in 
Le Gall's course \cite{MR1229519}. 

\subsection{Self-intersections}
In a series of remarkable papers in the 1950's, Dvoretsky, \Erdos,
Kakutani and Taylor established among other things the existence of
multiple points of arbitary (even infinite) order in planar Brownian
paths. It was not until the 1970's that any attempt was made to quantify the
extent of self-intersection of Brownian paths by consideration of
the occupation measure of Brownian increments
\eq
\lb{interloc}
\nu_{s,t}^{(2)}(\omega, dx):= \int_0^s du \,\int_0^t dv \, 1( B_u - B_v \in dx) .
\en
for a planar Brownian motion $B$, and $x \in \reals^2$. Wolpert \cite{MR80a:60104}
and Rosen \cite{MR942041} 
showed that this random measure is 
almost surely absolutely continuous with respect to Lebsesgue measure 
on $\reals^2$, with a density
\eq
\lb{interloc1}
\tilde{\alpha}(x; s,t), x \in \reals^2 - \{0\}, s,t, \ge 0), 
\en
which can be chosen to be jointly continuous  in $(x,s,t)$.
For fixed $x$ this process in $s,t$ is called the 
the process of \emph{intersection local times at} $x$.
However,
$$
\lim_{x \te 0} \tilde{\alpha}(x; s,t) = \infty  \mbox{ almost surely }
$$
reflecting the accumulation of immediate intersections of the Brownian
path with itself coming from times $u,v$ in \re{interloc} with $u$ close to $v$.
A measure of the extent of such self-intersections is obtained by
consideration of
$$
\nu_{s,t}^{(2)}f_n :=
\int_{\reals^2} f_n(x) \nu_{s,t}^{(2)}(\omega, dx) = 
\int_0^s du \,\int_0^t dv f_n( B_u - B_v ) 
$$
as $n \te \infty$ with $f_n(x):= n^2 f(nx)$ for a continuous non-negative
function $f$ with compact support and planar Lebesgue integral equal to $1$.  Varadhan 
\cite{vara69} showed that
$$
\nu_{s,t}^{(2)}f_n  - \ER (\nu_{s,t}^{(2)}f_n  ) \te \gamma_{s,t} \mbox{ as } n \te \infty
$$
in every $L^p$ space for some limit $\gamma_{s,t}$ which is independent of $f$.
Indeed, 
$$
\tilde{\alpha}(x; s,t) - \ER [\tilde{\alpha}(x; s,t) ]
\te \gamma_{s,t} \mbox{ as } x \te 0
$$
in $L^2$.
The limit process $(\gamma_{s,t} , s,t \ge 0 )$ is known to admit a continuous
version in $(s,t)$, called the
\emph{renormalized self-intersection local time}
of planar Brownian motion.
Rosen 
\cite{MR866345} 
obtained variants of Tanaka's formula for these local times of intersection.

These results have been extended in a number of ways. 
In particular,
for each $k >2$ the occupation measure of Brownian increments of order $k-1$
$$
\int_0^{s_1} du_1 \cdots \int_0^{s_k} du_k 
\prod_{i = 2}^k ( B_{u_{i}} - B_{u_{i-1} } \in dx_i )
$$
is absolutely continuous with respect to the Lebesgue measure
$dx_2 \cdots dx_{k}$, and Varadhan's renormalization result extends as follows:
the multiple integrals
$$
\int \cdots \int _{0 \le s_1 < \cdots < s_k \le t } ds_1 \cdots ds_k 
\prod_{i = 2}^k \overline{ f _n( B_{u_{i}} - B_{u_{i-1} } ) } ,
$$
where $\overline{X}:= X - \ER(X)$, converge in $L^p$ for every $ p \ge 1$
as $n \te \infty$, to define a $k$th order renormalized intersection local
time.
Amongst a number of deep applications of these intersection local times,
we mention the asymptotic series expansion of the area of the 
\emph{Wiener sausage}
$$
S_\epsilon(t) :=  \{ y :  | y - B_s | \le \eps  \mbox{ for some } 0 \le s \le t 
\},
$$
according to which for each fixed $n = 1,2, \ldots$, as $\epsilon \te 0$
$$
m( S_\epsilon (t ) ) = \sum_{k = 1}^n {\gamma_k(t) \over ( \log(1/\epsilon) )^k } + o \left( { 1 \over ( \log(1/\epsilon) )^n } \right) 
$$
where $\gamma_k(t)$ is the $k$th order normalized intersection local time.
These results, and much more in the same vein, can be found in the course of
Le Gall \cite{MR94g:60156}.
Subsequent open problems were collected in \cite{MR93k:60205}.

Questions about self-intersection of planar Brownian paths are closely
related to questions about intersections of two or more independent Brownian
paths.  Such questions are easier, because a process 
$L_{t,s}^{(1,2)}$ of local time of intersection at $0$  between
two independent Brownian paths $(B_u^{(1)}, 0 \le u \le t )$ and
$(B_v^{(2)}, 0 \le v \le s )$ is well-defined and finite.
Le Gall's proof of Varadhan's renormalization uses the convergence of a centered
sequence of such local times of intersection $L^{(1,2)}$.

\subsection{Exponents of non-intersection}
Another circle of questions involves estimates of the probability of
non-intersection of independent planar BM's $B^{(1)}$ and $B^{(2)}$.
If $Z_R^{(i)}$ is the range of the path of $B^{(i)}_t$  for
$0 \le t \le \inf \{t : | B_t | = R \}$, and 
$B^{(i)}_0$ is uniformly distributed on the unit circle, then there
exist two universal constants $c_1$ and $c_2$ such that
\eq \lb{exps}
c_1 R^{-5/4} \le \PR ( Z_R^{(1)} \cap Z_R^{(2)}  = \emptyset ) \le c_2 R^{-5/4}.
\en
More generally, for $p$ independent BM's, there is a corresponding
exponent 
$$\xi_p = (4 p^2 - 1)/12 $$
instead of $5/4$, so that
$$
c_{1,p} R^{-\xi_p} 
\le 
\PR \left( \mbox{$\cap_{i = 1}^p Z_R^{(i)}$} = \emptyset \right) 
\le c_{2,p} R^{-\xi_p} 
$$
Another interesting family of exponents, the \emph{non-disconnection exponents}
$\eta_p$ are defined by
$$
c_{1,p}^\prime R^{-\xi_p} 
\le 
\PR \left( {\cal D } ( \mbox{$\cup_{i = 1}^p Z_R^{(i)} )$} \right) 
\le c_{2,p} ^\prime
R^{-\xi_p} 
$$
where ${\cal D }(\Gamma)$ is the event that the random set $\Gamma$
does not disconnect $0$ from $\infty$. It is known that these exponents
exist, and are given by the formula
$$
\eta_p = { (\sqrt{ 2 4 p + 1 } - 1 )^2 - 4 \over 48 } .
$$
In particular, $\eta_1 = 1/4$.
For an account of these results, see 
Werner \cite[p. 165, Theorem 8.5]{MR2079670}. 
These results are the keys to computation of the Hausdorff  dimension $d_H(\Gamma)$ of a number of interesting random
sets $\Gamma$ derived from a planar Brownian path, in particular
\begin{itemize}
\item
the set $C$ of \emph{cut points} $B_t$ for $t$ such that $B[0,t] \cap B[t,1] = \emptyset$ : $$d_H(C) = 2 - \xi_1 =  3/4$$
\item
the set $F$ of \emph{frontier points} 
$B_t$ for $t$ such that ${\cal D }(B[0,1] - B _t)$ :
$$d_H(F) = 2 - \eta_2 = 4/3$$
\item
the set $P$ of \emph{pioneer points} 
$B_t$ such that ${\cal D }(B[0,t] - B_t )$ :
$$d_H(P) = 2 - \eta_1 = 7/4.$$
\end{itemize}
In particular, the second of these evaluations established a famous conjecture
of Mandelbrot.
 For introductions to these ideas and the closely related theory of {\em Schramm-Loewner evolutions} driven by Brownian motion,
see Le Gall's review \cite{MR1886755} of the work of
Kenyon, Lawler and Werner on critical exponents for random walks and 
Brownian motion, the appendix of 
\cite{MR2604525} by Schramm and Werner, and the monograph of 
Lawler \cite{MR2129588} 
on conformally invariant processes in the plane,
\section{Multidimensional BM}
\subsection{Skew product representation}
The skew product representation \re{skew} of planar Brownian motion
can be generalized as follows to a BM $B$ in $\reals ^d$ for $d \ge 2$:
\eq
\lb{skewd}
B_t = |B_t | \theta( \mbox{$\int_0^t ds/|B_s|^2$} ),       ~~~~~~ t \ge 0
\en
where the radial process $(|B_t|, t \ge 0 )$ is a $d$-dimensional Bessel process,
as discussed in Section \ref{sec:bes}, and the angular Brownian motion
$(\theta_u = \theta(u), u \ge 0 )$ is a BM on the sphere $\SR^{d-1}$,
independent of the radial process. This is a particular case of BM on 
a manifold, as discussed further in Section \ref{sec.manifold}.
For now, we just indicate Stroock's representation of the angular motion
$(\theta_u, u \ge 0)$
as the solution of the Stratonovich differential equation
\eq
\lb{stratok}
\theta_u^i = \theta_0^i + \sum_{j = 1}^d \int_0^u ( \delta_{ij} - \theta_s^i \theta_s^j ) \circ dW_s^j,       ~~~~~~(1 \le i \le d)
\en
where $(W_s)$ is a $d$-dimensional BM independent of $(|B_s|)$.
This representation \re{stratok} may be obtained by application of \Ito's
formula to $f(x) = x/|x|$ and Knight's theorem 
of Section \ref{sec:ktthm}
on orthogonal martingales.

\subsection{Transience}
The BM in $\reals^d$ is transient for $d \ge 3$, and the study of its
rate of escape to $\infty$ relies largely on the application of one-dimensional
diffusion theory to the radial process.
See for instance Shiga and Watanabe \cite{sw73}.
According to the theory of last exit decompositions, the \emph{escape process} 
$$ S_r := \sup \{t,  |B_t|  = r \}, ~~~~~~~ r \ge 0 $$
is a process with independent increments, whose distribution
can be described quite explicitly 
\cite{getoor79be}. 
In the special case $d = 3$ this process is the stable process of
index $\hf$, with stationary independent increments, and is identical
in law to $(T_r, r \ge 0)$ where $T_r$ is the first hitting time of $r$
by a one-dimensional Brownian motion $B$. This and other coincidences
involving one-dimensional Brownian motion and the three-dimensional Bessel
process BES(3) are explained by the fact that BES(3) is the Doob 
$h$-transform of BM on $(0,\infty)$ with killing at $0$.

\subsection{Other geometric aspects}

Studies of stochastic  integrals such as
$$
\int_0^t ( B_s^i dB_s^j - B_s^j dB_s^i ) , ~~~~ i \ne j
$$
arise naturally in various contexts, such as Brownian motion on the
Heisenberg group.
See e.g. 
Gaveau   \cite{MR0461589} and 
Berthuet \cite{MR859843}.

\subsubsection{Self-intersections}

Dvoretsky, \Erdos\ and Kakutani \cite{MR0034972} showed that
two independent BM's in $\reals^3$ intersect almost surely, even if started apart.
Consequently,  a single Brownian path in $\reals^3$ has self-intersections almost surely.
An occupation measure of Brownian increments
$\nu_{s,t}^{(3)}$ can be defined exactly as in \re{interloc}, for $B$ with values in
$\reals^3$ instead of $\reals^2$, and this measure admits  a 
density $\tilde{\alpha}(x; s,t), x \in \reals^3 - \{0\}, s,t, \ge 0$ which may be chosen
jointly continuous in $(x,s,t)$. Again $\lim_{x \te 0 } \tilde{\alpha}(x; s,t) = \infty$ almost surely,
but now Varadhan's renormalization phenomenon does not occur. Rather, there is a weaker result.
In agreement with Symanzik's program for quantum field theory, Westwater \cite{MR84k:82074a} showed
weak convergence as $n \te \infty$ of the measures
\eq
\lb{ww}
{\exp \left( - g \nu_{s,t}^{(3)}(f_n) \right) 
\over Z_n^g (s,t) }
\cdot W^{(3)}
\en
where $g >0$, $W^{(3)}$ is the Wiener measure,
$$
\nu_{s,t}^{(3)}(f_n)  = \int_{\reals^3} f_n(x) \nu_{s,t}^{(3)}(\omega; dx)  
= \int_0^s du \int_0^t dv f_n(B_u - B_v)
$$
where $f_n(x):= n^3 f(nx)$ for a continuous non-negative
function $f$ with compact support and Lebesgue integral equal to $1$.  
Westwater showed that as $g$ varies the weak limits $W^{(3,g)_{s,t}}$ are mutually singular,
and that under each $W^{(3,g)_{s,t}}$ the new process, while no longer a BM, still has self-intersections.

For $\delta \ge 4$, two independent BM's in $\reals^\delta$ do not intersect, and consequently BM$(\reals^\delta)$
has no self-intersections. The analog of \re{ww} can nonetheless be studied, with the result that these measures
still have weak limits. 
Some references are
\cite{MR1240717} 
\cite{MR729794} 
\cite{MR3382172} 
\cite{MR968950}.  

\subsection{Multidimensional diffusions}

Basic references on this subject are:

\begbib
\item[\cite{MR81f:60108}]D. ~W. Stroock and S.~R.~S. Varadhan. \newblock {\em Multidimensional diffusion processes} (1979).
%
\item[\cite{MR88d:60203}] D.~W. Stroock. \newblock {\em Lectures on stochastic analysis: diffusion theory} (1987).

\item [\cite{varadhan01}]  S. R. S. Varadhan. {\em Diffusion processes} (2001).  

\item [\cite{watanabe01}]  S. Watanabe. {\em It\^o's stochastic calculus and its applications} (2001). 

\endbib

See also 
\cite[Chapter 5]{ks88},
\cite{oks03}
\cite{bass98deo}.
Basic notions are weak and strong solutions of stochastic differential equations,
and martingale problems.
We refer to Platen \cite{MR2002f:65019} for a survey of literature on
numerical methods for SDE's, including Euler approximations,
stochastic Taylor expansions, multiple It\^o and Stratonovich integrals, 
strong and weak approximation methods, Monte Carlo simulations and
variance reduction techniques for functionals of diffusion processes.
An earlier text in this area is Kloeden and Platen \cite{MR94b:60069}.


\subsection{Matrix-valued diffusions}
Another interesting example of BM in higher dimensional spaces is provided
by matrix-valued BM's, which are of increasing interest in the theory of
random matrices. See for instance O'Connell's survey \cite{oc.survey}.
C\'epa  and L\'epingle  \cite{MR2002j:60180} 
interpret Dyson's model for the eigenvalues of $N \times N$ 
unitary random random matrices as a system of $N$ Brownian interacting
Brownian particles on the circle with electrostatic repulsion. 
They discuss more general particle systems allowing collisions between
particles, and measure-valued limits of such systems as $N \te \infty$.
This relates to the asymptotic theory of random matrices, which concerns 
asymptotic features as $N \te \infty$ of various statistics of
$N\times N$ matrix ensembles.
See also \cite{MR1132135}. 
Another interesting development from random matrix theory is the theory of
Dyson's Brownian motions
\cite{MR914993} 
\cite{MR2299928}. 

\subsection{Boundaries}

\subsubsection{Absorbtion}
It is natural in many contexts to consider Brownian motion and diffusions
in some connected open subset $D$ of $\reals^\deltaN$. The simplest such process
is obtained \emph{absorbing} $B$ when it first reaches the
boundary at time $T_D = \inf \{t \ge 0 : B_t \notin D \}$.
This gives a Markov process $X$,
with state space the closure of $D$,
defined by
$X_t = B_{t \wedge T_D }$.
A key connection with classical analysis
is provided by considering the density of the mean occupation measure of the killed BM: for all 
non-negative Borel measurable functions $f$ vanishing off $D$
\eq \lb{greenD}
\ER^x \int_0^\infty f(X_t) dt = \ER^x \int_0^{T_D} f(B_t) dt  = \int_D f(y) g_D(x,y) dy
\en
where $g_D(x,y)$ is the classical \emph{Green function} associated with the domain $D$.
A classical fact, not particularly obvious from \re{greenD}, is that $g_D$ is a symmetric
function of $(x,y)$.
See 
M\"orters and Peres \cite[\S 3.3]{MR2604525} 
and Chung \cite{chung02} for further discussion.

\subsubsection{Reflection}
If $H$ is the half-space of $\reals^\deltaN$ on one side of a hyperplane, there is an obvious
way to create a process $R$ with continuous paths in the closure of $H$ by reflection through
the hyperplane of a Brownian motion in $\reals^\deltaN$. By Skorokhod's analysis of reflecting 
Brownian motion on $[0,\infty)$ in the one-dimensional case, this reflecting Brownian motion $R$ 
in $H$ can be described as a semi-martingale, which is the sum of a Brownian motion and
process of bounded variation which adds a push only when $R$ is on the boundary hyperplane,
in a direction normal to the hyperplane, and of precisely the right magnitude to keep $R$ on
one side of the hyperplane.
This semi-martingale description has been generalized to characterize reflecting BM 
in a convex polytope bounded by any finite number of hyperplanes, and further to domains with
smooth boundaries.
See for instance
\cite{MR745330} 
\cite{MR1648166} 
\cite{MR2308342} 
\cite{MR1714981} 
\cite[Ch. 5]{MR3184496} 
A basic property of such reflecting Brownian motions $R$ is that in great generality 
Lebesgue measure on the domain is the unique invariant measure.
In particular, if $D$ is compact, as $t \te \infty$ the limiting distribution 
of $R_t$ is uniform  on $D$.
More complex boundary behaviour is possible, and of interest in applications
of reflecting BM to queuing theory 
\cite{MR1648166} 
and the study of Schramm-Loewner evolutions
\cite[Appendix C]{MR2129588}. 

\subsubsection{Other boundary conditions}
See Ikeda-Watanabe \cite{ikwat89} for a general discussion of boundary conditions for diffusions, with references
to the large Japanese literature of papers by Sato, Ueno, Motoo and others dating back to the 60's.

\section{Brownian motion on manifolds}
\lb{sec.manifold}
\subsection{Constructions}
A Riemannian manifold $M$ is a manifold equipped with a Riemannian metric.
Starting from this structure, there are various expressions for the Laplace-Beltrami operator
$\delm$, and the Levi-Civita connection.
Closely associated with the Laplace-Beltrami operator is the
fundamental solution of the heat equation on $M$ derived from $\hf \delm$.
This defines a semigroup of transition probability operators 
from which one can construct a Brownian motion on $M$.
Alternatively, the Brownian motion on $M$ with generator $\hf \delm$ 
can be constructed by solving a martingale problem associated with $\hf \delm$.
Note that in general the possibility of explosion must be allowed:
the $M$-valued Brownian motion $B$ may be defined only up to some
random \emph{explosion time} $e(B)$.

At least two other constructions of Brownian motion on $M$ may be considered,
one known as \emph{extrinsic}, the other as \emph{intrinsic}.
Some examples of the extrinsic construction appear in the work of
Lewis 
and van den Berg 
\cite{MR859959} \cite{MR806240}. 
In general, this construction relies on Nash's embedding of $M$ as 
a submanifold of $\reals^\ell$, with the induced metric.
Following Hsu \cite[Ch. 3]{MR2003c:58026},
let $\{ \xi_\alpha , 1 \le \alpha \le \ell\}$  be the standard orthonormal
basis in $\reals^\ell$, let $P_\alpha$ be the orthogonal projection of
$\xi_\alpha$ onto $T_xM$, the tangent space at $x \in M$.
Then $P_\alpha$ is a vector field on $M$, and the Laplace-Beltrami
operator $\delm$ can be written as
$$
\delm = \sum_{\alpha = 1}^\ell P_\alpha^2 
$$
and the Brownian motion started at $x \in M$ may be constructed as the solution of the Stratonovich SDE
$$
d X_t = \sum_\alpha P_\alpha (X_t) \circ dW_t^\alpha, ~~~~X_0 = x \in M .
$$
where $(W^\alpha, 1 \le \alpha \le \ell)$ is a BM in $\reals^\ell$.
See also Rogers and Williams \cite{rw2} and Stroock 
\cite[Ch. 4]{MR1715265} 
for further development of the extrinsic approach.

The intrinsic approach to construction of BM on a manifold involves a lot
more differential geometry. See Stroock \cite[Chapters 7 and 8]{MR1715265} and
other texts listed in the references, which include the 
theory of semimartingales on manifolds, as developed by L. Schwartz, P. A. Meyer and M. Emery. 

\subsection{Radial processes}
Pick a point $o \in M$, a Riemannian manifold of dimension $2$ or more, and with the help of the exponential map based at $o$,
define polar coordinates $(r,\theta)$ and hence processes 
$r(B_t)$ and $\theta(B_t)$ where $B$ is Brownian motion on $M$.
The \emph{cutlocus of $o$}, denoted $C_o$, is a subset of $M$ such that
$r$ is smooth on $M - \{o\} - C_o$, the set $M - C_o$ is a dense open subset of $M$, and the distance from $o$ to $C_o$ is positive.
Kendall \cite{MR88k:60151} showed that
there exists a one-dimensional Brownian motion $\beta$ and a non-decreasing
process $L$, the local time process at the cutlocus, which increases only when $B_t \in C_o$ 
such that
$$
r(X_t) - r(X_0) =  \beta_t  + \hf \int_0^t \delm \, r (B_s ) ds - L_t
~~~~~~t  < e(B).
$$
See also 
\cite{MR96d:58153},
\cite[Th. 3.5.1]{MR2003c:58026} and
\cite{MR89i:58157}.  
Hsu \cite[\S 4.2]{MR2003c:58026},
shows how this representation of the radial process allows a
comparison with a one-dimensional diffusion process to conclude 
that a growth condition on the lower bound of the Ricci curvature 
provides a sufficient condition for the BM not to explode. The condition
is also necessary under a further regularity condition on $M$.

\subsection{References}
We refer to the following monographs and survey articles
for further study of Brownian motion and diffusions on manifolds.

\paragraph{Monographs}
\begbib
\item[\cite{MR84d:58080}]K.~D. Elworthy. \newblock {\em Stochastic differential equations on manifolds} (1982).

\item[\cite{MR90c:58187}]David Elworthy. \newblock {\em Geometric aspects of diffusions on manifolds} (1988).

\item[\cite{MR2001f:58072}]K.~D. Elworthy, Y.~Le~Jan, and Xue-Mei Li. \newblock {\em On the geometry of diffusion operators and stochastic flows} (1999).

\item[\cite{MR90k:58244}]M. {\'E}mery. \newblock {\em Stochastic calculus in manifolds} (1989).

\item[\cite{MR2003c:58026}]Elton~P. Hsu. \newblock {\em Stochastic analysis on manifolds} (2002). 

\item[\cite{ikwat89}]N.~Ikeda and S.~Watanabe. \newblock {\em Stochastic differential equations and diffusion processes} (1989). 

\item[\cite{MR81d:60077}]Paul Malliavin. \newblock {\em G\'eom\'etrie diff\'erentielle stochastique} (1978). 

\item[\cite{MR658725}] P.-A. Meyer {\em G\'eometrie diff\'erentielle stochastique. II } (1982).

\item[\cite{rw2}]L.~C.~G. Rogers and D.~Williams. \newblock {\em {Diffusions, Markov Processes and Martingales, Vol. II: It{\^o}   Calculus}} (1987).

\item[\cite{MR83k:60064}]L. Schwartz. \newblock {\em G\'eom\'etrie diff\'erentielle du 2\`eme ordre, semi-martingales et   \'equations diff\'erentielles stochastiques sur une vari\'et\'e   diff\'erentielle. } (1982).

\item[\cite{MR86b:60085}]L. Schwartz. \newblock {\em Semimartingales and their stochastic calculus on manifolds} (1984).

\item[\cite{MR1715265}]D. ~W. Stroock. \newblock {\em An introduction to the analysis of paths on a {R}iemannian   manifold} (2000).

\endbib
\paragraph{Survey articles}
\begbib
\item[\cite{MR54:1404}]S.~A. Mol{\v{c}}anov. \newblock {\em Diffusion processes, and {R}iemannian geometry } ( 1975).

\item[\cite{MR84e:60084}]P.-A. Meyer. \newblock {\em A differential geometric formalism for the {I}t\^o calculus} (1981)
\endbib


\section{Infinite dimensional diffusions}

\begbib
\item[\cite{MR1465436}] G. Kallianpur and J. Xiong {\em Stochastic differential equations in infinite-dimensional spaces} (1995).
\endbib

\subsection{Filtering theory}

We refer to the textbook treatments of \cite[Ch. VI]{oks03} and \cite[Ch. VI]{rw2} Rogers-Williams and the monographs 
\begbib
\item[\cite{MR583435}] G. Kallianpur {\em Stochastic filtering theory} (1980).
\item[\cite{MR966884}] G. Kallianpur and R. L. Karandikar {\em White noise theory of prediction, filtering and smoothing} (1988)
\endbib

See also the survey papers by Kunita 
\cite{MR2884592} 
\cite{MR2884593}. 

\subsection{Measure-valued diffusions and Brownian superprocesses}

Some monographs are

\begbib
\item[\cite{MR1242575}]D.A. Dawson. \newblock {\em Measure-valued Markov processes} (1994). 
\item[\cite{MR2003c:60001}]E.~B. Dynkin. \newblock {\em Diffusions, superdiffusions and partial differential equations} (2002).
\item[\cite{MR1779100}] A. M. Etheridge {\em An introduction to superprocesses} (2000).
\item[\cite{legall99}]J.-F. Le~Gall. \newblock {\em Spatial branching processes, random snakes and partial   differential equations} (1999).  
\item[\cite{MR2031708}]B. Mselati. \newblock {\em Classification and probabilistic representation of the positive   solutions of a semilinear elliptic equation } (2004).
\endbib

These studies are related to the non-linear PDE $\Delta u = u^2$.
See also Dynkin \cite{MR2069167} who treats the equation
$\Delta u  = u^\alpha$ for $1 < \alpha \le 2$
and Le Gall \cite{legall00}.

\subsection{Malliavin calculus}
\label{sec:malliavin}

The term {\em Malliavin calculus}, formerly called {\em stochastic calculus of variations}, refers to the study of
infinite-dimensional, Gaussian probability spaces inspired by 
P. Malliavin's paper \cite{MR81f:60083}. 
The calculus is designed to study the probability densities of functionals of Gaussian
processes, hypoellipticity of partial differential operators, and the theory of non-adapted stochastic integrals.
The following account is quoted from Ocone's review of Nualart's text \cite{MR96k:60130} in {\em Math. Reviews}:

\begin{quote}
A partial differential operator $A$ is hypoelliptic if $u$ is $C^\infty$ on those open sets where $Au$ is $C^\infty$.  
H\"ormander's famous hypoellipticity theorem states that if $A=X_0+\sum_1^dX_i^2$, where
$X_0,\cdots,X_d$ are smooth vector fields, and if the Lie algebra
generated by $X_0,\cdots,X_d$ is full rank at all points, then $A$ is
hypoelliptic. Now, second-order operators such as $A$ appear as
infinitesimal generators of diffusion processes solving stochastic
differential equations driven by Brownian motion.  Thus the study of
$A$ can be linked to the theory of stochastic differential equations
(SDEs). In particular, hypoellipticity of the generator is
connected to the existence of smooth densities for the probability
laws of the solution.  If the Fokker-Planck operator is hypoelliptic,
the solutions of the Fokker-Planck equation are smooth functions
providing transition probability densities for the corresponding
diffusion. Conversely, it is possible to work back from the existence
of smooth densities to hypoellipticity.  Because solutions of SDEs
are functionals of the driving Brownian motion, the question of
hypoellipticity of the generator is then an aspect of a much more
general problem. Given an $\reals^n$-valued functional $G(W)$ of a Gaussian
process $W$, when does the probability distribution of $G(W)$ admit a
density with respect to Lebesgue measure and how regular is it?
Malliavin realized how to approach this question using a differential
calculus for Wiener functionals.  His original work contained two
major achievements: a general criterion for the existence and
regularity of probability densities for functionals of a Gaussian
process, and its application to solutions of SDEs, leading to a
fully stochastic proof of H\"ormander's theorem.

Malliavin's paper was tremendously influential, because it provided
stochastic analysts with a genuinely new tool.  Of the major lines of
investigation that ensued, let us mention the following, in only the
roughest manner, as background to Nualart's new text. First is the
continued study of existence and representation of densities of Wiener
functionals, its application to hypoellipticity, short-time
asymptotics, index theorems, etc. of solutions to second-order
operator equations, and its application to solutions of
infinite-dimensional stochastic evolution equations, such as
stochastic PDE, interacting particle systems, or delay
equations. The seminal work here is due to Stroock, Kusuoka, Watanabe,
and Bismut.  Second, Wiener space calculus has found application to
the quite different problem of analysis of non-adapted Brownian
functionals, following a paper of Nualart and Pardoux, who derived a
calculus for non-adapted stochastic integrals, using heavily the
Sobolev spaces defined in Wiener space analysis.  This made possible a
new study of stochastic evolution systems in which anticipation of the
driving noise occurs or in which there is no flow of information given
by a filtration.  Finally, we mention that the Malliavin calculus has
led to new progress and applications of the problem of
quasi-invariance of Wiener processes under translation by nonlinear
operators, a non-adapted version of the Girsanov problem.
\end{quote}

\paragraph{General references}
\begbib
\item[\cite{MR88m:60155}]D.~R. Bell. \newblock {\em The {M}alliavin calculus} (1987). 

\item[\cite{MR86f:58150}]J.-M. Bismut. \newblock {\em Large deviations and the {M}alliavin calculus} (1984). 

\item[\cite{ikwat89}]N.~Ikeda and S.~Watanabe. \newblock {\em Stochastic Differential Equations and Diffusion Processes} (1989). 

\item[\cite{MR99b:60073}]P. Malliavin. \newblock {\em Stochastic analysis} (1997). 

\item[\cite{MR89f:60058}]J. Norris. \newblock {\em Simplified {M}alliavin calculus} (1986). 

\item[\cite{MR96k:60130}]D. Nualart. \newblock {\em The {M}alliavin calculus and related topics} (1995). 

\item[\cite{MR84d:60092a}] and \item[\cite{MR84d:60092b}].
D.~W. Stroock. \newblock {\em The {M}alliavin calculus and its application to second order   parabolic differential equations, I and II} (1981).
%
See also \cite{MR83h:60076} and \cite{MR85c:60088}].

\cite{MR86b:60113} S. Watanabe {\em Lectures on stochastic differential equations and Malliavin calculus} (1984). 

\item[\cite{MR83e:60001}]D. Williams. \newblock {\em To begin at the beginning: {$\ldots $}} (1981).
\endbib

\paragraph{Applications to mathematical finance}


\begbib

\item [\cite{MR1968106}] {\em Conference on {A}pplications of {M}alliavin {C}alculus in {F}inance} (2003)


\item [\cite{MR2189710}] P. Malliavin and A. Thalmaier.  {\em Stochastic calculus of variations in mathematical finance} (2006).


\endbib

\section{Connections with analysis}

Kahane \cite{kahane97} provides a historical review of a century of interplay between Taylor series, Fourier series and Brownian motion.


\subsection{Partial differential equations}
\label{sec:pdes}
For general background on PDE, we refer to
L. C. Evans \cite{MR99e:35001}. 
The texts of Katazas and Shreve \cite[Chapter 4]{ks88},
Freidlin \cite{MR87g:60066} and 
Durrett [Ch. 8]\cite{MR87a:60054} all contain
material on connections between BM and PDE.
Other sources are 
Bass \cite{bass95p} 
and Doob \cite{MR731258}, 
especially for parabolic equations, and
Glover \cite{gl88}.

\subsubsection{Laplace's equation: harmonic functions}
We begin by considering Laplace's equation: \[\Delta u  = 0.\] where
$\Delta u \coloneqq \sum_{i=1}^n u_{x_i,x_i}$.  
A {\em harmonic function} is a  $C^2$ function $u$ which solves 
Laplace's equation.

Let $\ball(x,r)\coloneqq\{y: |y-x|<r\}$ and let $D$ be a {\em domain}, that is a connected open subset of
${\reals}^n$. Let $\tau_D \coloneqq \inf\{t:B_t\in D^c\}$. Since each
component of a Brownian motion $B$ in $\reals^n$ is a.s. unbounded, $P(\tau_D<\infty)=1$ for any bounded
domain $D$.  If $u$ is harmonic in $D$, then
\eq
u(x)=\fint_{\partial \ball(x,r)} u(y)S(dy)
\en
for every $x\in D$ and  $r>0$ such that $\overline{\ball(x,r)}\subset D$,
where $S$ is the uniform probability distribution on ${\partial \ball(x,r)}$ which is
invariant under orthogonal transformations.
To see this, observe that by It\^{o}'s formula,
\begin{equation*}
\begin{split}
u(B_{t\wedge
\tau_{\ball(x,r)}})&=u(x)+\sum_{i=1}^n\int_0^{t\wedge\tau_{\ball(x,r)}}u_{x_i}(B_s)dB_s+ \frac{1}{2} \int_0^{t\wedge\tau_{\ball(x,r)}}\Delta u(B_s)ds\\
    &=u(x)+\sum_{i=1}^n\int_0^{t\wedge\tau_{\ball(x,r)}} u_{x_i}(B_s)dB_s
\end{split}
\end{equation*}
which is a continuous local martingale.
But since $u(B_{t\wedge\tau_{\ball(x,r)}})-u(x)$ is uniformly bounded, it is a
true martingale with mean 0. Taking expectation and using
$P(\tau_{\ball(x,r)}<\infty)=1$ gives
\[
u(x) =  \ER^x\left( u(B_{\tau_{\ball(x,r)}}) \right) = \fint_{\partial \ball(x,r)}
u(y) dS.
\]
where the last equality is by the symmetry of Brownian motion with respect to orthogonal transformations.
Conversely, if $u:D\rightarrow \mathbb{R}$ has this mean value property, then $u$ is $C^\infty$ and harmonic.
(Gauss's theorem).

\subsubsection{The Dirichlet problem}
Some references are \cite[\S 4.2]{ks88},
\cite{bass95p}, 
\cite{MR87a:60054}, 
\cite{MR2604525}. 
Consider the equation
\begin{equation}
\Delta u =0 \mbox{, in } D \mbox{ and }u=f \mbox{ on }\partial D
\label{eq:pde}
\end{equation}
where $\ER^x\left( |f(B_{\tau_D})| \right) < \infty$.  
Then $u(x)\coloneqq\mathbb{E}_x\left(f(B_{\tau_D})\right)$ has
$\Delta u=0$ in $D$.
because
\begin{equation*}
\begin{split}
u(x)=\mathbb{E}_x f(B_{\tau_D}) &=\mathbb{E}_x\left(\mathbb{E}_x\left(
f(B_{\tau_D})|{\mathcal F}_{\tau_{\ball(x,r)}} \right)\right)\\
&=\mathbb{E}_x u(B_{\tau_{\ball(x,r)}})\quad\mbox{ by strong Markov property}\\
&=\fint_{\partial \ball(x,r)}u(y)dS
\end{split}
\end{equation*}
So in order to have a solution to the partial differential equation
(\ref{eq:pde}), we need:
\[\lim_{x\rightarrow a} \mathbb{E}_x \left( f(B_{\tau_D}) \right) =
f(a),\quad a\in\partial D.\]
This is true under a natural condition on the boundary. It should
be {\em regular}:
\[ \mbox{if } \sigma_D=\inf\{t>0: B_t\in D^c\}, \mbox{ then }
\mathbb{P}_x(\sigma_D=0)=1,\quad\forall x\in\partial D.\]
See \cite{ikwat89} 
for a more refined discussion, treating irregular boundaries.

\subsubsection{Parabolic equations}
Following is a list of PDE's of parabolic type related to BM$(\reals^\delta)$.
For $u: [0,\infty) \times \reals^\delta \te \reals$, write $u = u(t,x)$,
let $u_t:= \partial u /\partial t$, and let $\Delta$ denote the Laplacian, and $\nabla$ the gradient, acting on the variable $x$.
Assume the initial boundary condition $u(0,x) = f(x)$.
Then
\eq \lb{e1}
u_t = \hf \Delta u
\en
is solved by
$$
u(t,x) = \ER^x [ f(B_t ) ].
$$
\eq \lb{e2}
u_t = \hf \Delta u + g
\en
for $g : [0,\infty) \times \reals^\delta \te \reals$ is solved by
$$
u(t,x) = \ER^x \left[ f(B_t ) + \int_0^t g(t-s,B_s) ds \right].
$$
\eq \lb{e3}
u_t = \hf \Delta u + c u  
\en
where $c = c(x) \in \reals$
is solved by
\eq \lb{e3sol}
u(t,x) = \ER^x \left[ f(B_t ) \exp \left(  \int_0^t c(B_s) ds \right) \right].
\en
\eq \lb{e4}
u_t = \hf \Delta u + b \cdot \Delta u  
\en
for $b = b(x) \in \reals$
is solved by
\eq \lb{e4sol}
u(t,x) = \ER^x \left[ f(B_t ) \exp \left( \int_0^t b(B_s) \cdot d B_s - \hf \int_0^t | b(B_s)|^2 ds \right) \right].
\en

Equation \re{e1} is the classical \emph{heat equation}, discussed
already in Section 
\ref{infgen}.  
See \cite[\S 4.3]{ks88} for the one-dimensional case, and 
Doob \cite{MR731258} 
for a more extensive discussion.

Equation \re{e4} is the variant when Brownian motion $B$ is replaced by a BM with
drift $b$, which may be realized as the solution of the SDE
\eq \lb{sdeb}
X_t = x + B_t + \int_0^t b(X_s) ds .
\en
It is a nice result, due to 
Zvonkin \cite{zvonkin1974transformation},
for dimension $\delta = 1$ and
Veretennikov-Krylov \cite{veretennikov1976explicit}
for $\delta = 2,3, \ldots$, that for $b$ Borel
bounded, equation \re{sdeb} has a unique strong solution;
that the solution is unique in law is a consequence of Girsanov's theorem.
The right side of \re{e4sol} equals
$\ER^x [ f(X_t) ]$ for $(X_t)$ the solution of \re{sdeb}.

The expression on the right side of \re{e3sol} is the celebrated 
\emph{Feynman-Kac formula}. 
See \cite{keskac86} for a historical review,
and \cite{jpy96} 
for further discussion of the case
$\delta = 1$ using Brownian excursion theory to analyse
$$
\int_0^\infty dt e^{-\lambda t } \ER^x \left[ f(B_t ) \exp \left(  \int_0^t c(B_s) ds \right) \right].
$$
A well known application of the Feynman-Kac formula is Kac's derivation of
L\'evy's arcsine law for the distribution of $\int_0^t 1(B_s >0 ) ds$
for $B$ a BM$(\reals)$.
Kac also studied the law of $\int_0^t |B_s| ds$ by the same method.
See Biane-Yor \cite{MR886959} 
and \cite{perwel95} for related results.
Extensions of the Feynman-Kac formula to more general Markov processes
are discussed in \cite{rw1}
\cite{MR1963670} 
\cite{MR1670526}. 

\subsubsection{The Neumann problem}
Just as the distribution of $B$ stopped when it first hits $\partial D$ solves the Dirichlet boundary value problem,
for suitably regular domains $D$ the transition function $p(t,x,y)$ of $B$ reflected in $\partial D$ is the fundamental solution of the Neumann boundary value problem:
\eq
\left( \frac{ \partial} {\partial t} - \frac{1}{2}  \Delta_x \right) p(t,x,y) = 0 \qquad ( t >0, x, y \in D )
\en
with the Neumann boundary condition
\eq
\frac{ \partial} {\partial n_x} p(t,x,y) = 0 \qquad ( t >0, x \in D , y \in D)
\en
where $n_x$ is the inner normal at the point $x \in \partial D$, and
initial condition
\eq
\lim_{t \downarrow 0 } p(t,x,y) = \delta_y(x) \qquad (x,y \in D ).
\en
See Fukushima \cite{MR0231444}, 
Davies \cite{MR990239}, 
Ikeda-Watanabe \cite{ikwat89}, 
Hsu \cite{MR2633560} 
and Burdzy \cite{MR3184496}. 
See \cite{MR2396121} 
for treatment of related problems for irregular domains.

\subsubsection{Non-linear problems}
Le Gall's monograph
\cite{MR1714707} 
treats spatial branching processes and random snakes derived from Brownian motion, and their relation to 
non-linear partial differential equations such as $\Delta u = u^2$.

\subsection{Stochastic differential Equations}

In order to consider more general PDEs, we need to introduce the notion of
a {\em stochastic differential equations (SDEs)}. We say that the semimartingale
$X$ solves the SDE
\[dX_t=\sigma(X_t)dB_t+b(X_t)dt\]
if
\begin{equation}
X_t=X_0+\int_o^t \sigma(X_s)dB_s+\int_0^tb(X_s)ds \label{eq:sde2}
\end{equation}
Solutions to three equations exist in particular when $\sigma$ and $b$ are
bounded and Lipschitz. The proof is based on Picard's iteration.

\begin{claim}
If $X_t$ solves (\ref{eq:sde2}), then
\[M_t^f=f(X_t)-f(X_0)-\int_0^t Af(X)ds,\;t\geq0,\;f\in C^2\]
where
\[Af(x)=1/2\sum_{i,j=1}^na^{ij}(x)f_{x_i x_j}( x)+\sum_{i=1}^n
b^i(x)f_{x_i}(x)\]
and $a=\sigma\sigma^T$, is a martingale.
\end{claim}
\begin{proof}
\begin{equation*}
\begin{split}
<X^i,X^j>&=\sum_{k,l=1}^n <\sigma^{ik}(X).B^k,\sigma^{jl}(X).B^l>_t\\
    &=\sum_{k,l=1}^n \sigma^{ik}\sigma^{jl}. <B^k,B^l>_t\\
    &=\int_0^t a^{ij}<X_s>ds.
\end{split}
\end{equation*}

So by It\^{o}'s formula,
\begin{equation*}
\begin{split}
f(X_t)&=f(X_0)+\sum_{i=1}^n\int_0^tf_{X_i}(X)dX^i+1/2\sum_{i,j=1}^n\int_0^tf_{x_i x_j}(X_s)d<X^i,X^j>_s\\
&=f(X_0)+\sum_{i,j=1}^n\int_0^t \sigma^{ij}(X_s)f_{x_i}(X_s)dB^j_s  +
\int_0^tAf(X)ds
\end{split}
\end{equation*}

Now assume that $u\in C^2(D)\cap C(\overline{D})$ is a solution of
\begin{equation*}
-Au(X) =f(X) \mbox{ in } D, \mbox{ and }u=0 \mbox{ on }\partial D.
\end{equation*}
Then $u(x)=\mathbb{E}_x\left(\int_0^{\tau_D} f(X_s)ds\right)$. By It\^{o},
\begin{equation*}
\begin{split}
u(X_{t\wedge\tau_D})-u(x)&=M_{t\wedge\tau_D}^f+\int_0^{t\wedge\tau_D}Au(X_s)ds\\
    &=M_{t\wedge\tau_D}^f-\int_0^{t\wedge\tau_D} f(X_s)ds
\end{split}
\end{equation*}
Now taking expectation and limit as $t\rightarrow\infty$,
\[\mathbb{E}_x u(X_{\tau_D})-u(x)=-\mathbb{E}_x\int_0^{\tau_D}f(X_s)ds\] and
so
\[u(x)=\mathbb{E}_x\int_0^{\tau_D}f(X_s)ds.\]
\end{proof}

\subsubsection{Dynamic Equations}

We consider {\em dynamic equations} of the form
\begin{eqnarray*}
u_t&=&Au-cu\;\mbox{ in }(0,\infty)\times {\mathbb R}^n\\
u(0,x)&=&f(X)
\end{eqnarray*}

We show that the $C^2$ solutions of this equation are of the form
\[u(t,x)=\mathbb{E}_x\left( f(X_t)exp \int_0^tc(X_s)ds\right)\]
The first step is to show that
$u(t-s,X_s)exp\left(-\int_0^tc(X_s)ds\right)$ is a local martingale on
$[0,t)$.

If $c,u$ is bounded, then $M_s$ above is a bounded martingale. The martingale
convergence theorem implies that as $s\nearrow t$, $M_s\rightarrow M_t$.
Since $u$ is continuous and $u(0,x)=f(x)$, we must have
\[lim_{s\nearrow t} M_s=f(B_t)exp\left( - \int_0^tc(X_s)ds\right).\]

So we have
\[\mathbb{E}_x f(X_t)exp\left(-\int_0^tc(X_s)ds\right)=u(t,x)\]

We have seen that the solution to
\begin{eqnarray*}
-\Delta u&=&f\mbox{ in }D\\
u&=&0 \mbox{ on }\partial D
\end{eqnarray*}
is
\[u(x)=\mathbb{E}_x\int_0^{\tau_D}f(B_s)ds.\]
So if $f=1$, we have $\mathbb{E}_x(\tau_D)$ is the solution of
\begin{eqnarray*}
-\Delta u &=f\mbox{ in }D\\
u&=0\mbox{ on }\partial D
\end{eqnarray*}
For example, if $D=B(0,1)$, then the solution is $(1-|x|^2)/n$,
$\Rightarrow\; \mathbb{E}_x(\tau_{\ball(x,1)})=(1-|x|^2)/n$.

\subsection{Potential theory}

The probabilistic theory of Brownian motion is closely related to the classical potential theory of the Laplace operator and the parabolic
potential theory of the heat operator.  

\paragraph{Analytic treatments of potential theory}
\begbib
\item [\cite{MR0222317}]   O. D. Kellogg. {\em Foundations of potential theory} (1929).
\item [\cite{MR0261018}]  L. L. Helms. {\em Introduction to potential theory} (1969).
\item[\cite{MR1801253}] D. H.  Armitage and S. J. Gardiner {\em Classical potential theory} (2001).
\item[\cite{MR2907452}] N. A. Watson.  {\em Introduction to heat potential theory} (2012).
\endbib

\paragraph{Relations between Brownian motion and potential theory}

\begbib

\item[\cite{MR45:2798}]M. Kac. \newblock {\em Aspects probabilistes de la th\'eorie du potentiel} (1970). 

\item[\cite{MR83g:60096}]M.~Kac. \newblock {\em Integration in function spaces and some of its applications} (1980). 

\item[\cite{MR58:11459}]S.~C. Port and C.~J. Stone. \newblock {\em Brownian motion and classical potential theory} (1978). 

\item[\cite{MR55:13589}]M. Rao. \newblock {\em Brownian motion and classical potential theory} (1977). 

\item[\cite{MR30:2562}]F. Spitzer. \newblock {\em Electrostatic capacity, heat flow, and {B}rownian motion} (1964).

\item[\cite{MR92i:60002}]R. Durrett and H. Kesten, editors. \newblock {\em Random walks, {B}rownian motion, and interacting particle   systems} (1991).

\item[\cite{MR731258}] J.~L. Doob. \newblock {\em Classical potential theory and its probabilistic counterpart} (1984).
\endbib


%

\subsection{BM and harmonic functions}

Some references are:
%
%
%
%

\begbib

\item[\cite{bass95p}]R.~F. Bass. \newblock {\em Probabilistic techniques in analysis} (1995).  

\item[\cite{bass98deo}]R.~F. Bass. \newblock {\em Diffusions and elliptic operators} (1998). 

\item[\cite{MR57:14162}]D.~L. Burkholder. {\em Brownian motion and classical analysis} (1977). . \newblock In {\em Probability (Proc. Sympos. Pure Math., Vol. XXXI, Univ.   Illinois, Urbana, Ill., 1976)}, pages 5--14. Amer. Math. Soc., Providence,   R.I., 1977.   

\item[\cite{MR80j:30001}]B. Davis. \newblock {\em Brownian motion and analytic functions} (1979). 

\item[\cite{MR87a:60054}]R.  Durrett. \newblock {\em Brownian motion and martingales in analysis} (1984). 

\item[\cite{MR91c:60059}]R. ~F. Gundy. \newblock {\em Some topics in probability and analysis} (1989). 

\item[\cite{MR53:6768}]J.-P. Kahane. \newblock {\em Brownian motion and classical analysis } (1976). 

\item[\cite{MR58:31383}]K. ~E. Petersen. \newblock {\em Brownian motion, {H}ardy spaces and bounded mean oscillation} (1977). 

\endbib
\subsection{Hypercontractive inequalities}


Hypercontractive estimates show that some semigroup $(P_t)$ acting on a probability space is a contraction form $L^p$ to $L^q$ with $p<q$, and some $t>0$. 
Neveu \cite{MR0431400} 
used stochastic integration with respect to Brownian motion to establish a hypercontractive estimate due to Nelson for conditional expectations
of Gaussian variables.
Bakry and \'Emery 
\cite{MR889476} 
treat hypercontractive diffusions and their relation to the
     logarithmic Sobolev inequalities of Gross \cite{MR0420249}.
See also  Stroock \cite{MR545251}, 
Ba\~nuelos and Davis \cite{MR1206335}.  
Saloff-Coste \cite{MR1150597} relates {P}oincar\'e, {S}obolev, and {H}arnack inequalities  
in the setting of a second order partial differential operator on a manifold.
Ledoux  
\cite{MR1600888} 
\cite{MR2858471} 
offers overviews of the broader domain of isoperimetry and Gaussian analysis.
		
%

\section{Connections with number theory}


%
%
%

According to the central limit theorem of
Erd\"os--Kac 
\cite{MR0002374} 
in the theory of additive number theoretic functions, if $\omega(n)$ is the number of distinct prime factors of $n$, then for $n$ picked uniformly at random from the integers
from $1$ to $N$, as $N \to \infty$ the limit distribution of $(\omega(n) - \log \log n )/\sqrt{ \log \log n }$ is standard normal.
Billingsley \cite{MR0345144} 
showed how Brownian motion appears as the limit distribution of a random path created by a natural extension of this construction.
See also Philipp \cite{MR0354602} 
and Tennenbaum \cite{MR3363366}.  

The articles  of Williams \cite{MR1064572} 
and
Biane, Pitman and Yor \cite{MR1848256}
study various Brownian functionals whose probability densities
involve Jacobi's theta functions, and whose Mellin transforms involve the Riemann zeta function and other
zeta functions that arise in analytic number theory. In particular, the well known functional equation satisfied by the
Riemann zeta function involves moments of the common distribution of the range of a Brownian bridge and the
maximum of a Brownian excursion.
See also \cite{MR2966105} 
regarding other probability distributions related to the Riemann zeta function, and further references.


\section{Connections with enumerative combinatorics}
\label{sec:enum}

There are several contexts in which Brownian motion, or some conditioned fragment of Brownian motion
like Brownian bridge or Brownian excursion, arises in a natural way as the limit distribution of
a random path created from a random combinatorial object of size $n$ as $n \to \infty$.
In the first instance, as in Section \ref{sec:rw}, these limit processes are obtained from a path with uniform distribution on $2^n$ paths of length $n$
with increments of $\pm$, or on a suitable subset of such paths.
Less obviously, Brownian excursion also arises as the limit distribution of a path encoding the structure of any one of a number of natural combinatorial
models of random trees. Harris \cite{MR0052057} 
may have been the first to recognize the branching structure encoded in a random walk. This structure was 
exploited by Knight \cite{kt63} 
in his analysis of the local time process of Brownian motion defined by limits of occupation times of random walks.
Aldous \cite{MR1207226} 
developed the concept of the {\em Brownian continuum random tree}, now understood as a particular kind of random real tree  
\cite{MR2351587} 
that is naturally encoded in the path of a Brownian excursion. 
Aldous showed how Brownian excursion arises as the weak limit of a height profile process derived derived from a Galton-Watson tree conditioned to have $n$ vertices, as $n \te \infty$, 
for any offspring distribution with finite variance. This includes a number of natural combinatorial models for trees with $n$ vertices.
See \cite{csp} and \cite{MR2351587} for reviews of the work of Aldous and others on this topic.

For a uniformly distributed random mapping from an $n$ element set to itself, 
Aldous and Pitman \cite{MR1293075} 
showed how 
encoding trees in the digraph of the mapping as excursions of a random walk
leads to a functional limit theorem in which a reflecting Brownian bridge is obtained as the limit. This yields a large number of
limit theorems for attributes of the random mapping such as height of the tallest tree in the mapping digraph.
See \cite[\S 9]{csp} and \cite{MR2351587} for an overview and further references, and
\cite[\S  6.4]{csp} for a brief account of Aldous's theory of critical random graphs and the multiplicative coalescent.

Much richer limiting continuum structures related to the Brownian continuum tree have been derived in the last decade in the analysis
of limit distributions for random planar maps and related processes. Some recent articles on this topic are

\begbib
\item[\cite{MR3469151}] G. Miermont {\em Random maps and continuum random 2-dimensional geometries} (2013).
\item[\cite{MR3204495}] J.-F. Le Gall {\em The Brownian map: a universal limit for random planar maps} (2014).
\item[ \cite{MR3606744}] N. Curien and J.-F. Le Gall {\em Scaling limits for the peeling process on random maps} (2017).
\item [\cite{MR3758730}] J. Bertoin,  and N. Curien and I. Kortchemski.  {\em Random planar maps and growth-fragmentations} (2018).
\endbib

\subsection{Brownian motion on fractals}

Some articles are
\begbib
\item[\cite{MR1701339}] M. T. Barlow and R. F. Bass. {\em Brownian motion and harmonic analysis on {S}ierpinski carpets}  (1999).

\item[\cite{MR2639315}] M. T. Barlow,  R. F. Bass,  T. Kumagai and A. Teplyaev. {\em Uniqueness of {B}rownian motion on {S}ierpi\'nski carpets} (2010).

\item[\cite{MR1227032}] T. Kumagai {\em Estimates of transition densities for {B}rownian motion on nested fractals} (1993)

\item[\cite{MR1736196}] T. Kumagai. {\em Brownian motion penetrating fractals: an application of the trace theorem of {B}esov spaces} (2000).

\endbib

\subsection{Analysis of fractals}

Lalley \cite{MR1188049} relates Brownian motion in the plane to the equilibrium measure on the {J}ulia set of a rational mapping.
For analytic treatment of related results about harmonic measure for domains with complicated boundaries, including
Makarov's theorem on harmonic measure \cite{MR794117}, 
see Garnett and Marchall \cite{MR2150803}.

\subsection{Free probability}
Some references are:
\begbib
\item[\cite{MR1426833}] Ph. Biane {\em Free Brownian motion, free stochastic calculus and random matrices} (1997).
\item[\cite{MR3585560}] J. A. Mingo and R. Speicher {\em Free probability and random matrices} (2017).
\endbib

\section{Applications}

\subsection{Economics and finance}
Some historical sources are
Bachelier's thesis \cite{bach1900}, 
Black and Scholes \cite{MR3363443}. 
Some textbooks:

\begbib

\item[\cite{MR1768877}] J.-P. Fouque, G. Papanicolaou and K. R. Sircar. {\em Derivatives in financial markets with stochastic volatility}, (2000).

\item [\cite{MR2362458}] D. Lamberton and B. Lapeyre.  {\em Introduction to stochastic calculus applied to finance} (2008).

\item [\cite{MR2189710}] P. Malliavin and A. Thalmaier.  {\em Stochastic calculus of variations in mathematical finance} (2006).

\item[\cite{MR1783083}] J. M. Steele. {\em Stochastic calculus and financial applications} (2001).

\item[\cite{MR1718056}] G. Kallianpur and R. K. Karandikar. {\em Introduction to option pricing theory} (2000).

\endbib
See also Karatzas-Shreve \cite[\S 5.8]{ks88}.  

The theory of optimal stopping in continuous time is developed by {\O}ksendal \cite[Ch. X]{oks03}.
Two monographs on this topic are:

\begbib
\item[\cite{MR2374974}] A. N. Shiryaev {\em Optimal stopping rules} (1978).
\item [\cite{MR2256030}] G.  Peskir and A. Shiryaev {\em Optimal stopping and free-boundary problems} (2006).
\endbib

\subsection{Statistics}

The application of Brownian motion to the theory of empirical processes is treated in 
\cite{sw86} 
and \cite{MR3524796}. 
Siegmund \cite{MR799155} 
 treats Brownian motion approximations of the sequential probability ratio test.
This kind of application has motivated many studies of boundary crossing probabilities for Brownian motion, such as
\cite{MR861122}.  

Some texts on statistical inference for diffusion processes and fractional Brownian motions are
\cite{MR3289986} 
\cite{MR3015023} 
\cite{MR2778592} 
\cite{MR2144185} 
\cite{MR1717690}. 

\subsection{Physics}

Some basic references are Einstein's 1905 paper \cite{einstein1905motion},
and Nelson's book \cite{MR35:5001}. 
See also the translation \cite{langevin97} of Paul Langevin¿s 1908 paper \cite{langevin08} on Brownian motion.
Redner \cite{MR2003f:82076} provides a physical perspective on first-passage processes.
Hammond \cite{MR3733949} 
offers a recent review of Smoluchowski's theory of coagulation and diffusion.
Sme general texts on applications of stochastic processes in physical sciences are
\begbib
\item[\cite{MR648937}] N. G. van Kampen {\em Stochastic processes in physics and chemistry} (1981).
\item[\cite{gardiner2009stochastic}] C. Gardiner. {\em Stochastic methods: A handbook for the natural and social sciences} (2009).
\endbib


Brownian motion has played an important role in the development of various models  of physical processes involving
random environments and random media. The Brownian map mentioned in Section \ref{sec:enum} 
may be regarded as providing a random two-dimensional
geometry. Closely related studies are the theory of the Gaussian free field and Liouville quantum gravity.
Some monographs in this vein are
\begbib
\item[\cite{MR1717054}] A.-S. Sznitman. {\em Brownian motion, obstacles and random media} (1998).
\item[\cite{MR1890289}] E. Bolthausen and A.-S. Sznitman.  {\em Ten lectures on random media} (2002).
\item[\cite{MR2327824}] J.-P. Fouque, J. Garnier,  G. Papanicolaou, and K. S\o lna. {\em Wave propagation and time reversal in randomly layered media} (2007).
\item[\cite{MR2932978}] A.-S. Sznitman. {\em Topics in occupation times and Gaussian free fields} (2012).
\item[\cite{MR3616274}] L. Zambotti. {\em Random obstacle problems} (2017).
\endbib

Other themes related to diffusions in random environments are treated in:

\begbib
\item[\cite{MR3362353}] J. Engl\"ander {\em Spatial branching in random environments and with interaction} (2015).
\item[\cite{MR1082348}] P. R\'ev\'esz {\em Random walk in random and nonrandom environments} (1990).
\item[\cite{MR2952852}]  T. Komorowski, C. Landim and S. Olla {\em Fluctuations in Markov processes} (2012).
\item[\cite{MR2226845}] Z. Shi {\em Sinai's walk via stochastic calculus}, (2001).
\endbib


\subsection{Fluid mechanics}

Systems of interacting Brownian motions have been used as a computational tool
in fluid mechanics. The Navier-Stokes equations
and the Prandtl boundary layer equations of fluid mechanics
can be interepreted as Fokker-Planck (or Kolmogorov) equations for 
interacting particles which diffuse by Brownian motion and carry vorticity 
(=rotation). At high Reynolds numbers it is computationally more efficient to 
model and sample the Brownian motions than to solve the original equations.
To satisfy the boundary conditions particles have to be created at solid walls
by a branching construction.
This idea was introduced in A. Chorin  
in 
\cite{MR52:16280} 
and 
\cite{chorin78}
for the Prandtl equations.
A good theoretical treatment can be found in 
the paper of Goodman 
\cite{MR88d:35159}
See also the book of Cottet and Koumoutsakos
\cite{MR2001c:76093}
 for a fairly comprehensive account of the theory and practice 
 of these methods.

\subsection{Control of diffusions}

An introduction is provided by {\O}ksendal \cite[Ch. XI]{oks03}.
We thank Vivek Borkar for providing the following list of monographs on this topic:
\begbib
\item[\cite{MR80h:60081}] Gihman and Skorohod, {\em Controlled stochastic processes} (1979)

\item[\cite{MR56:13016}] W. H. Fleming and R. W. Rishel, {\em Deterministic and stochastic optimal control} (1975).

\item[\cite{MR94e:93004}] W. H. Fleming and H. M. Soner, {\em Controlled Markov processes and viscosity solutions} (1993)).

\item[\cite{MR90h:93115}] V. Borkar, {\em Optimal control of diffusion processes } (1989)).

\item[\cite{MR84k:93059}] A. Bensoussan, {\em Stochastic control by functional analysis methods}  (1982).

\item[\cite{MR93i:93001}] A. Bensoussan, {\em Stochastic control of partially observable systems} (1992). 

\item[\cite{MR82a:60062}] A. Krylov, {\em Controlled diffusion processes} (1980)
\endbib

More recent activity is in applications to finance. See for instance
\cite{MR2000h:91052}
\cite{MR3629171} 
\cite{MR2533355} 
\cite{MR2109687}. 
See \cite{kushner2013numerical} regarding numerical methods for stochastic control problems in continuous time.  

%
%
%
%
%
%
%
%
%
%
%
%
%
%
%
%
%
%
\def\cprime{$'$} \def\cprime{$'$}
  \def\polhk#1{\setbox0=\hbox{#1}{\ooalign{\hidewidth
  \lower1.5ex\hbox{`}\hidewidth\crcr\unhbox0}}} \def\cprime{$'$}
  \def\cprime{$'$} \def\cprime{$'$}
  \def\polhk#1{\setbox0=\hbox{#1}{\ooalign{\hidewidth
  \lower1.5ex\hbox{`}\hidewidth\crcr\unhbox0}}} \def\cprime{$'$}
  \def\cprime{$'$} \def\polhk#1{\setbox0=\hbox{#1}{\ooalign{\hidewidth
  \lower1.5ex\hbox{`}\hidewidth\crcr\unhbox0}}} \def\cprime{$'$}
  \def\cprime{$'$} \def\cydot{\leavevmode\raise.4ex\hbox{.}} \def\cprime{$'$}
  \def\cprime{$'$} \def\cprime{$'$} \def\cprime{$'$} \def\cprime{$'$}
  \def\cprime{$'$} \def\cprime{$'$} \def\cprime{$'$} \def\cprime{$'$}
  \def\cprime{$'$}


%
\end{document}